\documentclass{article}
\usepackage[utf8]{inputenc}
\usepackage{multicol}
\usepackage{threeparttable} 

\PassOptionsToPackage{numbers}{natbib}
\usepackage[preprint]{ssttyy}

\usepackage[utf8]{inputenc} 
\usepackage[T1]{fontenc}   
\usepackage{hyperref}    
\hypersetup{
    colorlinks=true,
    linkcolor=blue,
    citecolor=blue,
    urlcolor=blue
}
\usepackage{url}           
\usepackage{booktabs}       
\usepackage{amsfonts}       
\usepackage{nicefrac}       
\usepackage{microtype}      
\usepackage{xcolor}         
\usepackage{setspace}      
\usepackage{microtype}
\usepackage{graphicx}
\usepackage{subfigure}
\usepackage{placeins}
\usepackage{threeparttable}
\usepackage{tablefootnote}

\usepackage{amsmath}
\usepackage{amssymb}
\usepackage{mathtools}
\usepackage{amsthm}
\usepackage{enumitem}
\usepackage{wrapfig}
\usepackage[capitalize,noabbrev]{cleveref}
\usepackage{amstext} 

\usepackage{algorithm}
\usepackage{algorithmic}  
\theoremstyle{plain}
\newtheorem{theorem}{Theorem}[section]

\newtheorem{lemma}[theorem]{Lemma}

\theoremstyle{definition}

\newtheorem{assumption}{Assumption}

\newtheorem{remark}[theorem]{Remark}
\DeclareMathOperator*{\argmax}{\arg\!\max}
\DeclareMathOperator*{\argmin}{\arg\!\min}
\newcommand{\gap}{\mathrm{Gap}}
\newcommand{\lmo}{\mathrm{lmo}}

\newcommand{\alignop}{\mathrm{align}}

\title{ Boosted Stochastic Frank-Wolfe for Constrained Nonconvex Optimization}

\author{%
  Navil Nandhan
  \thanks{National University of Singapore (NUS), Singapore.\\
  \texttt{e0727209@u.nus.edu}} \\
  \And
  Abbas Khademi
  \thanks{School of Mathematics and Computer Science, Iran University of Science and Technology, Iran. \\
  \texttt{abbaskhademi92@gmail.com}} \\
  \And
  Antonio Silveti-Falls
  \thanks{CVN, CentraleSup\'elec, Universit\'e Paris-Saclay, Inria, France.\\
  \texttt{tonys.falls@gmail.com} } 
}
\date{\today}

\begin{document}
\newgeometry{margin=0.85in,footskip=0.35in} 
\maketitle

\begin{abstract}
  The boosted Frank-Wolfe algorithm accelerates the classical Frank-Wolfe algorithm by better aligning the update direction with the negative gradient. Its analysis, however, has been limited to deterministic convex problems, with step sizes that require either line search or knowledge of the Lipschitz constant of the gradient. We develop a novel step size strategy that does not depend on the Lipschitz constant of the gradient, which allows us to extend the boosted Frank-Wolfe algorithm to the stochastic setting. We prove that boosting with this step size strategy can be combined with many modern gradient estimators, including SAGA, L-SVRG, SAG, Heavy Ball momentum, and zeroth-order estimators, among others, while retaining the worst-case convergence rates of ordinary stochastic Frank-Wolfe. Our analysis also yields the first convergence rates for boosted Frank-Wolfe on nonconvex and quasar-convex objectives, results which are new even for deterministic problems. Experiments on sparse logistic regression and quantum process tomography show that stochastic boosted Frank-Wolfe achieves faster convergence per gradient oracle call (and on wall-clock) compared to the non-boosted baseline.
\end{abstract}

\begin{wrapfigure}[18]{r}{0.5\textwidth}
\vspace{-1.25cm}
  \begin{center}
    \includegraphics[width=0.375\textwidth]{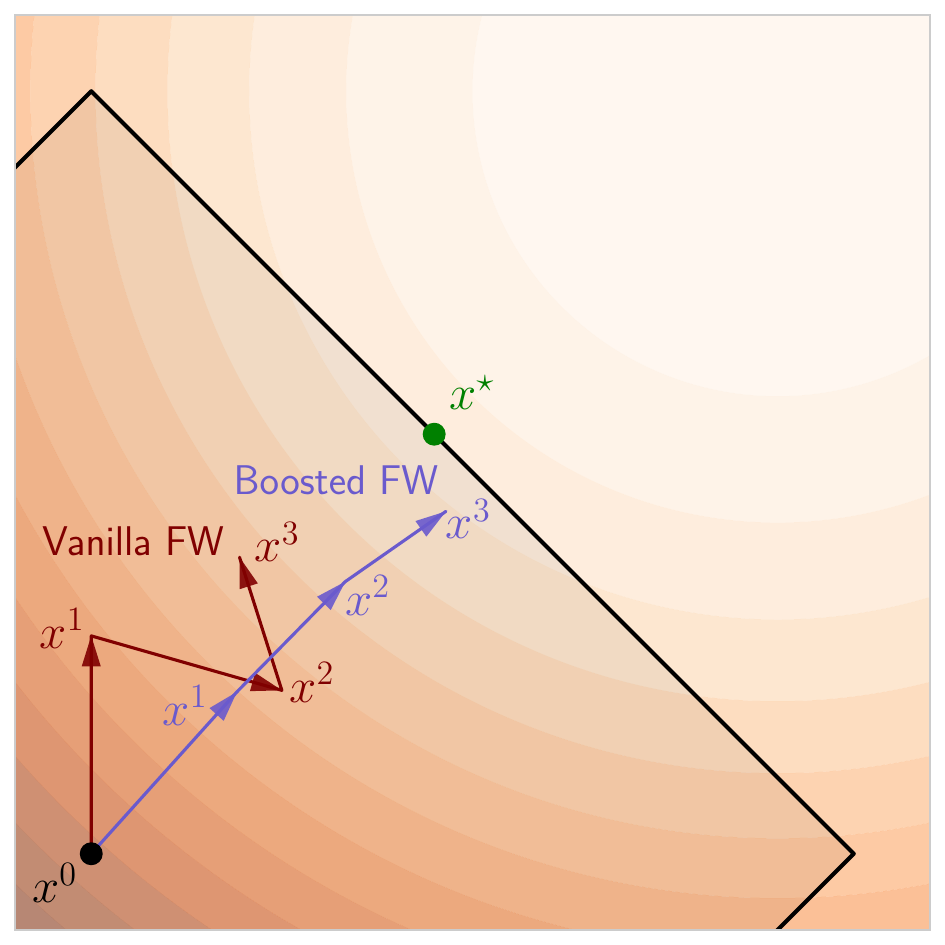}
  \end{center}
  \caption{Comparison of FW and BFW on a toy problem. The boundary of the constraint set is shown in black and $x^{\star}$ marks the minimizer in green. The FW trajectory in red shows zigzagging while the BFW trajectory in purple avoids this.}\label{fig:boosting_vs_zigzag}
\end{wrapfigure}


\section{Introduction}\label{sec: intro}
The Frank-Wolfe (FW) or Conditional Gradient algorithm is a method for solving constrained optimization problems that avoids projections onto the constraint set \cite{frank1956algorithm,levitin1966constrained}. Instead, it only requires access to the gradient and a linear minimization oracle (LMO) over the constraint set. This oracle can be much cheaper to compute than the projection in many problems of interest, e.g., low-rank inverse problems with nuclear norm regularization, traffic assignment, video co-localization and more (see \cite{combettes2021complexitylinearminimizationprojection} for a more exhaustive list of examples). With this potentially lower per-iteration cost and a worst-case convergence rate of $\mathcal{O}(1/t)$ for smooth convex problems, FW appears to be a competitive method. However, in practice, the method can suffer from a zigzagging phenomenon that slows convergence.

In practice, the LMO always returns an extreme point of the constraint set, e.g., a vertex if the set is polyhedral. If the optimum lies in between two extreme points, then the LMO outputs will alternate between these two points, which causes the zigzagging. Because zigzagging slows convergence by wasting updates on movement that is orthogonal to the optimal descent direction, many FW variations have been proposed which aim to extend beyond simply moving towards an extreme point, e.g., away-step FW changes the LMO to include moving away from extreme points. The boosted Frank-Wolfe (BFW) algorithm uses the LMO multiple times per-iteration to refine and align the update direction with the negative gradient, in order to avoid the zigzagging \cite{combettes2020boosting}. Figure~\ref{fig:boosting_vs_zigzag} shows how the zigzagging behavior is mitigated in BFW. 


While boosting is effective for avoiding zigzagging, the existing analysis of BFW in \cite{combettes2020boosting} assumes that the function to be minimized is convex, ruling out all nonconvex problems of interest. Moreover, the step sizes they suggest require either knowing the Lipschitz constant of the gradient or performing a line search, and are only guaranteed to converge for convex functions. As a result, the BFW algorithm has yet to be applied to nonconvex problems, nor has it reaped the benefits of the wide literature on inexact Frank-Wolfe, such as stochastic, zeroth-order, or distributed methods summarized in \cite{nazykov2024stochastic}.

We overcome these limitations by developing a new step size strategy for the BFW algorithm that requires neither knowledge of the Lipschitz constant of the gradient nor a line search. This enables the extension to the inexact gradient setting, allowing the algorithm to be applied to stochastic problems. We prove the convergence of the BFW algorithm with the proposed step size using several different gradient estimators, matching the convergence rate of BFW with known Lipschitz constant and exact gradients given in \cite{combettes2020boosting} and that of the stochastic FW algorithm \cite{nazykov2024stochastic} in all settings considered. Our analysis provides the first convergence rates for the BFW algorithm in the nonconvex and quasar-convex cases. Establishing these rates was previously unaddressed, even in the deterministic setting. We confirm the effectiveness of our step size in experiments on sparse logistic regression, where we observe improvements over vanilla Frank-Wolfe for every gradient estimator considered.

\paragraph{Contributions} Our contributions are three-fold:

\emph{Theory:} We demonstrate the first convergence guarantees for deterministic BFW without assuming convexity of the objective function. In the nonconvex setting we show that the so-called FW gap converges with a rate of $\mathcal{O}(1/\sqrt{t})$. We also show that the functional-value gap converges with a rate of $\mathcal{O}(1/t)$ under a relaxed assumption of quasar-convexity, which includes the convex setting and matches the rate given in \cite{combettes2020boosting}. We also introduce stochastic BFW (BSFW) and prove convergence in expectation with a rate of $\mathcal{O}(1/t)$ in the quasar-convex case and $\mathcal{O}(1/\sqrt{t})$ in the nonconvex case. Our step size is the first provably convergent strategy for BFW that does not require the Lipschitz constant of the gradient nor line search. In the quasar-convex (and hence also for convex) setting, we obtain the same worst-case convergence rate as the existing BFW method despite relaxing assumptions. 

\emph{Unification:} Through a general assumption on gradient inexactness, we simultaneously cover stochastic gradient estimators like Heavy Ball, SAGA, SARAH, L-SVRG; zeroth-order methods like JAGUAR; among others (eight in total), matching the results of \cite{nazykov2024stochastic} for stochastic FW.

\emph{Experimental Confirmation:} Our experimental results show a clear improvement over the vanilla FW version of each estimator we consider on sparse logistic regression in terms of number of stochastic gradients processed or number of coordinate gradients processed.

Finally, our treatment of quasar-convexity is primarily for completeness since there are several ML problems involving quasar-convexity, but it is not central to the novelty of our work. A discussion about this is given in Appendix~\ref{app subsec: quasar convex appns in ML}.

\paragraph{Related work}
Several works have proposed modifications of FW to avoid zigzagging, one of the earliest examples being the away-step FW \cite{guelat1986some,wolfe1970convergence}. More recently, such methods include the pairwise FW \cite{lacoste2015global}, the blended pairwise FW \cite{tsuji2022pairwise}, and the BFW \cite{combettes2020boosting}, which we directly extend to the nonconvex and stochastic settings.

For stochastic FW on nonconvex problems, many prior works exist covering convergence analysis, e.g., early work in \cite{reddi2016stochastic} followed by  \cite{nazykov2024stochastic} for several different variance-reduced estimators and \cite{pethick2025training} for the Heavy Ball momentum estimator. However, none of those analyses include boosting; the original analysis of BFW in \cite{combettes2020boosting} could not accommodate stochasticity and our work bridges this gap.

The original analysis of BFW in \cite{combettes2020boosting} relies on convexity, which we relax to quasar-convexity, or drop entirely in the nonconvex case. Quasar-convexity has been used in several recent works \cite{khademiaconvexity}, e.g., as a practical relaxation of convexity. It has also been used to analyze FW with improved rates compared to the general nonconvex case in \cite{khademi2025adaptive, martinez2025smooth,millan2025frank} and now our work. It was used in other first-order algorithms in \cite{ding2024optimizing,fu2023accelerated,lara2025delayed}, but its use in BFW has been unexplored until now.

\section{Preliminaries}\label{sec: prelims}
We consider the following optimization problem
\begin{equation}
    \tag{P}\min\limits_{x\in\mathcal{C}} f(x),
    \label{eq:P}
\end{equation}
where $\mathcal{C}\subset\mathbb{R}^n$ is a nonempty compact convex subset and $f$ is a continuously differentiable function whose gradient is Lipschitz-continuous on $\mathcal{C}$. We will assume that the projection onto the set $\mathcal{C}$ is unavailable in closed-form or otherwise computationally intractable. Instead, we will assume access to the LMO over the set $\mathcal{C}$, defined for all $v\in\mathbb{R}^n$ as
\begin{equation*}
    \lmo(v) = \argmin\limits_{s\in\mathcal{C}}\langle s,v\rangle.
\end{equation*}
For many common constraint sets $\mathcal{C}$ used in machine learning, the LMO is computable in closed form or otherwise cheaper than the corresponding projection operation; we refer to \cite{combettes2021complexitylinearminimizationprojection} for an in-depth study on this comparison.

For the analysis of first-order methods applied to unconstrained minimization of smooth nonconvex functions, the gradient norm $\|\nabla f(x)\|$ is typically employed as a surrogate measure of optimality. However, in constrained problems like \eqref{eq:P}, the gradient is not necessarily $0$ at a critical point. Instead, a critical point corresponds to
\begin{equation*}
    0 \in \nabla f(x) + \mathrm{N}_{\mathcal{C}}(x),
\end{equation*}
where $\mathrm{N}_{\mathcal{C}}(x)$ is the usual normal cone to the set $\mathcal{C}$ at $x$. We will therefore analyze convergence using the \emph{Frank–Wolfe gap} at a point $x \in \mathcal{C}$, which is defined as
\begin{equation}\label{eq:gap}
    \gap(x) := \max_{s \in \mathcal{C}} \langle \nabla f(x),\, x - s \rangle.
\end{equation}
This quantity is exactly the analog of $\|\nabla f(x)\|$ in the constrained setting, in the sense that it is nonnegative and certifies first-order optimality as
\begin{equation*}
    \gap(x) = 0 \iff 0 \in \nabla f(x) + \mathrm{N}_{\mathcal{C}}(x).
\end{equation*}
In deterministic problems, this quantity is easily computed at run-time since
\begin{equation*}
    \gap(x) =\left \langle \nabla f(x),x-\lmo(\nabla f(x)) \right\rangle,
\end{equation*}
which is the inner product of quantities already computed during FW or BFW.

\paragraph{Assumptions}
Throughout the paper, the $\|\cdot\|$ norm is defined as the standard Euclidean norm, i.e., $\|\cdot\| = \|\cdot\|_2$. We now formalize the smoothness assumption that we make on $f$ in problem \eqref{eq:P}.

\begin{assumption}[$L$-Smoothness]\label{assum: L_smooth}
    The gradient $\nabla f$ is Lipschitz-continuous on the set $\mathcal{C}$ with constant $L>0$, i.e.,
    \begin{equation*}
        \exists L > 0;~ \forall x, y \in \mathcal{C}: \quad \|\nabla f(x) - \nabla f(y) \| \leq L \|x - y\|.
    \end{equation*}
\end{assumption}

In addition to general nonconvex functions, we will also consider the class of \emph{quasar-convex} functions that satisfy the following assumption.


\begin{assumption}[Quasar-Convexity]
    \label{assum: quasar-convexity}
    The function $f$ is \emph{quasar-convex} with parameter $\rho \in {]0,1]}$, i.e., there exists $x^{\star} \in \argmin_{x \in \mathcal{C}} f(x)$ such that
    \begin{equation*}
        \forall y \in \mathcal{C}: \quad f^{\star} - f(y) \geq \tfrac{1}{\rho} \langle \nabla f(y),\, x^{\star} - y \rangle.
    \end{equation*}
\end{assumption}

Note that quasar-convexity recovers the well-known star-convexity when $\rho=1$. Furthermore, every convex function with a minimizer is quasar-convex with $\rho=1$, so our convergence analyses in this context will also encompass all convex functions.

\paragraph{Notation} 
Considering the problem \eqref{eq:P}, we define $f^{\star}=\min_{x \in \mathcal{C}} f(x)$, and the functional-value gap at iteration $t$ by $F_t = f(x^t) - f^{\star}$. For stochastic problems, we will denote $m^t$ as the stochastic estimator of the deterministic gradient $\nabla f(x^t)$ and $\Delta^t = m^t - \nabla f(x^t)$ as the difference between the stochastic estimator and the deterministic gradient. We denote the diameter of the set $\mathcal{C}$ by  $D = \max_{x,y \in \mathcal{C}} \|x - y\|$. The notation $\mathbb{E}$ is defined as the full expectation, while $\mathbb{E}_t[\cdot]$ is defined as the conditional expectation with respect to the randomness generated until iteration $t$, i.e.,  $\mathbb{E}_t[\cdot] = \mathbb{E}[\cdot\mid \sigma(x^0, x^1, \cdots, x^t) ]$ where $\sigma(x^0,x^1,\cdots, x^t)$ is the $\sigma$-algebra generated by the random variables $x^0,x^1,\cdots,x^t$. 

\begin{figure*}[!t]
    \centering
    \includegraphics[width=1\textwidth]{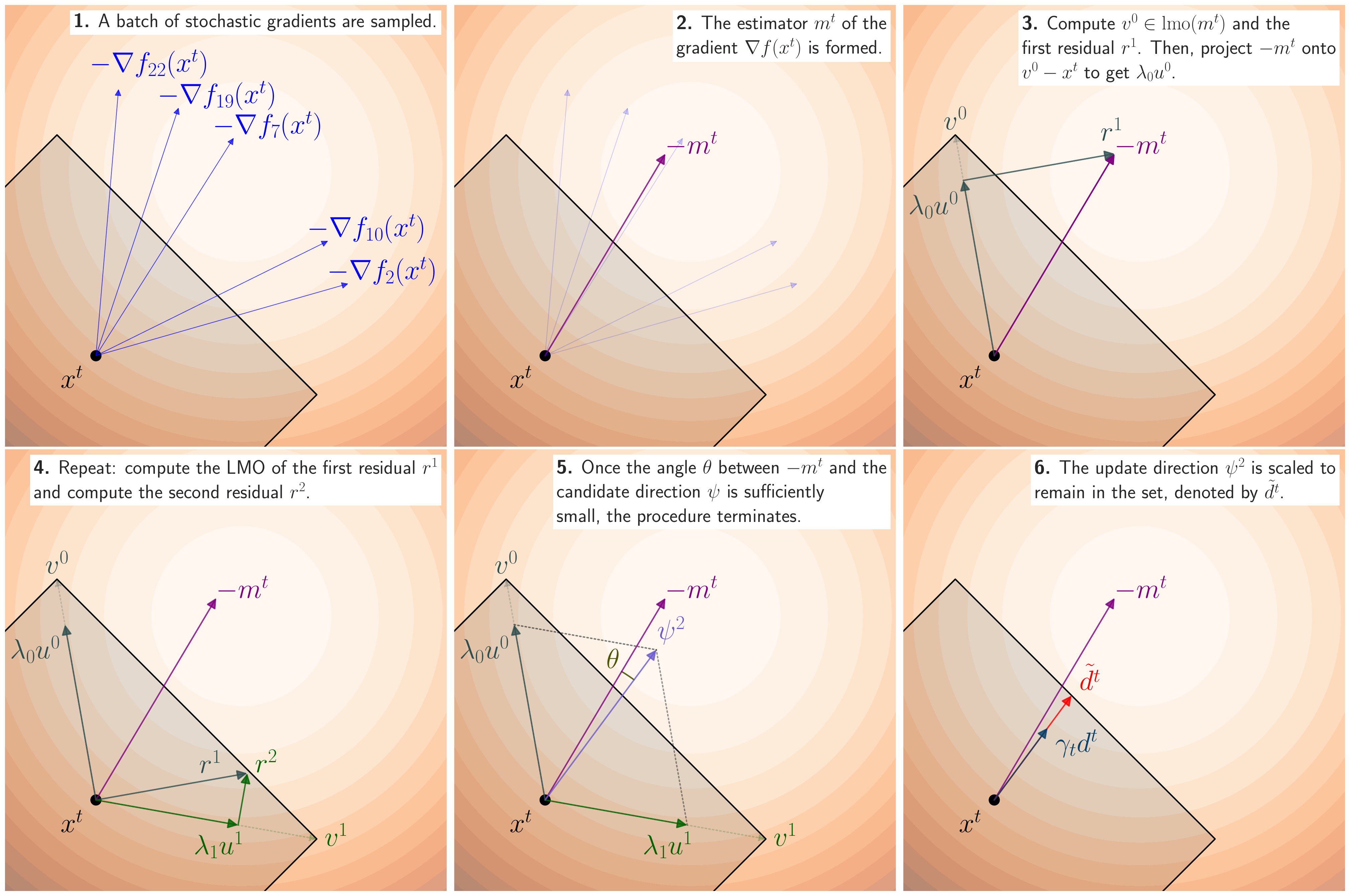}
\caption{Diagram of the stochastic boosting procedure used in Algorithm~\ref{alg:BSFW}. Panels 1 and 2: At iteration $t$, stochastic gradients are sampled and $-m^t$ is formed to estimate $-\nabla f(x^t)$ based on the estimator chosen, e.g., Heavy Ball as we show here. Panel 3: The LMO is calculated for $-m^t$ and the first residual $r^1$ is computed by projecting $-m^t$ onto $v^0-x^t$. Panel 4: The procedure then proceeds recursively, computing the LMO for the current residual $r^k$ and using that to compute the next residual $r^{k+1}$. Panel 5: A candidate direction $\psi^k$ is formed from the projections of the residuals and, if the alignment of $\psi^k$ with $-m^t$ has not improved enough, then the procedure terminates. Panel 6: The magnitude of the candidate direction $\psi^t$ is scaled $\tilde{d^t} = {\psi^2}/{(\lambda_0 + \lambda_1)}$ to remain in the set $\mathcal{C}$. Finally, this feasible direction is scaled by the step size and $x^{t+1} = x^t + \gamma_t d^t$ can be computed.}
    \label{fig:boosting_expln}
\end{figure*}

\section{Methods}\label{sec: methods}

\subsection{Boosting: Why, What, How, and When?}

To avoid zigzagging, the algorithm should update in the direction of $-\nabla f(x^t)$. However, such an update may not ensure feasibility and thus typically requires projection back onto $\mathcal{C}$, which we aim to avoid in this work. This motivates the FW algorithm, which avoids projection by constructing its update using a convex combination between the current iterate $x^t$ and an extreme point $s^t\in\lmo(\nabla f(x^t))$ of $\mathcal{C}$. The resulting direction $s^t-x^t$ may be nearly orthogonal or otherwise not aligned with $-\nabla f(x^t)$, for instance when the optimum lies between extreme points of $\mathcal{C}$ or in a face of the boundary $\partial \mathcal{C}$ as in Figure~\ref{fig:boosting_vs_zigzag}.

For many constrained problems that arise in machine learning, the LMO associated to the set $\mathcal{C}$ has low computational overhead relative to the gradient. In \cite{combettes2020boosting}, the authors leverage this by using the LMO several times per-iteration in a \emph{boosting} procedure to find a feasible direction that is better aligned with $-\nabla f(x^t)$ than $\lmo(\nabla f(x^t))$, the output used in FW.

The boosting procedure, corresponding to \textbf{Boost} in Algorithm~\ref{alg:BSFW} (expanded in Algorithm~\ref{alg:boosting}) and demonstrated in Figure~\ref{fig:boosting_expln} panels 3-5 (if we replace the stochastic estimator $-m^t$ by $-\nabla f(x^t)$), approximates the conical decomposition of $-\nabla f(x^t)$ in $\mathrm{cone}(\mathcal{C}-x^t)$, i.e., it approximately solves the problem
\begin{equation*}
    \argmin\limits_{d\in\mathrm{cone}(\mathcal{C}-x^t)}\tfrac{1}{2}\|-\nabla f(x^t)-d\|^2.
\end{equation*}
This decomposition of $-\nabla f(x^t)$ is constructed through multiple refinement steps, each of which starts by computing a residual $r^k = -\nabla f(x^t)-\psi^k$ (the part of $-\nabla f(x^t)$ which is not captured by the current approximate decomposition $\psi^k$) and then calling the LMO to find an extreme point $v^k\in\lmo(-r^k)$ most aligned with $r^k$. The candidate direction is then updated to $\psi^{k+1}$ by including $v^k-x^t$ in the decomposition. This recursive refinement continues until insufficient progress is made, as measured by the alignment between the candidate direction $\psi^k$ and $-\nabla f(x^t)$ (lines 20-26 in Algorithm~\ref{alg:boosting}). The alignment is measured through the modified cosine similarity, defined between two vectors $d$ and $\hat{d}$ as
\begin{equation*}
    \alignop(d, \hat{d}):= \begin{cases}\dfrac{\langle d, \hat{d}\rangle}{\|d\|\|\hat{d}\|}, & \hat{d} \neq 0 \\
-1. &  \hat{d}=0
\end{cases}
\end{equation*}

\vspace{-0.25cm}

This gives a decomposition of $-\nabla f(x^t)$ in terms of some set $\{v^0-x^t, v^1-x^t,\cdots,v^{K_t}-x^t\}\subset\mathrm{cone}(\mathcal{C}-x^t)$ with coefficients $\lambda_1,\cdots,\lambda_{K_t}$, e.g., Figure~\ref{fig:boosting_expln} Panel 5 shows $\psi^2$ as a conical combination $\lambda_0(v^0-x^t) + \lambda_1(v^1-x^t)$. We note that line 11 allows for the analog of an ``away-step'' in the construction of the conical decomposition and is necessary for the convergence of the procedure as noted in \cite{combettes2020boosting,locatello2017greedy}.

This approximate conical decomposition is what ensures that $x^t + \tilde{d}^t$ remains in $\mathcal{C}$, as it is used to compute the normalization in line 29 of Algorithm~\ref{alg:boosting} with $\Lambda^t$. Indeed, it is the construction of $\psi^k$ as a positive combination of extreme points in $\mathcal{C}$ that allows us to normalize $\psi^k$ into the feasible update $\tilde{d}^t$ which is then scaled by the step size $\gamma_t$ to get the final update.

Computing the exact conical decomposition of $-\nabla f(x^t)$ is not necessary, which is why the boosting procedure includes two hyperparameters: the maximum number of refinements, i.e., LMO calls, $K$ and an alignment improvement tolerance $\delta$, so that if doing another refinement of the direction does not improve the alignment by at least $\delta$, the procedure will terminate. Algorithm~\ref{alg:BSFW Full} in Appendix~\ref{subsec:complete BSFW Full} is the complete Boosted Stochastic FW algorithm. 

\renewcommand{\thealgorithm}{BSFW}
\begin{algorithm}[H]
\caption{Boosted Stochastic Frank-Wolfe (Full version stated in Appendix~\ref{subsec:complete BSFW Full})}
\label{alg:BSFW}
\textbf{Input:} initial estimator $m^{\mathrm{init}} \in \mathbb{R}^n$, gradient estimator $\{\Phi_t\}_{t=0}^{T-1}$, max no of rounds for boosting $K \in \mathbb{N} \setminus \{0\}$, alignment improvement tolerance $\delta \in ]0,1]$, and step decay $\{\eta_t\}_{t=0}^{T-1} \in ]0,1]$\\
\textbf{Output:} $x^T \in \mathcal{C}$.
\begin{algorithmic}[1]

\STATE $x^0 \gets \lmo(m^{\mathrm{init}})$
\FOR{$t = 0 \text{ to } T-1$}
    \STATE Sample $\xi_t\sim\mathcal{P}$ and compute $g^t = \nabla f(x^t,\xi_t)$
    \STATE Compute $m^t = \Phi_t(g^t)$
    \STATE $\tilde{d}^t \gets \textbf{Boost}(m^t)$\COMMENT{perform boosting procedure}
    \STATE $\gamma_t\gets \min\left\{\eta_t\tfrac{\|\lmo(m^t)-x^t\|}{\|\tilde{d}^t\|},1\right\}$ \textbf{if} $\tilde{d}^t \neq 0$ \textbf{else} $1$
    \IF{$\gamma_t < 1$}
        \STATE $d^t \gets \tilde{d}^t$ \COMMENT{use the boosted direction}
        \STATE $x^{t+1} \gets x^t + \gamma_t d^t$
    \ELSE
        \STATE $d^t \gets \lmo(m^t) - x^t$ \COMMENT{revert to vanilla FW direction}
        \STATE $x^{t+1} \gets x^t + \eta_td^t$ 
    \ENDIF
\ENDFOR
\end{algorithmic}
\end{algorithm}

\renewcommand{\thealgorithm}{Boost}
\begin{algorithm}[H]
\caption{(Boosting Procedure)}
\label{alg:boosting}
\textbf{Input:} estimator $m^t \in \mathbb{R}^n$, max no of rounds for boosting $K \in \mathbb{N} \setminus \{0\}$, and alignment improvement tolerance $\delta \in ]0,1]$\\
\textbf{Output:} boosting direction $\tilde{d}^t$
\begin{algorithmic}[1]

    \STATE $\psi^0 \gets 0$
    \STATE $\Lambda_t \gets 0$
    \STATE $k \gets 0$
    \WHILE{$k \leq K - 1$}
        \STATE $r^k \gets -m^t - \psi^k$ \COMMENT{$k$-th residual}
        \STATE $v^k \gets \lmo(-r^k)$ \COMMENT{FW oracle}
        \IF{$k = 0$}
            \STATE $s^t \gets v^k$
        \ENDIF
        \IF{$\psi^k\neq 0$}
            \STATE $u^k \gets \argmax_{u \in\left\{v^k - x^t, -\tfrac{\psi^k}{\|\psi^k\|}\right\}}\langle r^k,u\rangle$
        \ELSE
            \STATE $u^k\gets v^k - x^t$
        \ENDIF
        \IF{$u^k = 0$}
            \STATE $k\gets k+1$; \textbf{break} \COMMENT{exit $k$-loop}
        \ENDIF
        \STATE $\lambda_k \gets \tfrac{\langle r^k, u^k \rangle}{\|u^k\|^2}$
        \STATE $\phi^k \gets \psi^k + \lambda_k u^k$
        \IF{$\alignop(-m^t, \phi^k) - \alignop(-m^t, \psi^k) \geq \delta$}
            \STATE $\psi^{k+1} \gets \phi^k$
            \STATE $\Lambda_t \gets 
            \begin{cases} 
                \Lambda_t + \lambda_k, &  u^k = v^k - x^t \\
                \Lambda_t \left(1 - \tfrac{\lambda_k}{\|\psi^k\|}\right), &  u^k = -\tfrac{\psi^k}{ \|\psi^k\|}
            \end{cases}$
            \STATE $k\gets k+1$
        \ELSE
            \STATE $k\gets k+1$; \textbf{break} \COMMENT{exit $k$-loop}
        \ENDIF
    \ENDWHILE
    \STATE $K_t \gets k$
    \STATE $\tilde{d}^t \gets \dfrac{\psi^{K_t}}{\Lambda_t}$ \textbf{if} $\Lambda_t \neq 0$ \textbf{else} $0$ \COMMENT{normalize direction}

\end{algorithmic}
\end{algorithm}

\subsection{Step Size Strategy}
Prior work on boosting \cite{combettes2020boosting} has used either a line search $\gamma_t \in \argmin_{\gamma\in[0,1]}f(x^t + \gamma \tilde{d}^t)$ or $\gamma_t=\min\left\{\alignop(-\nabla f(x^t),\tilde{d}^t)\|\nabla f(x^t)\|/(L\|\tilde{d}^t\|), 1\right\}$, a strategy similar in spirit to the short-step in vanilla FW.

At every iteration $t$, we define our step size by
\begin{equation}\label{eq: step}
    \gamma_t = \min\left\{\eta_t \dfrac{\|s^t - x^t\|}{\|\tilde{d}^t\|}, 1 \right\},
\end{equation}
where $\eta_t$ is a step decay, $s^t \in \lmo(m^t)$, and $\tilde{d}^t$ is the aligned output from the boosting procedure (unless $\gamma_t=1$, in which case the update is not $\tilde{d}^t$ but rather $d^t = s^t-x^t$; this does not occur in practice as we demonstrate in Section~\ref{sec: experiments}). Lines 10-13 in Algorithm~\ref{alg:BSFW} describe the revert-to-FW step, which is the same as in \cite{combettes2020boosting}. The scaling by $\tfrac{\|s^t-x^t\|}{\|\tilde{d}^t\|}$ gives an update that is similar in magnitude to FW. This scaling, combined with appropriately chosen step decay $\{\eta_t\}$, guarantees convergence of Algorithm~\ref{alg:BSFW} similar to FW with open-loop step size strategies. Contrasting this with the strategies given in \cite{combettes2020boosting} for convex functions, our step size avoids line searches and knowledge of the Lipschitz constant $L$. A short note about the complexity of the step size is given in Appendix~\ref{app subsec: complexity of step size}.

\subsection{Gradient Estimators}
We are able to combine the boosting procedure with a slew of stochastic estimators that satisfy Assumption~\ref{assum: est}. Table~\ref{table: estimator_params} summarizes the estimators we consider and which we prove satisfy Assumption~\ref{assum: est} with Algorithm~\ref{alg:BSFW} in Appendix~\ref{appendix: estimators}. The main consideration in choosing estimators for Algorithm~\ref{alg:BSFW} is variance reduction rather than unbiasedness. This is more obvious when one considers that the LMO output $s^t\in\lmo(m^t)$ need not be an unbiased estimate for $\lmo(\nabla f(x^t))$ even if $m^t$ is unbiased for $\nabla f(x^t)$, since the LMO is typically discontinuous for the constraint sets $\mathcal{C}$ that are common in machine learning. Since the boosting procedure relies on the LMO heavily, it is thus also better adapted to estimators with variance reduction such as the ones we consider, e.g., SARAH, SAG, SAGA, and many others. It also clarifies why we do not require $C=0$ for all of our convergence results (although only one estimator, ZOJA, requires $C\neq 0$): there can be substantial bias as long as the variance of the estimator decreases appropriately.



\section{Analysis}\label{sec: analysis}
\label{sec:analysis}

We divide our results between the deterministic and stochastic settings. For the sake of presentation, we defer the proofs of the convergence analyses to the appendix. We emphasize that the rates we prove are \emph{non-asymptotic}, and that we use the big-$\mathcal{O}$ notation for convenience rather than necessity; when constants are omitted in the main text, we present them in the appendix. We will make a distinction between any-time convergence results, that do not require specifying the horizon (number of iterations) $T$ in advance, and horizon-dependent convergence results. We cannot describe all the stochastic estimators we use in the main body of the paper so we compile Table~\ref{table: estimator_params} in the Appendix, which describes them in full detail and shows that they satisfy Assumption~\ref{assum: est}.

\subsection{Deterministic Setting}
In this subsection, we will assume that $m^t=\nabla f(x^t)$, i.e., the deterministic gradient is computed exactly.
Our first result is a $\mathcal{O}(1/t)$ any-time convergence rate on the functional-value gap when $f$ satisfies both Assumption~\ref{assum: L_smooth} and Assumption~\ref{assum: quasar-convexity}.

\begin{theorem}[Convergence Rate for Deterministic Quasar-Convex Problems]
\label{thm:det_quasar}
Let $f$ be a function that satisfies Assumptions~\ref{assum: L_smooth} and~\ref{assum: quasar-convexity}. 
Consider the sequence $\{x^t\}_{t=0}^{+\infty}$ generated by Algorithm~\ref{alg:BSFW} with $m^t=\nabla f(x^t)$ and $\eta_t = \tfrac{2}{\rho(t+2)}$. 
Then, for all $t\geq 0$,  
\begin{equation*}
    F_t \leq \frac{1}{t+1} \max \left\{F_0, \dfrac{2 L D^2}{\rho^2}\right\} = \mathcal{O}\left(\frac{1}{t}\right).
\end{equation*}
\qed
\end{theorem}

\begin{remark}
    In the convex setting, $\rho=1$ and the step size is parameter agnostic: it does not depend on knowledge of any of the problem-specific constants like the Lipschitz constant $L$ of $\nabla f$.
\end{remark}

The next result guarantees a $\mathcal{O}(1/\sqrt{t})$ any-time convergence rate on the Frank-Wolfe gap defined in \eqref{eq:gap} by adjusting the step decay $\{\eta_t\}$ compared to the quasar-convex case.

\begin{theorem}[Convergence for Deterministic Nonconvex Problems]
\label{thm:det_nonconvex}
Let $f$ be a function that satisfies Assumption~\ref{assum: L_smooth}. 
Consider the sequence $\{x^t\}_{t=0}^{+\infty}$ generated by Algorithm~\ref{alg:BSFW} with $m^t=\nabla f(x^t)$ and $\eta_t = \tfrac{1}{\sqrt{t+1}}$. 
Then, for all $t\geq 0$,
\begin{equation*}
    \min_{0 \leq i \leq t} \gap(x^i) 
\leq \dfrac{F_0 + L D^2}{\sqrt{t+1}} = \mathcal{O}\left(\frac{1}{\sqrt{t}}\right).
\end{equation*}
\qed
\end{theorem}
\begin{remark}
    The step size used in the convergence result above is parameter agnostic in addition to being an any-time guarantee. A horizon-dependent convergence rate for Algorithm~\ref{alg:BSFW} in the deterministic setting with a step size that depends on the horizon $T$ is given in Theorem~\ref{thm:det_ncv_fixed_horizon} in Appendix~\ref{app:proofs}, with an improved constant but the same order of convergence.
\end{remark}

\subsection{Stochastic Setting}
For the convergence analysis in the stochastic setting, we will make a general assumption on the second moment of the error of the stochastic estimator of the gradient, summarized in Assumption~\ref{assum: est} and inspired by Assumption 2.1 of \cite{nazykov2024stochastic}. Under this assumption, we can give convergence results for BSFW applied to both nonconvex and quasar-convex problems. Since many stochastic estimators can be shown to satisfy this assumption when used with BSFW, the resulting analysis is unified.

\begin{assumption}[Estimator Assumptions \cite{nazykov2024stochastic}]
\label{assum: est}
Let $\{x^t\}_{t=0}^T$ denote the sequence of iterates generated by Algorithm~\ref{alg:BSFW}. There exist constants $A, B, C, E \ge 0$, parameters $\rho_1, \rho_2 \in ]0, 1]$, and a (possibly random) sequence $\{\sigma_t\}_{t \ge 0}$ such that the following conditions hold $\forall t \ge 1$:
\begin{equation*}
\begin{aligned}
    \mathbb{E}_{t-1}\left[\|\Delta^t \|^2\right]
&\le (1 - \rho_1)\|\Delta^{t-1}\|^2 
+ A \sigma_{t-1}^2 + \eta_{t-1}^2 B D^2 + C,\\
    \mathbb{E}_{t-1}\!\left[\sigma_t^2\right] 
&\le (1 - \rho_2)\sigma_{t-1}^2 + \eta_{t-1}^2 E D^2.
\end{aligned}
\end{equation*}
\end{assumption}
In the appendix, we show that a slew of stochastic estimators satisfy this assumption when used with BSFW; this is then summarized in Table~\ref{table: estimator_params}.

As in the deterministic setting, we start with quasar-convex functions with an appropriately chosen step decay $\{\eta_t\}$. However, in this setting we will use horizon-dependent step decays and any-time versions.

\begin{theorem}[Convergence for Stochastic Quasar-Convex Problems]\label{thm:stoch_fixed_horizon_quasar_convex}
Let $f$ be a function that satisfies Assumptions~\ref{assum: L_smooth} and~\ref{assum: quasar-convexity}. Suppose the stochastic gradient estimator $m^t$ and auxiliary sequence $\{\sigma_t\}$ satisfy Assumption~\ref{assum: est} with parameters $\rho_1,\rho_2\in]0,1]$ and constants $A,B,C,E\geq 0$. Consider the sequence $\{x^t\}_{t=0}^T$ generated by Algorithm~\ref{alg:BSFW} by fixing $T\in\mathbb{N}$ and using the constant, horizon-dependent step decay 
\begin{equation*}
    \eta_t = \begin{cases}
        \tfrac{1}{\rho d}, &  T\leq d\\
        \tfrac{1}{\rho d}, &  T>d\text{ and }t\leq t_0\\
        \tfrac{2}{\rho(2d+t-t_0)}, & T>d\text{ and }t\geq t_0
    \end{cases}
\end{equation*}
with $d:=\tfrac{2}{\min\{\rho_1,\rho_2\}}$ and $t_0:=\lfloor T/2 \rfloor$. Then,
\begin{equation*}
    \mathbb{E}[F_T]=\mathcal{O}\left(\frac{1}{T} + \sqrt{C}\right).
\end{equation*}
If the estimator satisfies Assumption~\ref{assum: est} with $C=0$, then the last term vanishes and we obtain a $\mathcal{O}(1/T)$ rate.
\qed
\end{theorem} 

\begin{theorem}[Any-Time Convergence for Stochastic Quasar-Convex Problems]\label{thm:stoch_quasar_convex_param_ag}
Let $f$ be a function that satisfies $L$-smoothness Assumption~\ref{assum: L_smooth} and $\rho$-quasar-convexity Assumption~\ref{assum: quasar-convexity}. Suppose the stochastic estimator $m^t$ and auxiliary sequence $\{\sigma_t\}$ satisfy Assumption~\ref{assum: est} with parameters $\rho_1, \rho_2 \in ]0,1]$ and constants $A, B, C, E \geq 0$. Let $\{x^t\}_{t=0}^{+\infty}$ be a sequence generated by Algorithm~\ref{alg:BSFW} by choosing the step decay
\begin{equation*}
\eta_t = \frac{2}{\rho(t + \nu)}, \quad \text{where }\quad  \nu = \max\left\{2, \frac{4}{\min\{\rho_1, \rho_2\}}\right\}.
\end{equation*}
Then, the expected functional-value gap satisfies 
\begin{equation*}
\mathbb{E}[F_t] \leq\sqrt{\frac{16 D^2}{\rho^2(t + \nu)} \left( \frac{32 D^2 B}{\rho^2(t + \nu)\rho_1} + \frac{64 D^2 AE}{\rho^2(t + \nu)\rho_1 \rho_2} + \frac{2CT}{\rho_1} \right)} +\frac{4\nu^2 \mathbb{E}[r_0]}{(t + \nu)^2} + \frac{8 D^2 L}{\rho^2(t + \nu)}.
\end{equation*}
where $r_t$ is a Lyapunov function defined by
\begin{equation*}
\forall t: \quad r_t := F_t + \frac{2\alpha^{\star}}{\rho_1 L} \|\Delta^t\|^2 + \frac{4\alpha^{\star} A}{\rho_1\rho_2 L} \sigma_t^2,
\end{equation*}
with 
\begin{equation*}
    \alpha^{\star} = \sqrt{\left( \frac{16 D^2 L}{\rho^2(T + \nu)} \right) / \left( \frac{32 D^2 B}{\rho^2(T + \nu) \rho_1 L} + \frac{64 D^2 AE}{\rho^2(T + \nu) \rho_1 \rho_2 L} + \frac{2CT}{\rho_1 L} \right)}.
\end{equation*}
If $C = 0$, the last term in the square root vanishes and we obtain a $\mathcal{O}(1/t)$ rate.
\qed
\end{theorem}

For nonconvex functions, we have the following convergence rate for the Frank-Wolfe gap in expectation; an any-time step decay and convergence rate can be found in Appendix~\ref{app:proofs} with an additional $\ln(t)$ factor.

\begin{theorem}[Convergence for Stochastic Nonconvex Problems]\label{thm:stoch_fixed_horizon_nonconvex}
Let $f$ be a function that satisfies Assumption~\ref{assum: L_smooth}. Suppose the stochastic gradient estimator $m^t$ and auxiliary sequence $\{\sigma_t\}$ satisfy Assumption~\ref{assum: est} with parameters $\rho_1, \rho_2 \in ]0,1]$ and constants $A, B, C, E \ge 0$. 
Consider the sequence $\{x^t\}_{t=0}^{T}$ generated by Algorithm~\ref{alg:BSFW} by fixing a $T \in \mathbb{N}$ and using the constant, horizon-dependent step decay $\eta_t = \tfrac{1}{\sqrt{T}}$. Then, 
\begin{equation*}
    \mathbb{E}\left[\min_{0 \le t \le T-1}  \operatorname{Gap}(x^t)\right]= \mathcal{O}\left(\frac{1}{\sqrt{T}} + \sqrt{C}\right).
\end{equation*}
If the estimator satisfies Assumption~\ref{assum: est} with $C=0$, then the last term vanishes and we obtain a $\mathcal{O}(1/\sqrt{T})$ rate.
\qed
\end{theorem}

\begin{remark}
    The Heavy Ball estimator satisfies a slightly more general version of Assumption~\ref{assum: est} which allows $\rho_1, \rho_2,$ and $B$ to depend on the horizon $T$. We prove convergence for this estimator separately in Appendix~\ref{subsec: Heavy Ball}, with a worse convergence rate of $\mathcal{O}(1/T^{1/4})$ for nonconvex problems that matches the convergence rate for this estimator when used with vanilla SFW \cite{pethick2025training}.
\end{remark}

\section{Experiments}\label{sec: experiments}
We study Algorithm~\ref{alg:BSFW} with different estimators that we have analyzed, with SFW using those same estimators as a baseline. The boosting procedure was built upon the code provided in \cite{combettes2020boosting}. We consider sparse logistic regression, which is convex, and quantum process tomography with a nonconvex objective function. We perform more experiments, for instance collaborative filtering using a nuclear norm constraint, measuring the expected value of the optimality gap over several runs, measuring the number of oracle calls as a function of the tolerance. In all plots, the darker colors represent Algorithm~\ref{alg:BSFW} with similar hues corresponding to the same estimator.

\subsection{Sparse Logistic Regression}\label{sec:sparse_logistic_regression}
We consider the following problem
\begin{equation}\label{exp:sparsity_problem} 
\min\limits_{x \in \mathcal{C}}~ \frac{1}{m} \sum_{i=1}^{m} \ln\left(1 + \exp\left(-y_i a_i^\top x\right)\right),
\end{equation}
where $\mathcal{C} = \left\{ x \in \mathbb{R}^n:~ \|x\|_1 \leq \tau \right\}$ for some radius $\tau > 0$ that we pick according to each dataset. We denote $\{a_i, y_i\}_{i=1}^m$ as samples drawn from an experiment-specific dataset, where $\forall i, a_i \in \mathbb{R}^n$ and $\forall i, y_i \in \{-1, 1\}$. The LMO of the $\ell_1$ ball with radius $\tau$ is given by \eqref{eq: lmo_l1}. The LMO of the $\ell_1$ ball has a $\mathcal{O}(n)$ complexity \cite{combettes2021complexitylinearminimizationprojection}, hence making it an appropriate constraint set for boosting, as discussed further in Appendix~\ref{app subsec: when to boost};
\begin{align}
    \label{eq: lmo_l1}
    \forall g \in \mathbb{R}^n: \quad v^{\ast} = \lmo(g) \iff  v^{\ast} = -\tau \operatorname{sign}(g_i)\, e_i 
~~\text{with}~~
i = \argmax_{j} |g_j|.
\end{align}

We use the {\ttfamily rcv1} train dataset \cite{lewis2004rcv1}, {\ttfamily mushrooms} dataset \cite{mushroom_73}, and the {\ttfamily breast cancer} \cite{breast_cancer_wisconsin} datasets from the LIBSVM library of datasets \cite{chang2011libsvm}. The parameters of the experiments used are given in Table~\ref{table: exp_params}. 
\begin{figure*}[hbpt!]
    \centering
    \includegraphics[width=0.3\textwidth]{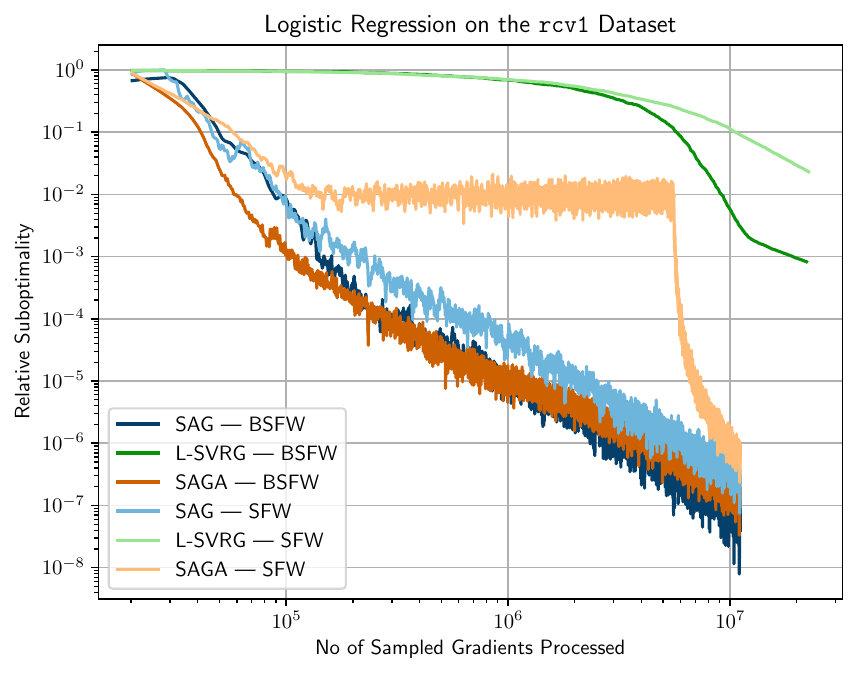}
    \includegraphics[width=0.3\textwidth]{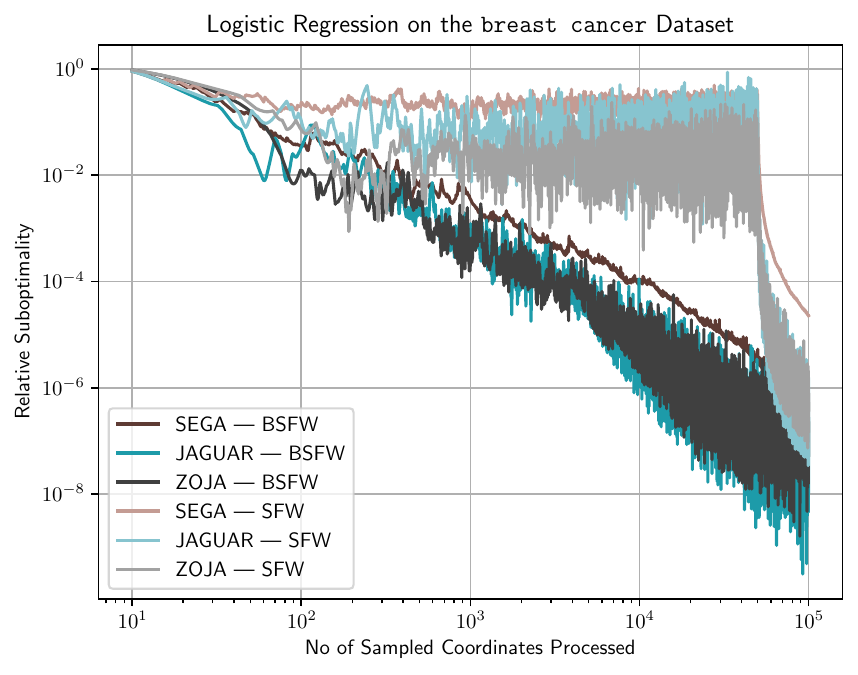}
    \includegraphics[width=0.3\textwidth]{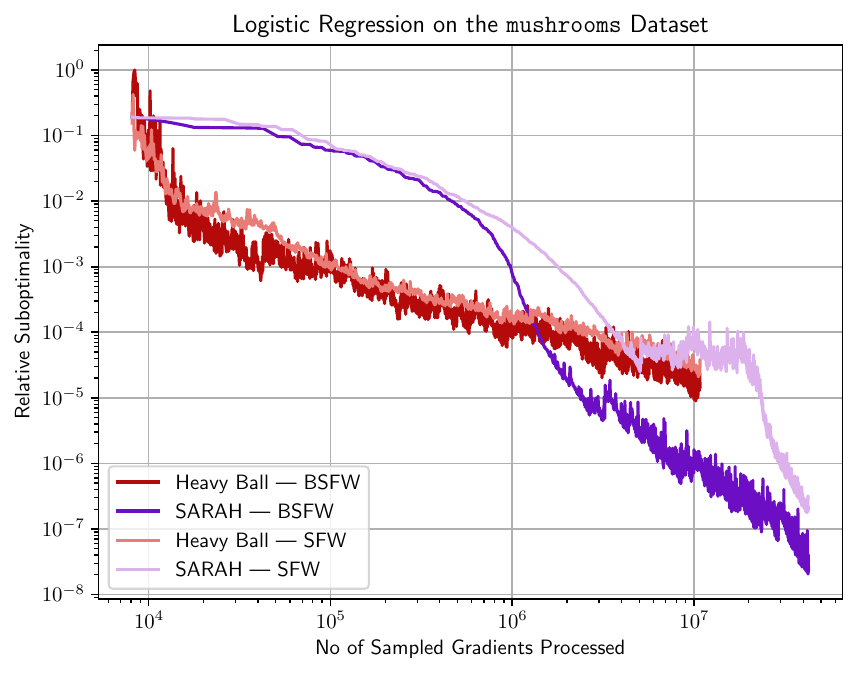}
    \caption{Suboptimality in the Logistic Regression problem is measured by $    \tfrac{f(x^t) - f_{\min}}{f_{\max} - f_{\min}},$ with $f_{\max},f_{\min}$ estimated as the $\max$ and $\min$ across all runs.}\label{fig: convergence_plots}
\end{figure*}
We run Algorithm~\ref{alg:BSFW} and SFW with every estimator listed in Table~\ref{table: estimator_params}. An explanation of stochastic and coordinate methods is given by Appendices~\ref{appendix: abt stoch} and \ref{appendix: abt coord}. For BSFW, we use the step decay given by Theorem~\ref{thm:stoch_quasar_convex_param_ag}, while for SFW, we use the step size provided by the corresponding prior work \cite{JMLR:v21:18-764,nazykov2024stochastic,pmlr-v119-negiar20a}. For all of the experiments, we pass in an alignment tolerance $\delta = 10^{-4}$, and a large max number of oracles $K = 10,000$ for the boosting procedure (effectively uncapped). The performance results of the different algorithms with the different estimators are shown in Figure~\ref{fig: convergence_plots}. Although a batch size $b_s$ is passed in as a parameter, the actual number of gradients sampled and computed per-iteration $t$ is counted for precision, since it can differ between estimators.



\begin{table}[hpbt!] 
\centering
\caption{Summary of dataset parameters.}
\label{table: exp_params}
\begin{tabular}{cccccc}
\toprule
\textbf{Name} & \textbf{$n$} & \textbf{$m$} & \textbf{$b_s$} & \textbf{$b_c$} & \textbf{$\tau$} \\ \midrule
{\ttfamily rcv1} & 47,236 &  20,242  & 742 & -- & 100 \\ 
{\ttfamily mushrooms}  & 112 &  8,124 & 404 & -- & 50\\ 
{\ttfamily breast cancer}  & 10 & 683 & -- & 1 & 5\\ 
\bottomrule
\end{tabular}
\begin{tablenotes}
{\footnotesize\item 
For each of the datasets, $n$ refers to the dimension of the features, $m$ refers to the total number of samples, $b_s$ refers to the batch size sampled, while $b_c$ refers to the coordinate batch size (number of coordinates) sampled.}
\end{tablenotes}
\end{table}


Apart from relative suboptimality, one might be interested in the progress in loss $\left(\mathbb{E}[f(x^t) - f^{\star}]\right)$ made per-iteration $t$. Figure~\ref{fig: exp_loss} shows this. To approximate the true expectation by a numerical expectation, we run every experiment $10$ times each. On a per-iteration basis, BSFW has a clear advantage over SFW for all estimators considered.

\begin{figure}[hbpt]
    \centering
    \includegraphics[width=0.32\textwidth]{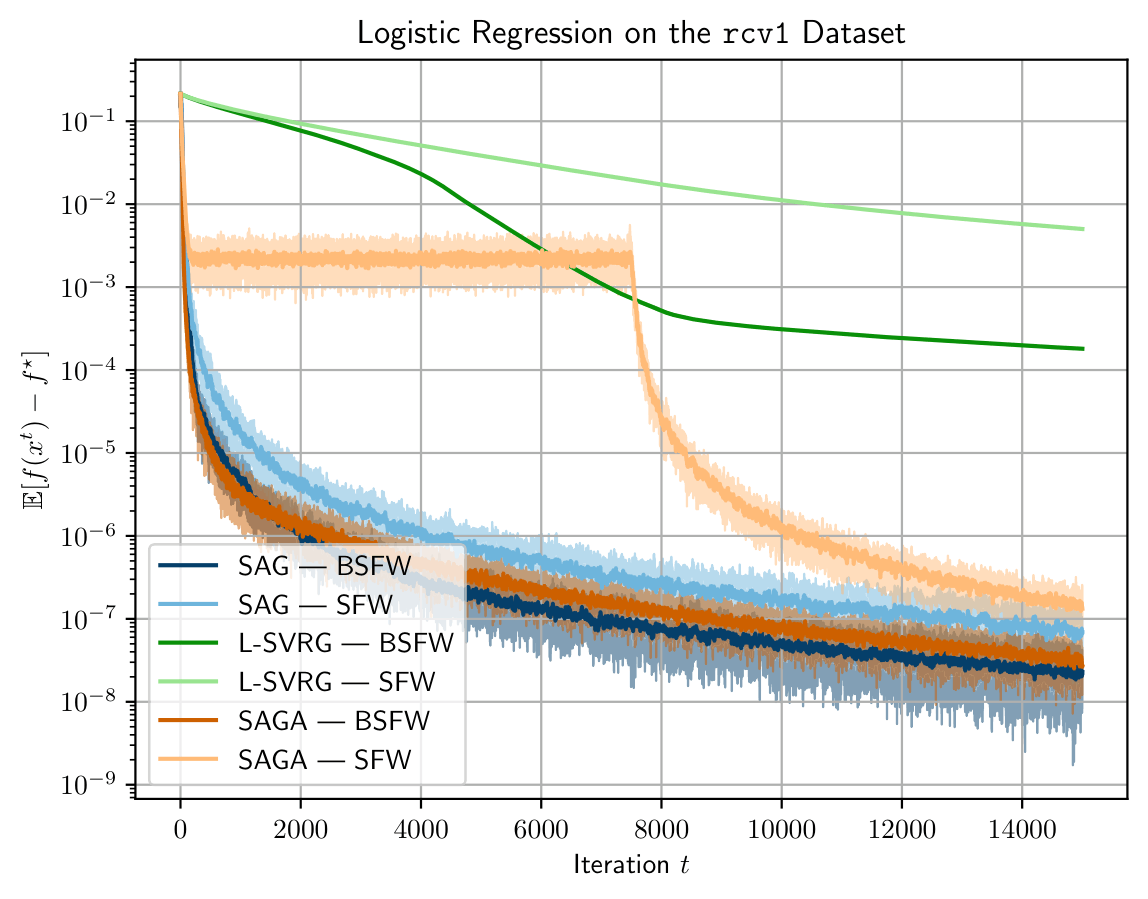}
    \includegraphics[width=0.32\textwidth]{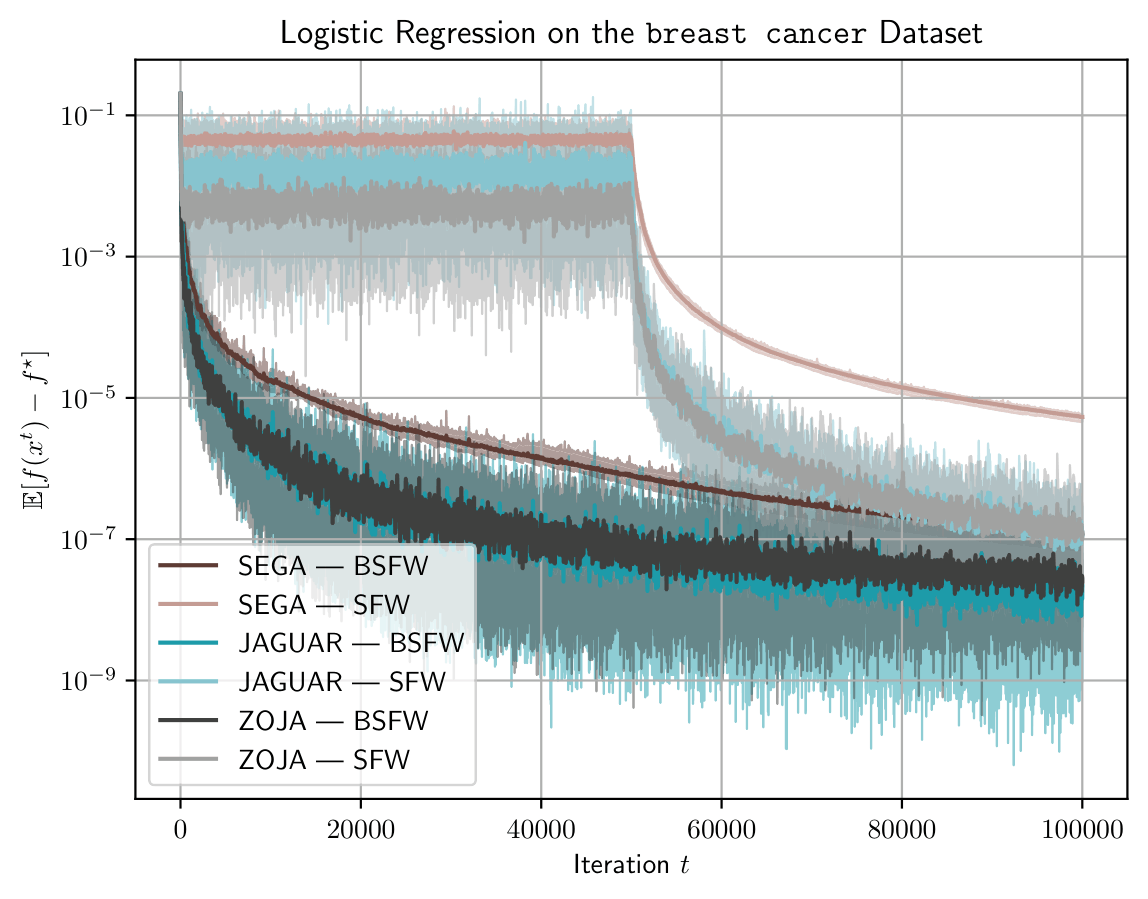}
    \includegraphics[width=0.32\textwidth]{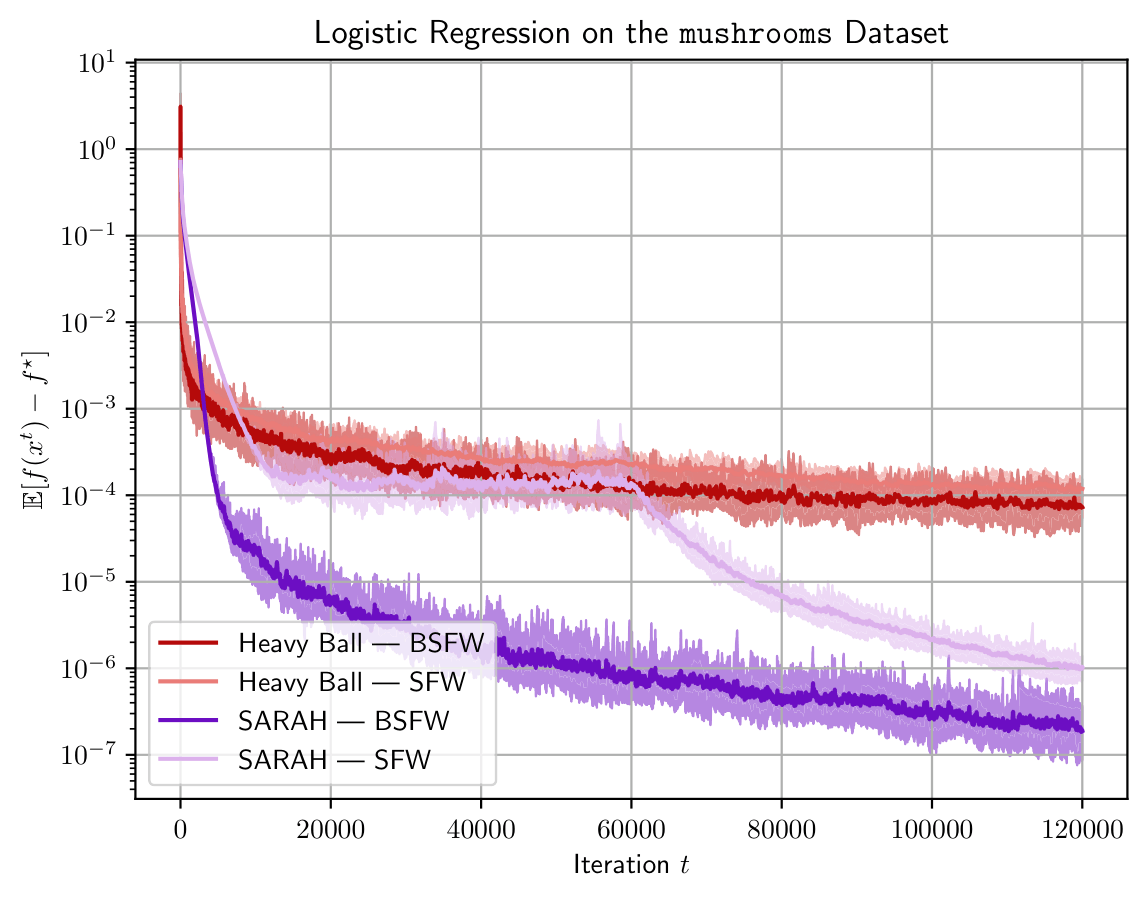}
    \caption{Numerical expected loss vs iteration $t$ on the different datasets.}\label{fig: exp_loss}
\end{figure}

The number of FW oracles called per-iteration $t$ as part of the boosting procedure depends on the alignment tolerance $\delta$ passed to Algorithm~\ref{alg:BSFW}. By definition of the boosting procedure, if the alignment tolerance $\delta$ is small, it is normal to expect that a higher number of rounds will be needed to reach a satisfactory alignment (until the while loop in Algorithm~\ref{alg:boosting} is exited). Each refinement round requires exactly $1$ Frank-Wolfe oracle, and thus it is expected to see more oracles computed per-iteration $t$ when $\delta$ is very small. The experimental results showcase this phenomenon in Figure~\ref{fig:oracle_grid}.

\begin{figure}[hbpt]
    \centering
    \includegraphics[width=1\textwidth]{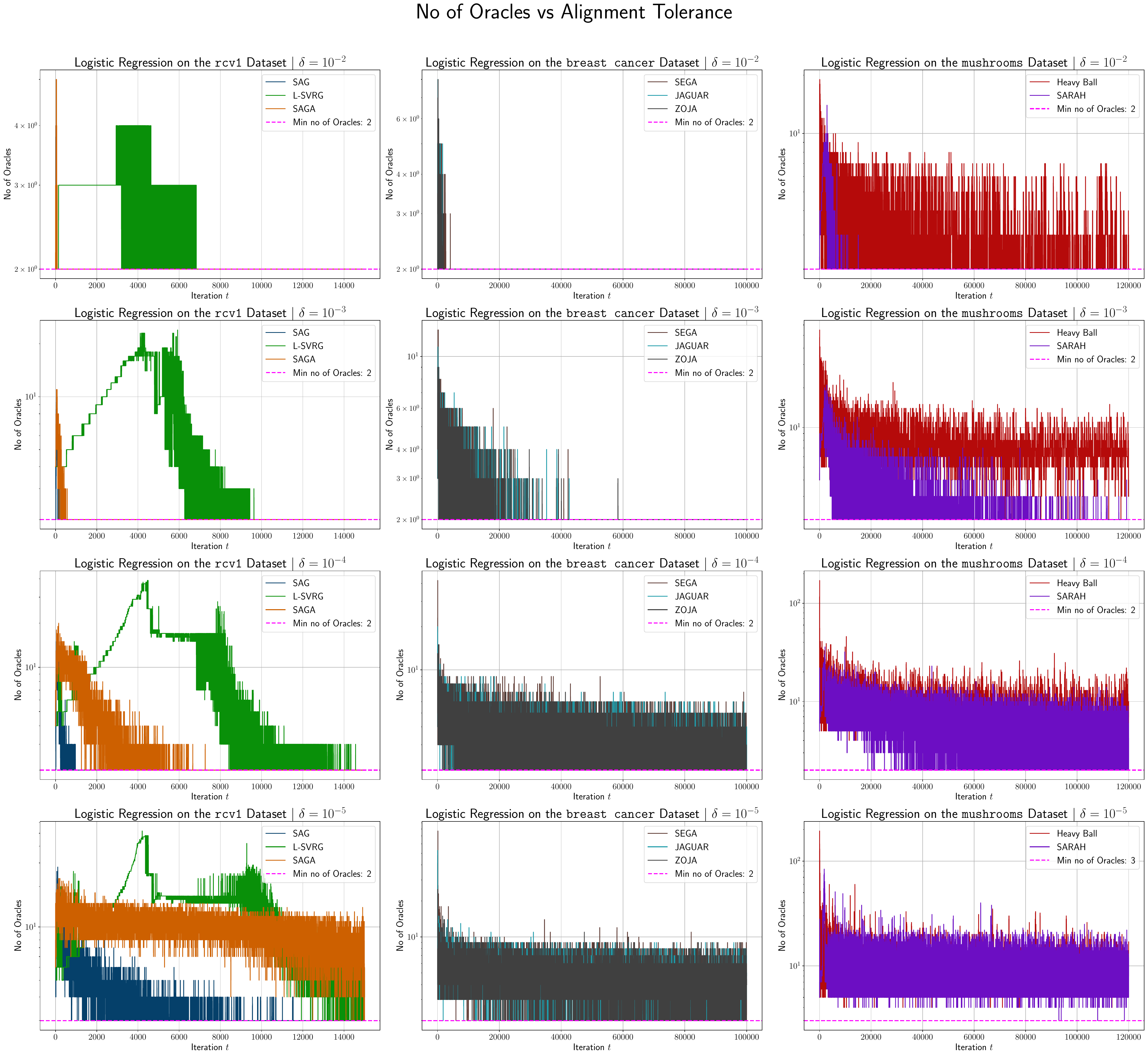}
    \caption{As the tolerance $\delta$ decreases, more oracles are used per-iteration for the estimators in general. Notably, the minimum number of oracle calls used by SARAH and Heavy Ball when $\delta = 10^{-5}$ becomes 3 instead of 2 as observed in the other experiments.}
    \label{fig:oracle_grid}
\end{figure}

Algorithm~\ref{alg:BSFW} is designed to revert back to stochastic FW if $\gamma_t = 1$. We denote the \emph{boosting percentage} by \eqref{eq: boosting_per} when the algorithm is executed for a total of $T$ iterations. 
\begin{equation}
    \label{eq: boosting_per}
    \text{Boosting Percentage} = \frac{\sum_{t=0}^{T-1} \mathbf{1}_{\{\gamma_t < 1\}}}{T} \times 100.
\end{equation}
This measure refers to the percentage of iterations where $\gamma_t < 1$, meaning it does not revert to stochastic FW. Our experimental results confirm that reversion to stochastic FW does not occur in nearly $100\%$ of the iterations for all experiments as seen in Figure~\ref{fig: gamma_plot}.

\begin{figure}[hbpt]
    \centering
    \includegraphics[width=0.32\textwidth]{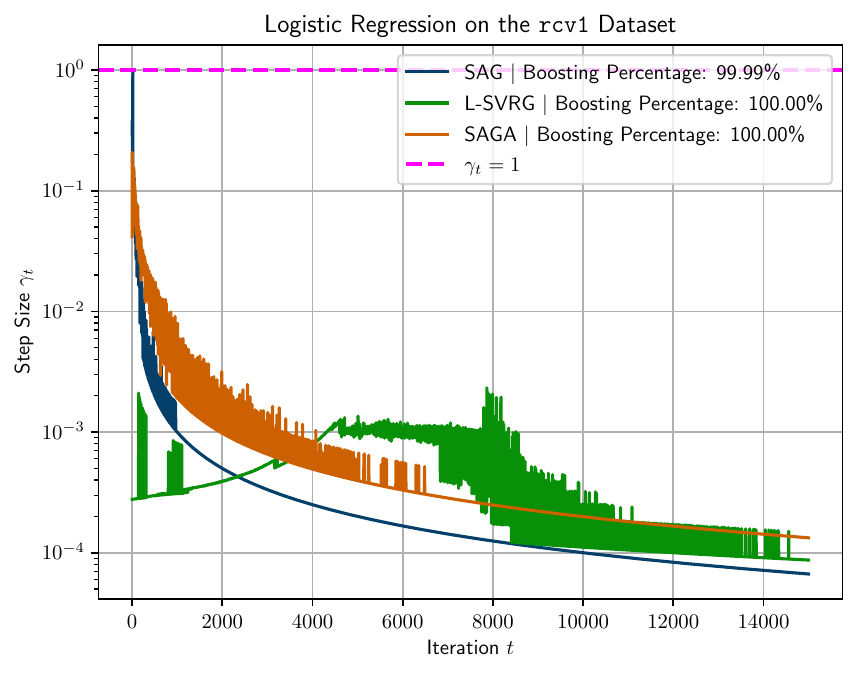}
    \includegraphics[width=0.32\textwidth]{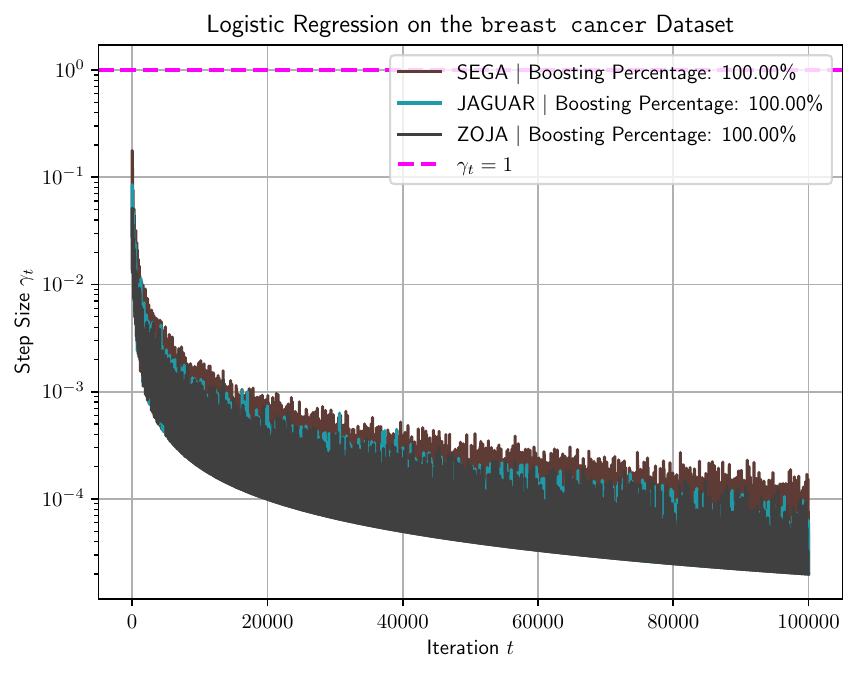}
    \includegraphics[width=0.32\textwidth]{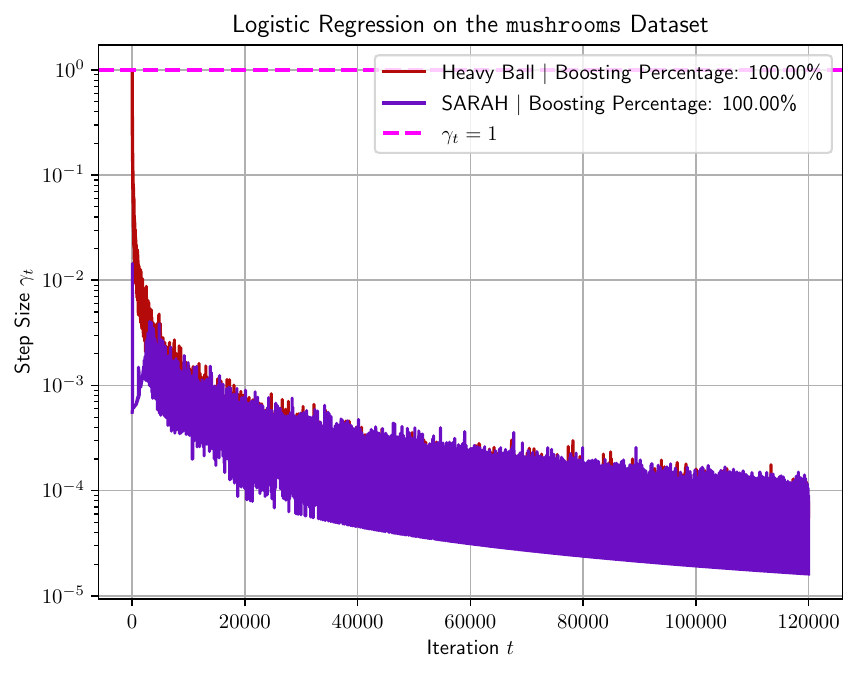}
    \caption{The step sizes $\{\gamma_t\}$ for all of the different estimators are effectively never equal to $1$ (shown as the dashed pink line), so that the reversion to the FW direction is never used. Our proposed scaling results in a step that tends to decrease.}\label{fig: gamma_plot}
\end{figure}

In Theorem~\ref{thm:stoch_fixed_horizon_quasar_convex}, we show the convergence analysis for $\rho$-quasar-convex functions using a similar piecewise step decay as provided in \cite{nazykov2024stochastic}. However, the performance results were not as strong compared to the step decay provided in Theorem~\ref{thm:stoch_quasar_convex_param_ag}. Figure \ref{fig: pw_res} shows the performance results using this piecewise step decay. We pass in the same parameters used in section~\ref{sec: experiments}, notably the parameters in Table~\ref{table: exp_params}, an alignment tolerance $\delta = 10^{-4}$, and the max number of boosting rounds $K = 10^4$.

\begin{figure}[hbpt]
    \centering
    \includegraphics[width=0.32\textwidth]{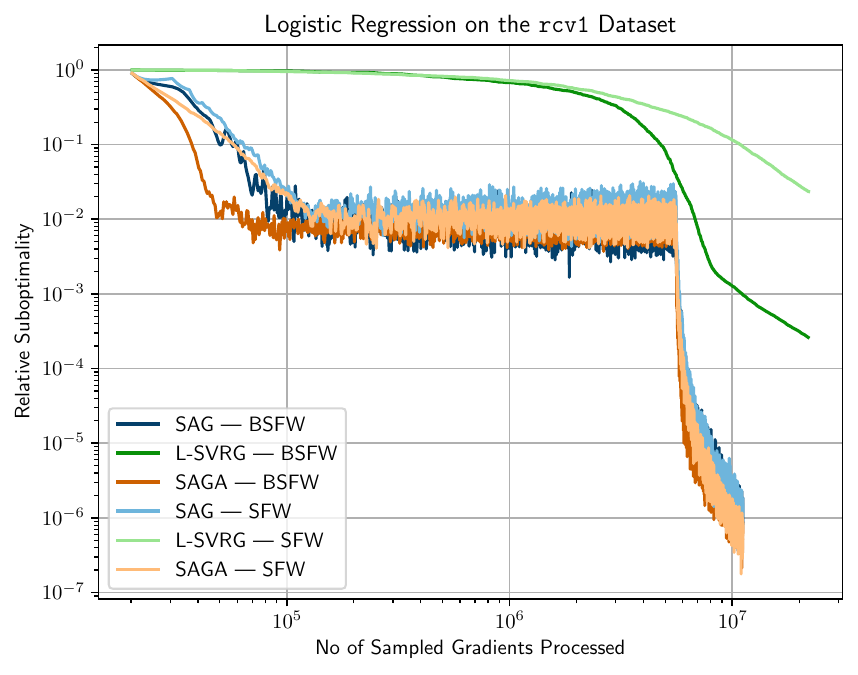}
    \includegraphics[width=0.32\textwidth]{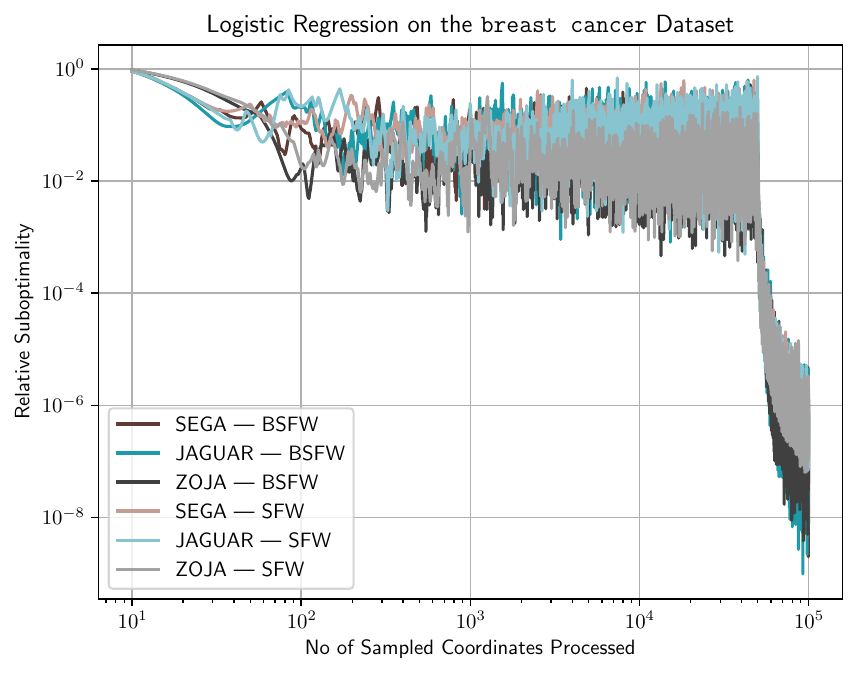}
    \includegraphics[width=0.32\textwidth]{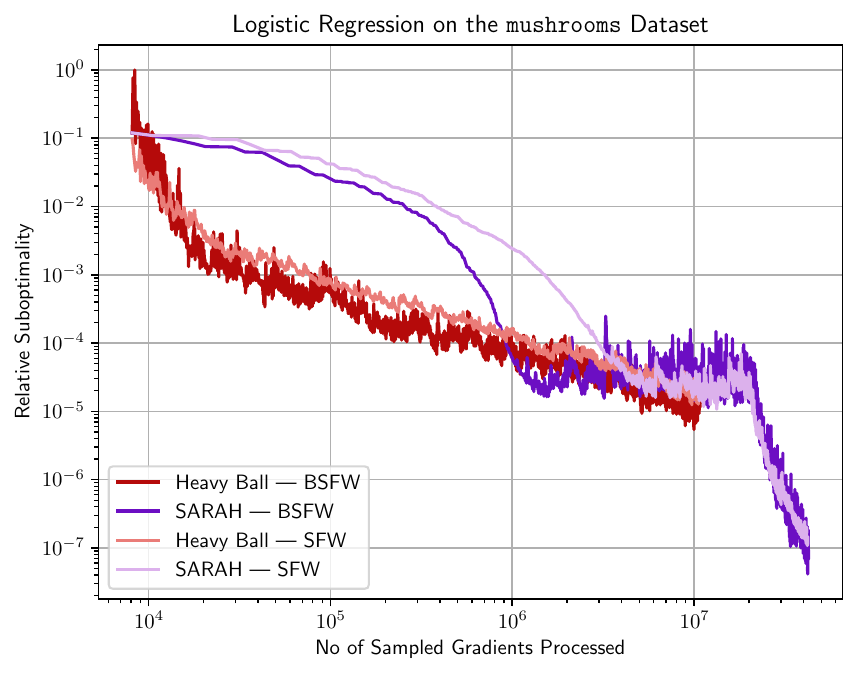}
    \caption{Relative suboptimality when using piecewise step decay on the different datasets.}\label{fig: pw_res}
\end{figure}

\subsection{Quantum Process Tomography}\label{exp subsec: QPT}

Quantum process tomography (QPT) aims to reconstruct the quantum
process $\chi$, given a set of measurements $\{f_{ijk}\}$ and a set
of measurement sensing operators $\{\mathcal{A}_{ijk}\}$. Here, $i$
refers to the input state required to generate the setup, $j$ refers
to a specific circuit setting of that input state, and $k$ refers to a
measurement outcome for that setting. In an $\tilde{n}$-qubit system,
by using Pauli bases \cite{quiroga2023using}, we have $4^{\tilde{n}}$
input states, $3^{\tilde{n}}$ circuit settings, and $2^{\tilde{n}}$
measurement outcomes possible; we have a total of $24^{\tilde{n}}$
sets of $(\mathcal{A}_s, f_s)$ combinations. $\mathcal{A}_s:\mathbb{C}^{d^2
\times d^2} \rightarrow \mathbb{R}^{2^n}$ is the sensing mechanism,
and the loss functions $F(\cdot)$ and $H(\cdot)$ are defined by
$$
    F(UU^{\dagger})
        := \frac{1}{2}\sum_{s}
            \bigl(f_{s} - \mathcal{A}_s(UU^{\dagger})\bigr)^2 ~~~\text{and}~~~
    H(UU^{\dagger})
        := \bigl\|
            \sum_{i,j}(UU^{\dagger})_{ij}\,
            \tilde{A}^{\dagger}_j \tilde{A}_i - I
            \bigr\|_F^2,
$$
where $\{\tilde{A}_k\}$ denotes the set of Pauli bases
\cite{quiroga2023using}, $F(\cdot)$ is the least-squares fidelity
function, and $H(\cdot)$ enforces the trace-preserving (TP) condition.

We consider the nonconvex problem using the Burer–Monteiro (BM) factorization,
where $\chi = UU^{\dagger}$ is the quantum process to be approximated
and $U^{\dagger}$ denotes the conjugate transpose of $U$. We
approximate $\chi$ by rank-1 matrices $U$ and implement the loss
functions faithfully as in \cite{quiroga2023using}:
\begin{equation}\tag{QPT}\label{exp eq: QPT}
    \min_{U \in \mathcal{C}}~
        F(UU^{\dagger}) + \lambda \cdot H(UU^{\dagger}),
\end{equation}
where the constraint set $\mathcal{C}$ for $\tilde{n}$-qubit systems
is
\begin{equation}\label{exp eq: QPT C def}
    \mathcal{C} =
        \bigl\{U \in \mathbb{C}^{4^{\tilde{n}} \times 1} :
        \|U\|_{\mathrm{op}} \leq \tau\bigr\}.
\end{equation}
This choice of $\mathcal{C}$ proves to be an effective one for improving the fidelity with Boosted FW. The spectral-norm LMO for a non-zero rank-1 matrix $G$ in the
constraint set $\mathcal{C}$ of radius $\tau$ is
\begin{equation}\label{app qpt eq: lmo def}
    V^* \in \mathrm{lmo}(G)
    \iff
    V^* = -\tau\,\tfrac{G}{\|G\|_2}.
\end{equation}

We run experiments for $3$-qubit systems (i.e.\ $\tilde{n}:=3$). To
set up each experiment, we first generate the ground-truth process
$\chi^{\star}$ following \cite{quiroga2023using} and then add
Gaussian noise at level $\xi$:
\begin{equation}\label{app qpt eq: noise addn}
    f_s = A_{s}(\chi^{\star}) + \xi\,\epsilon,
    \qquad \epsilon \sim \mathcal{N}(0,1).
\end{equation}
We set $\tau = 10$ after a grid search to identify the constraint
radius that yields the best attainable fidelity, and we set
$\lambda = 0.05$ in~\eqref{exp eq: QPT}. For all four experiment
settings we use the ``full-measurements'' setting, i.e.\ all
$24^{\tilde{n}}$ measurements per iteration. Wall-clock experiments
are timed on A100 GPUs.

Following \cite{quiroga2023using}, performance is measured by the
quantum process fidelity $F_p$, whose definition for two quantum
channels is \cite{javadi2024quantum}
\begin{equation}\label{app qpt eq: pro fid}
    F_p(\mathcal{E}, \mathcal{E}^{\star})
        = F_s\!\left(
            \tilde{\rho}_{\mathcal{E}},\,
            \tilde{\rho}_{\mathcal{E}^{\star}}
          \right),
\end{equation}
where $F_s$ denotes the quantum state fidelity, and
$\tilde{\rho}_{\mathcal{E}}$, $\tilde{\rho}_{\mathcal{E}^{\star}}$
are the normalized Choi-matrix representations of the channels
$\mathcal{E}$ and $\mathcal{E}^{\star}$, respectively (see
\cite{javadi2024quantum,wood2015tensornetworksgraphicalcalculus} for
details). Note that $\chi$ and $\chi^{\star}$ represent channels
$\mathcal{E}$ and $\mathcal{E}^{\star}$ through the specified basis
\cite{quiroga2023using}. The closer $F_p$ is to $1$, the better the
reconstruction. We implement this metric using the \texttt{qiskit}
library \cite{javadi2024quantum}.

From \cite{quiroga2023using}, four different levels of Gaussian noise
(parameterized by $\xi$) are added to the ground-truth measurement
process. We plot the fidelity achieved per wall-clock time in
Figure~\ref{fig: qpt fidelity vs wall clock} for different values of
the maximum number of oracles $K$, in order to compare the effect of
boosting. Note that $K=1$ corresponds to the vanilla FW method.

\begin{figure}[H]
    \centering
    \includegraphics[width=1.0\textwidth]{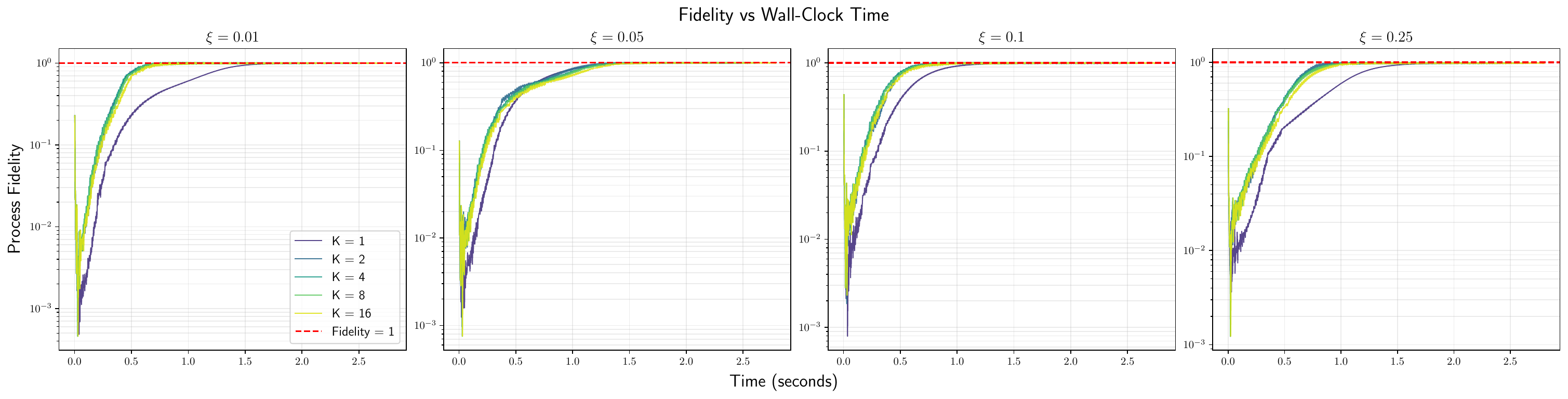}
    \caption{For all four experiments we set the radius $\tau = 10$.
             The red dotted line illustrates the target fidelity.
             Computing at least one additional LMO for boosting
             consistently yields higher fidelity than vanilla FW.}
    \label{fig: qpt fidelity vs wall clock}
\end{figure}

The per-iteration fidelity curves corresponding to
Figure~\ref{fig: qpt fidelity vs wall clock} are provided in
Figure~\ref{fig: qpt fidelity vs iter} for readers interested in
comparing per-iteration performance.

\begin{figure}[H]
    \centering
    \includegraphics[width=1.0\textwidth]{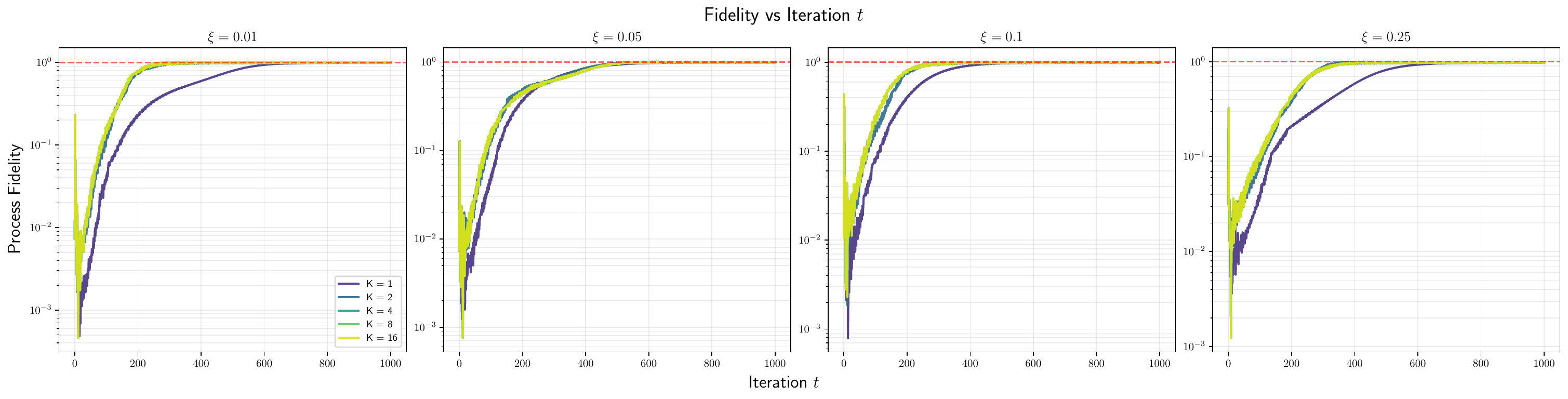}
    \caption{The fidelity results from
             Figure~\ref{fig: qpt fidelity vs wall clock} plotted as
             a function of iteration $t$.}
    \label{fig: qpt fidelity vs iter}
\end{figure}


\subsection{Collaborative Filtering}\label{app subsec: movielens}
We consider the following problem
\begin{equation*}
    \min\limits_{X\in\mathcal{C}} ~\frac{1}{|\Omega|}\sum\limits_{(i,j)\in\Omega}\ell(X_{ij},Y_{ij}),
    \end{equation*}
where $\mathcal{C} = \{X\in\mathbb{R}^{n\times m}\colon ~\|X\|_{\ast}\leq\tau\}$ is the nuclear norm ball of radius $\tau>0$, $\Omega$ is the observed entries of the matrix, and $\ell(x,y) = \tfrac{(x-y)^2}{2 + (x-y)^2}$, which is nonconvex. The LMO of the nuclear norm ball with radius $\tau$ is proportional to the outer product of the leading singular vectors
\begin{align}
    \label{eq: lmo_nuc}
    \forall G \in \mathbb{R}^{n\times m}: \quad v^{\ast} = \lmo(G) \iff v^{\ast} = -\tau u_1v_1^T 
    ~~\text{with}~~ 
    G = U\Sigma V^T.
\end{align}

We use the {\ttfamily MovieLens} dataset \cite{harper2015movielens} and run Algorithm~\ref{alg:BSFW} and SFW with the L-SVRG, Heavy Ball, and SAGA estimators. We use the horizon-dependent step size given in Theorem~\ref{thm:stoch_fixed_horizon_nonconvex} for Algorithm~\ref{alg:BSFW} and the parameters given in \cite{nazykov2024stochastic} for SFW, except for Heavy Ball for which we take the parameters given in \cite{pethick2025training}. Because the LMO in \eqref{eq: lmo_nuc} is reasonably \textbf{expensive} compared to the one used in \eqref{exp:sparsity_problem}, we set a smaller $K=5$ and $\delta = 10^{-4}$ for the boosting procedure. We set the batch size $b_s$ to be $10\%$ of the dataset (although we observed a larger gap between Algorithm~\ref{alg:BSFW} and SFW when the batches were larger). The Heavy Ball estimator outperformed the other estimators in this setting despite having a slower worst-case convergence rate guarantee. For all three estimators, boosting produced an improvement in performance.

\begin{figure*}[hbpt!]
    \centering
    \includegraphics[width=0.5\textwidth]{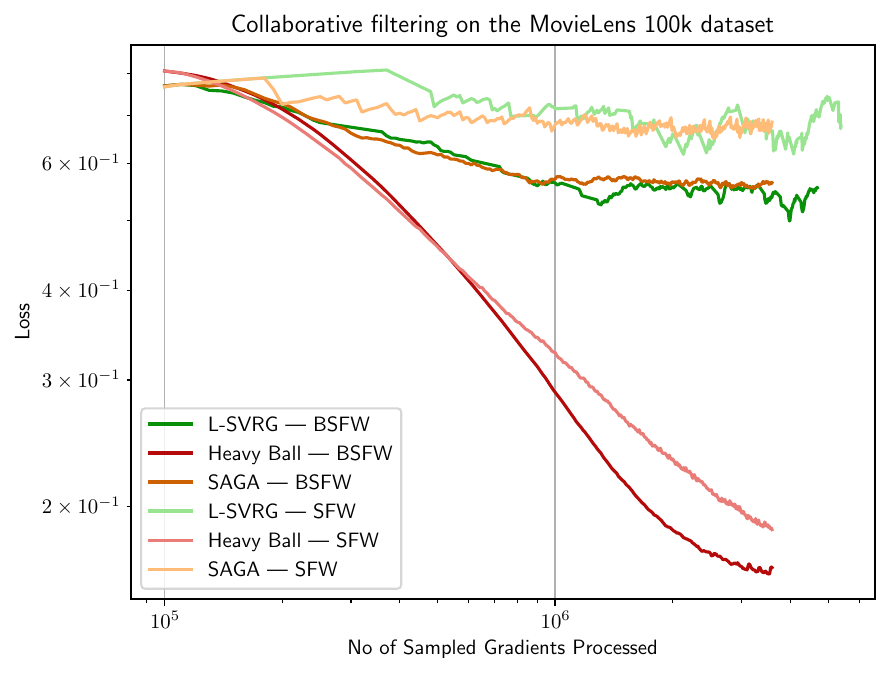}
    \caption{For this collaborative filtering problem, we plot the loss directly.}\label{fig: convergence_plots_movielens}
\end{figure*}




\newpage \appendix
 \begin{Large}
 \begin{center}
 \textbf{Appendices}
 \end{center}
 \end{Large}
The appendix of this work is divided into three sections. The first section contains more information about Algorithm~\ref{alg:BSFW} and some additional discussions. The second contains the proofs of the theoretical results presented in the main text. The third section provides detailed descriptions of the gradient estimators, along with additional theoretical results, e.g., proofs that they satisfy the assumptions made in the main text and in the subsequent proofs. 


\section{Additional Discussions}
\subsection{Complete Algorithm BSFW}\label{subsec:complete BSFW Full}
The complete Algorithm~\ref{alg:BSFW Full} is written here for reference. Note that we show in Algorithm~\ref{alg:BSFW Full} that in fact, $\mathrm{lmo}(m^t)$ at iteration $t$ is already computed during the boosting procedure and hence does not need to be computed again in lines 6 and 11 of Algorithm~\ref{alg:BSFW}.
\renewcommand{\thealgorithm}{BSFW (Full)}
\begin{algorithm}[hpt!]
\caption{Boosted Stochastic Frank-Wolfe}
\label{alg:BSFW Full}
\textbf{Input:} initial estimator $m^{\mathrm{init}} \in \mathbb{R}^n$, gradient estimator $\{\Phi_t\}_{t=0}^{T-1}$, max number of rounds for boosting $K \in \mathbb{N} \setminus \{0\}$, alignment improvement tolerance $\delta \in ]0,1]$, and step decay $\{\eta_t\}_{t=0}^{T-1} \in ]0,1]$\\
\textbf{Output:} $x^T \in \mathcal{C}$.
\begin{algorithmic}[1]

\STATE $x^0 \gets \lmo(m^{\mathrm{init}})$
\FOR{$t = 0 \text{ to } T-1$}
    \STATE Sample $\xi_t\sim\mathcal{P}$ and compute $g^t = \nabla f(x^t,\xi_t)$
    \STATE Compute $m^t = \Phi_t(g^t)$
    \STATE $\psi^0 \gets 0$
    \STATE $\Lambda_t \gets 0$
    \STATE $k \gets 0$
    \WHILE{$k \leq K - 1$}
        \STATE $r^k \gets -m^t - \psi^k$ \COMMENT{$k$-th residual}
        \STATE $v^k \gets \lmo(-r^k)$ \COMMENT{FW oracle}
        \IF{$k = 0$}
            \STATE $s^t \gets v^k$
        \ENDIF
        \IF{$\psi^k\neq 0$}
            \STATE $u^k \gets \argmax_{u \in\left\{v^k - x^t, -\tfrac{\psi^k}{\|\psi^k\|}\right\}}\langle r^k,u\rangle$
        \ELSE
            \STATE $u^k\gets v^k - x^t$
        \ENDIF
        \IF{$u^k = 0$}
            \STATE $k\gets k+1$; \textbf{break} \COMMENT{exit $k$-loop}
        \ENDIF
        \STATE $\lambda_k \gets \tfrac{\langle r^k, u^k \rangle}{\|u^k\|^2}$
        \STATE $\phi^k \gets \psi^k + \lambda_k u^k$
        \IF{$\alignop(-m^t, \phi^k) - \alignop(-m^t, \psi^k) \geq \delta$}
            \STATE $\psi^{k+1} \gets \phi^k$
            \STATE $\Lambda_t \gets 
            \begin{cases} 
                \Lambda_t + \lambda_k, &  u^k = v^k - x^t \\
                \Lambda_t \left(1 - \frac{\lambda_k}{\|\psi^k\|}\right), &  u^k = -\frac{\psi^k}{ \|\psi^k\|}
            \end{cases}$
            \STATE $k\gets k+1$
        \ELSE
            \STATE $k\gets k+1$; \textbf{break} \COMMENT{exit $k$-loop}
        \ENDIF
    \ENDWHILE
    \STATE $K_t \gets k$
    \STATE $\tilde{d}^t \gets \dfrac{\psi^{K_t}}{\Lambda_t}$ \textbf{if} $\Lambda_t \neq 0$ \textbf{else} $0$ \COMMENT{normalize direction}
    \STATE $\gamma_t\gets \min\left\{\eta_t\tfrac{\|s^t-x^t\|}{\|\tilde{d}^t\|},1\right\}$ \textbf{if} $\tilde{d}^t \neq 0$ \textbf{else} $1$
    \IF{$\gamma_t < 1$}
        \STATE $d^t \gets \tilde{d}^t$ \COMMENT{use the boosted direction}
        \STATE $x^{t+1} \gets x^t + \gamma_t d^t$
    \ELSE
        \STATE $d^t \gets s^t - x^t$ \COMMENT{revert to vanilla FW direction}
        \STATE $x^{t+1} \gets x^t + \eta_td^t$ 
    \ENDIF
\ENDFOR
\end{algorithmic}
\end{algorithm}


\subsection{Quasar-Convex Applications in Machine Learning}\label{app subsec: quasar convex appns in ML}
There are many problems in machine learning involving quasar-convex functions which have been explored prior to our work; we list a few of them here. The problem of learning linear dynamical systems (LDS) has been explored in \cite{fu2023accelerated} and \cite{hardt2018gradient}, where specifically in \cite{hardt2018gradient}, the authors prove that their objective function is quasar-convex but not necessarily convex. In \cite{ma2021local} and \cite{foster2018uniform}, the authors show that generalized linear models (GLM) with increasing link functions (among other assumptions) are quasar-convex, and study their convergence analysis. In \cite{foster2018uniform}, the authors also show that the problem of robust linear regression using Tukey’s biweight loss is quasar-convex. Furthermore, in \cite{wang2023continuized}, the authors show that GLMs (satisfying certain other conditions) belong to the family of quasar-convex functions (and develop algorithms for this problem class).

\subsection{When to Boost?}\label{app subsec: when to boost}
Boosting is recommended in problems where the cost of gradient computation is expensive relative to the cost of the LMO. It is then advantageous to compute the LMO several times per gradient computation in order to get an update that is better aligned with the gradient or its stochastic estimator. For instance, for matrices of size $100 \times 100$, the ratio of nuclear norm LMO vs gradient (for a sample least squares setting) complexity (in terms of wall-clock time) is on the order of $10^{-1}$ while that of $\ell_1$ ball LMO for vectors of size $10^4$ vs gradient (for the same sample least squares setting) complexity is in the order of $\mathcal{O}(10^{-3})$.

\subsection{Complexity of Step Size}\label{app subsec: complexity of step size}
The step size strategy we prescribe in Algorithm~\ref{alg:BSFW} depends only on $s^t \in \lmo(m^t)$ and $\tilde{d}^t$, which are already computed during the boosting procedure. Therefore, it does not impose additional computational complexity, unlike the line search used in \cite{combettes2020boosting} when the Lipschitz constant of $\nabla f$ is not known.

\section{Proofs of Main Results}\label{app:proofs}
In this section, we present the proofs for the core technical results and the main convergence theorems. We divide the analysis of Algorithm~\ref{alg:BSFW} into the deterministic setting in section~\ref{appendix: boosted_det_fw} and the stochastic setting in section~\ref{appendix: boosted_stoch_fw}. We begin with a geometric property of the alignment direction used in our variants.

\begin{lemma}[Alignment Inequality]\label{lem:alignment}
    Consider the sequence $\{x^t\}_{t=0}^{+\infty}$ generated by Algorithm~\ref{alg:BSFW}. Then for all $t \in \mathbb{N}$, the following inequality holds
    \begin{equation*}
        \langle m^t, d^t \rangle \leq \frac{\|d^t\|}{\|s^t - x^t\|} \langle m^t, s^t - x^t \rangle.
    \end{equation*}
    where $d^t$ denotes the alignment direction at iteration $t$ and $s^t \in \lmo(m^t)$.
\end{lemma}
\begin{proof}
     By construction of Algorithm~\ref{alg:BSFW}, we have $K_t \geq 1$ for all $t \in \mathbb{N}$. \\
     \noindent \textbf{Case I}: Suppose $\gamma_t < 1$. Then we have $d^t = \tilde{d^t}$. Applying Proposition~3.1 from \cite{combettes2020boosting}, we obtain
    \begin{equation*}
        \alignop(-m^t, d^t) \geq \alignop(-m^t, s^t - x^t).
    \end{equation*}
    By the definition of the alignment operator $\alignop$, this yields
    \begin{equation*}
        \frac{\langle -m^t, d^t \rangle}{\|m^t\|\|d^t\|} 
            \geq \frac{\langle -m^t, s^t - x^t \rangle}{\|m^t\|\|s^t - x^t\|}
        \implies \frac{\langle m^t, d^t \rangle}{\|d^t\|}
            \leq \frac{\langle m^t, s^t - x^t \rangle}{\|s^t - x^t\|}
        \implies \langle m^t, d^t \rangle
            \leq \frac{\|d^t\|}{\|s^t - x^t\|} \langle m^t, s^t - x^t \rangle.
    \end{equation*}
    \noindent \textbf{Case II}: Suppose $\gamma_t = 1$. Then $d^t = s^t - x^t$ and we have 
    \begin{equation*}
        \langle m^t, s^t - x^t \rangle \leq \frac{\|s^t - x^t\|}{\|s^t - x^t\|} \langle m^t, s^t - x^t \rangle.
    \end{equation*}
\end{proof}


\subsection{Boosted Deterministic Frank-Wolfe}
\label{appendix: boosted_det_fw}
We first provide descent Lemma~\ref{lem:recursion_det_base} which will help us prove the convergence analysis for $\rho$-quasar-convex objective functions in Theorem~\ref{proof thm:det_quasar} (any-time rate) and nonconvex functions in Theorem~\ref{proof thm:det_nonconvex} (any-time rate) and Theorem~\ref{thm:det_ncv_fixed_horizon} (fixed-horizon rate).

\begin{lemma}[Descent under Smoothness in Deterministic Setting]\label{lem:recursion_det_base}
Let $\{x^t\}_{t=0}^{+\infty}$ be the sequence generated by Algorithm~\ref{alg:BSFW}, where at every iteration $t$, $m^t = \nabla f(x^t)$.
Assume that $f$ satisfies $L$-smoothness Assumption~\ref{assum: L_smooth}. Then, for all $t \geq 0$, the following descent property holds
\begin{equation*}
f(x^{t+1}) \leq f(x^t) + \eta_t \langle \nabla f(x^t), s^t - x^t \rangle + \frac{L}{2} \eta_t^2 \|s^t - x^t\|^2.
\end{equation*}
\end{lemma}

\begin{proof}
By $L$-smoothness of $f$, for any feasible direction $d^t$ satisfying $x^{t}+\gamma_t d^t \in \mathcal{C}$ and any $\gamma_t \in [0,1]$, we have
\begin{equation}\label{eq:smoothness}
    f(x^{t}+\gamma_t d^t) \leq f(x^t) + \gamma_t \langle \nabla f(x^t), d^t \rangle + \frac{L}{2} \gamma_t^2 \|d^t\|^2.
\end{equation}
We distinguish two cases based on the value of $\gamma_t$, following the analysis of Algorithm~\ref{alg:BSFW}.  \\ 
\noindent\textbf{Case I}: $\gamma_t<1$. In this case, $\gamma_t = \eta_t \frac{\|s^t - x^t\|}{\|d^t\|}$ since $\tilde{d}^t = d^t$. Using Lemma~\ref{lem:alignment} and substituting into \eqref{eq:smoothness} yields
\begin{equation*}
\begin{aligned}
f(x^{t+1})
&\leq f(x^t) + \left(\eta_t\frac{\|s^t-x^t\|}{\|d^t\|}\right)\langle\nabla f(x^t),d^t\rangle + \frac{L}{2}\left(\eta_t\frac{\|s^t-x^t\|}{\|d^t\|}\right)^2 \|d^t\|^2
\\
&\leq f(x^t) + \left(\eta_t\frac{\|s^t-x^t\|}{\|d^t\|}\right) \left(\frac{\|d^t\|}{\|s^t - x^t\|} \langle \nabla f(x^t), s^t - x^t \rangle\right) + \frac{L}{2}\left(\eta_t\frac{\|s^t-x^t\|}{\|d^t\|}\right)^2 \|d^t\|^2
\\
&= f(x^t) + \eta_t \langle \nabla f(x^t), s^t - x^t \rangle + \frac{L}{2} \eta_t^2 \|s^t-x^t\|^2,
\end{aligned}
\end{equation*}
where the first inequality follows from \eqref{eq:smoothness} with the substitution $\gamma_t = \eta_t \frac{\|s^t - x^t\|}{\|d^t\|}$, the second inequality follows from Lemma~\ref{lem:alignment}, and the equality follows from algebraic simplification.

\noindent\textbf{Case II}: $\gamma_t=1$. Here, the direction is $d^t = s^t - x^t$ (as vanilla Frank-Wolfe direction) and the update is
$x^{t+1} = x^t + \eta_t d^t$. In this case, similarly, from $L$-smoothness we have
\begin{equation*}
f(x^{t+1})\leq f(x^t) + \eta_t \langle \nabla f(x^t), s^t - x^t \rangle + \frac{L}{2} \eta_t^2 \|s^t-x^t\|^2.
\end{equation*}
Thus, in both cases, the decrease inequality holds.
\end{proof}
We are ready to present convergence guarantees under two complementary regimes: structured (quasar-convex) objectives and general nonconvex smooth functions, in the following subsections.
\subsubsection{Quasar Convex Case}
\begin{theorem}[Formal Statement of Theorem~\ref{thm:det_quasar}]
\label{proof thm:det_quasar}
Let $f$ be a function that satisfies $L$-smoothness Assumption~\ref{assum: L_smooth} and $\rho$-quasar-convexity Assumption~\ref{assum: quasar-convexity}. 
Consider the sequence $\{x^t\}_{t=0}^{+\infty}$ generated by Algorithm~\ref{alg:BSFW} with $m^t=\nabla f(x^t)$ and $\eta_t = \frac{2}{\rho(t+2)}$. 
Then, for all $t\geq 0$,  
\begin{equation*}
    F_t \leq \frac{1}{t+1} \max \left\{F_0,\; \frac{2 L D^2}{\rho^2}\right\} = \mathcal{O}(1/t).
\end{equation*}
\end{theorem}

\begin{proof}
From Lemma \ref{lem:recursion_det_base} we have
\begin{equation*}
\begin{aligned}
f(x^{t+1}) - f(x^t) 
&\leq  \eta_t \langle \nabla f(x^t), s^t - x^t \rangle + \frac{L}{2} \eta_t^2 \|s^t-x^t\|^2\\ 
&\leq  \eta_t \langle \nabla f(x^t), x^\star- x^t \rangle + \eta_t^2 \frac{L D^2}{2}\\ 
&\leq \eta_t \rho (f(x^\star)- f(x^t)) + \eta_t^2 \frac{L D^2}{2},
\end{aligned}
\end{equation*}
where the second inequality uses $s^t \in \lmo(\nabla f(x^t))$ and the last follows from $\rho$-quasar-convexity of $f$.
Recall that $F_t = f(x^t) - f(x^\star) \geq 0$. Rearranging the terms and subtracting $f(x^\star)$ from both sides:
\begin{equation}\label{eq:recursion}
F_{t+1} \leq  F_t - \rho \eta_t F_t + \frac{L D^2}{2} \eta_t^2 = (1 - \rho \eta_t) F_t + \frac{L D^2}{2} \eta_t^2.
\end{equation}
Now we use $\eta_t= \frac{2}{\rho(t+2)}$, substituting it into 
\begin{equation}\label{eq:inequalty-short-step}
    F_{t+1} \leq \frac{t}{t+2} F_t + \frac{2 L D^2}{\rho^2 (t+2)^2}.
\end{equation}
Set  $B:= \max \left\{F_0, \frac{2L D^2}{\rho^2} \right\}$. We prove by induction that for all $t \geq 0$,
\begin{equation}\label{eq:induction}
 F_t \leq \frac{B}{t+1}.   
\end{equation}
Base case ($t = 0$). By definition of $B$, we have $F_0 \leq B$, and thus
\begin{equation*}
    F_0 \leq \frac{B}{1} = \frac{B}{0 + 1},
\end{equation*}
which establishes the claim for $t = 0$.
\newline
Inductive step. Assume \eqref{eq:induction} holds for some $t \geq 0$. Using \eqref{eq:inequalty-short-step} and the induction hypothesis,
\begin{align*}
F_{t+1}
&\leq \frac{t}{t+2} \cdot \frac{B}{t+1} + \frac{2 L D^2}{\rho^2 (t+2)^2} \\
&\leq \frac{t B}{(t+2)(t+1)} + \frac{B}{(t+2)^2}
\\
&\leq \frac{t B}{(t+2)(t+1)} + \frac{B}{(t+2)(t+1)}\\
&= \frac{(t+1) B}{(t+2)(t+1)}\\
&= \frac{B}{t+2}.
\end{align*}
This establishes \eqref{eq:induction} for $t+1$, completing the induction.
\end{proof}

\begin{remark}
This result immediately applies to classical convex and star-convex settings, as they both satisfy $\rho$-quasar-convexity with $\rho = 1$. Hence, Theorem~\ref{thm:det_quasar} ensures a $\mathcal{O}(1/t)$ convergence rate in these cases.
\qed
\end{remark}

\subsubsection{Nonconvex Case}
In the absence of any convexity-like assumptions, we can still guarantee convergence in terms of the Frank–Wolfe gap—a standard stationarity measure for constrained nonconvex optimization.

\begin{theorem}[Formal Statement of Theorem~\ref{thm:det_nonconvex}] 
\label{proof thm:det_nonconvex}
Let $f$ be a function that satisfies $L$-smoothness Assumption~\ref{assum: L_smooth}. 
Consider the sequence $\{x^t\}_{t=0}^{+\infty}$ generated by Algorithm~\ref{alg:BSFW} with $m^t=\nabla f(x^t)$ and $\eta_t = \frac{1}{\sqrt{t+1}}$. 
Then, for all $t\geq 0$,
\begin{equation*}
    \min_{0 \leq i \leq t} \gap(x^i) 
\leq \frac{F_0 + L D^2}{\sqrt{t+1}} = \mathcal{O}(1/\sqrt{t}).
\end{equation*}
\end{theorem}

\begin{proof}
From Lemma~\ref{lem:recursion_det_base}, we have
\begin{align*}
&f(x^{t+1})\leq f(x^t) + \eta_t \langle \nabla f(x^t), s^t - x^t \rangle + \frac{L}{2} \eta_t^2 \|s^t-x^t\|^2.
\end{align*}
Rearranging the inequality yields 
\begin{align*}
\eta_t \langle \nabla f(x^t), x^t-s^t \rangle\leq f(x^t) -f(x^{t+1}) + \frac{L}{2} \eta_t^2 \|s^t-x^t\|^2 \leq f(x^t) -f(x^{t+1}) + \frac{L}{2} \eta_t^2 D^2.
\end{align*}
Thus, 
\begin{align*}
\langle \nabla f(x^t), x^t-s^t \rangle \leq \frac{1}{\eta_t}(f(x^t) -f(x^{t+1})) + \frac{L}{2} \eta_t D^2,
\end{align*}
Now substitute $\eta_t = 1/\sqrt{t+1}$. Then, summing the inequality over $t = 0$ to $T$ yields
\begin{align}\label{eq:sum}
\sum_{t=0}^{T} \langle \nabla f(x^t), x^t - s^t \rangle 
&\leq \sum_{t=0}^{T} \sqrt{t+1} \, \left( f(x^t) - f(x^{t+1}) \right) 
   + \frac{L D^2}{2} \sum_{t=0}^{T} \frac{1}{\sqrt{t+1}},
\end{align}
For the first sum on the right-hand side of~\eqref{eq:sum}, note that $\sqrt{t+1} \leq \sqrt{T+1}$ for all $t \leq T$, so
\begin{align*}
\sum_{t=0}^{T} \sqrt{t+1} \, \left( f(x^t) - f(x^{t+1}) \right)
\leq &\sqrt{T+1} \sum_{t=0}^{T} \left( f(x^t) - f(x^{t+1}) \right)
\\ = &\sqrt{T+1} \, \left( f(x^0) - f(x^{T+1}) \right)
\\ \leq &\sqrt{T+1} \, \left( f(x^0) - f(x^\star) \right)
\\ =&\sqrt{T+1}~ F_0.
\end{align*}
For the second sum on the right-hand side of~\eqref{eq:sum}, we use the standard bound
$$
\sum_{t=0}^{T} \frac{1}{\sqrt{t+1}} = \sum_{k=1}^{T+1} \frac{1}{\sqrt{k}} 
\leq \int_{0}^{T+1} \frac{1}{\sqrt{x}} \, dx = 2\sqrt{T+1}.
$$
Combining both estimates, we obtain
$$
\sum_{t=0}^{T} \langle \nabla f(x^t), x^t - s^t \rangle 
\leq \sqrt{T+1} \, \left( F_0 + L D^2 \right).
$$
Since $\langle \nabla f(x^t), x^t - s^t \rangle \geq 0$, it follows that
$$
\left( \min_{0 \leq t \leq T} \langle \nabla f(x^t), x^t - s^t \rangle \right) (T+1) 
\leq \sum_{t=0}^{T} \langle \nabla f(x^t), x^t - s^t \rangle 
\leq \sqrt{T+1} \, \left( F_0 + L D^2 \right),
$$
and therefore
$$
\min_{0 \leq t \leq T} \langle \nabla f(x^t), x^t - s^t \rangle 
\leq \frac{F_0 + L D^2}{\sqrt{T+1}}.
$$
\end{proof}
\begin{theorem}[Fixed Horizon Convergence under nonconvexity in Deterministic Setting]
\label{thm:det_ncv_fixed_horizon}
Let $f$ be an objective function that satisfies $L$-smoothness Assumption~\ref{assum: L_smooth}. Let $T\in\mathbb{N}$ and $\{x^t\}_{t=0}^T$ be a sequence generated by Algorithm~\ref{alg:BSFW} where $m^t = \nabla f(x^t)$, and $\eta_t = \frac{1}{\sqrt{T+1}}$. Then we have
\begin{equation*}
    \min_{0 \leq t \leq T} \langle \nabla f(x^t), x^t - s^t \rangle \leq \frac{F_0 + \frac{L D^2}{2}}{\sqrt{T+1}}.
\end{equation*}
\end{theorem}
\begin{proof}
From Lemma~\ref{lem:recursion_det_base}, we have
\begin{align*}
&f(x^{t+1})\leq f(x^t) + \eta_t \langle \nabla f(x^t), s^t - x^t \rangle + \frac{L}{2} \eta_t^2 \|s^t-x^t\|^2.
\end{align*}
Rearranging the inequality yields 
\begin{align*}
\eta_t \langle \nabla f(x^t), x^t-s^t \rangle\leq f(x^t) -f(x^{t+1}) + \frac{L}{2} \eta_t^2 \|s^t-x^t\|^2 \leq f(x^t) -f(x^{t+1}) + \frac{L}{2} \eta_t^2 D^2,
\end{align*}
where we used $\|s^t - x^t\| \leq D$. Summing both sides from $t = 0$ to $T$ gives 
$$
\sum_{t=0}^{T} \eta_t \langle \nabla f(x^t), x^t - s^t \rangle \leq f(x^0) - f(x^{T+1}) + \frac{L D^2}{2} \sum_{t=0}^{T} \eta_t^2
 \leq f(x^0) - f(x^\star) + \frac{L D^2}{2} \sum_{t=0}^{T} \eta_t^2,
$$
Since $g_t \geq 0$, we have
$$
\left(\min_{0\leq t \leq T} \langle \nabla f(x^t), x^t - s^t \rangle \right) \sum_{t=0}^{T} \eta_t  \leq F_0 + \frac{L D^2}{2} \sum_{t=0}^{T} \eta_t^2.
$$
Now substitute $\eta_t = 1/\sqrt{T+1}$ for all $t = 0, \cdots, T$, Then
$$
\sum_{t=0}^{T} \eta_t = (T+1) \cdot \frac{1}{\sqrt{T+1}} = \sqrt{T+1}
\qquad \text{and}\qquad
\sum_{t=0}^{T} \eta_t^2 = (T+1) \cdot \frac{1}{T+1} = 1.
$$
Thus,
$$
\left( \min_{0 \leq t \leq T} \langle \nabla f(x^t), x^t - s^t \rangle \right) \sqrt{T+1} \leq F_0 + \frac{L D^2}{2},
$$
which implies
$$\min_{0 \leq t \leq T} \langle \nabla f(x^t), x^t - s^t \rangle \leq \frac{F_0 + \frac{L D^2}{2}}{\sqrt{T+1}}.$$
This concludes the proof.
\end{proof}
\subsection{Boosted Stochastic Frank–Wolfe}
\label{appendix: boosted_stoch_fw}

Here, we make Remark~\ref{rem: full grad bounds} on the estimators, which we use in the convergence analysis that follows. We show the analysis for both $\rho$-quasar-convex functions in subsection~\ref{apppendix_A_bsfw: quasar_convex} and nonconvex functions in subsection~\ref{appendix_A_bsfw: non_convex}.

\begin{remark}\label{rem: full grad bounds}
The recursive bounds on the gradient estimation error $\|\Delta^t\|^2$ (recall that $\Delta^t = m^t - \nabla f(x^t)$) and the auxiliary variance term $\sigma_t^2$ follow directly from Assumption~\ref{assum: est}. Applying the law of total expectation to both sides yields the unconditional recursions:
\begin{equation}\label{eq: assum 1}
    \mathbb{E}[\|\Delta^t\|^2] \leq (1 - \rho_1)\mathbb{E}[\|\Delta^{t-1}\|^2] + A\mathbb{E}[\sigma_{t-1}^2] + \eta_{t-1}^2 B D^2 + C,
\end{equation}
\begin{equation}\label{eq: assum 2}
\mathbb{E}[\sigma_t^2] \leq (1 - \rho_2)\mathbb{E}[\sigma_{t-1}^2] + \eta_{t-1}^2 E D^2.
\end{equation}
\end{remark}

\subsubsection{Quasar Convex Case}\label{apppendix_A_bsfw: quasar_convex}
In the $\rho$-quasar-convex setting, first we prove descent Lemma~\ref{lem: stoch_quasar_convex_descent}. This will help us prove both a fixed-horizon convergent rate in Theorem~\ref{thm:stoch_fixed_horizon_quasar_convex} using a step decay that is a modified version of the step size that is offered in \cite{nazykov2024stochastic}, and an any-time convergent rate in Theorem~\ref{thm:stoch_quasar_convex_param_ag} using an alternative step decay.

\begin{lemma}[Descent under Smoothness in Stochastic Setting for Quasar-Convex Functions]
    \label{lem: stoch_quasar_convex_descent}
    Let $f$ be an objective function that satisfies $L$-smoothness Assumption~\ref{assum: L_smooth} and $\rho$-quasar-convexity Assumption~\ref{assum: quasar-convexity}. Consider the sequence $\{x^t\}_{t=0}^{+\infty}$ generated by Algorithm~\ref{alg:BSFW}. Then for any $\alpha > 0$, we have
    \begin{equation*}
        \mathbb{E}[F_{t+1}] \leq (1 - \rho \eta_t) \mathbb{E}[F_t] + \frac{\alpha}{L} \mathbb{E}[\|\Delta^t\|^2] + \eta_t^2LD^2 \left( \frac{1}{\alpha}  + \frac{1}{2} \right).
    \end{equation*}
\end{lemma}
\begin{proof}
    From Assumption~\ref{assum: L_smooth}, we have
    \begin{equation*}
        f(x^{t+1}) \leq f(x^t) + \langle \nabla f(x^t), x^{t+1} - x^t \rangle + \frac{L}{2} \|x^{t+1} - x^t\|^2.
    \end{equation*}
    We distinguish two cases based on the value of $\gamma_t$, following the analysis of Algorithm~\ref{alg:BSFW}. \\
    \noindent\textbf{Case I}: $\gamma_t<1$. Then we have $x^{t+1} = x^t + \gamma_t d^t$. By using Lemma~\ref{lem:alignment},
    \begin{align*}
        f(x^{t+1}) &\leq f(x^t) + \gamma_t \langle \nabla f(x^t), d^t \rangle + \frac{L}{2} \gamma_t^2 \|d^t\|^2 \\
        &\leq f(x^t) + \gamma_t \langle \nabla f(x^t) - m^t, d^t \rangle + \gamma_t \langle m^t, d^t \rangle + \frac{L}{2} \gamma_t^2 \|d^t\|^2 \\
        &\leq f(x^t) + \gamma_t \langle \nabla f(x^t) - m^t, d^t \rangle \\
        &~~~~~~+ \eta_t \left( \frac{\|s^t - x^t\|}{\|d^t\|} \right) \left( \frac{\|d^t\|}{\|s^t - x^t\|} \right)  \langle m^t, s^t - x^t \rangle + \frac{L}{2} \eta_t^2 \frac{\|s^t - x^t\|^2}{\|d^t\|^2} \|d^t\|^2 \\ 
        &\leq f(x^t) + \gamma_t \langle \nabla f(x^t) - m^t, d^t \rangle + \eta_t \langle m^t, s^t - x^t \rangle + \frac{L}{2} \eta_t^2 \|s^t - x^t\|^2.
    \end{align*}
    Since $s^t \in \lmo(m^{t})$, we get
    \begin{align*}
        f(x^{t+1}) &\leq f(x^t) + \gamma_t \langle \nabla f(x^t) - m^t, d^t \rangle + \eta_t \langle m^t, x^{\star} - x^t \rangle + \frac{L}{2} \eta_t^2 \|s^t - x^t\|^2 \\
        &\leq f(x^t) + \gamma_t \langle \nabla f(x^t) - m^t, d^t \rangle + \eta_t \langle m^t - \nabla f(x^t), x^{\star} - x^t \rangle \\
        &~~~~~+ \eta_t \langle \nabla f(x^t), x^{\star} - x^t \rangle + \frac{L}{2} \eta_t^2 \|s^t - x^t\|^2.
    \end{align*}
    By using the Young's inequality on both $\langle \nabla f(x^t) - m^t, \gamma_t d^t \rangle $ and $ \langle m^t - \nabla f(x^t), \eta_t( x^{\star} - x^t ) \rangle$ with a parameter $\beta = \frac{\alpha}{L}$ for an arbitrary $\alpha > 0$, we have
    \begin{align*}
        \langle \nabla f(x^t) - m^t, \gamma_t d^t \rangle &\leq \frac{\alpha}{2L} \|\nabla f(x^t) - m^t \|^2 + \gamma_t^2 \frac{L}{2\alpha} \|d^t\|^2, \\
        \langle m^t - \nabla f(x^t), \eta_t( x^{\star} - x^t ) \rangle &\leq \frac{\alpha}{2L} \|\nabla f(x^t) - m^t\|^2 + \eta_t^2 \frac{L}{2\alpha} \|x^{\star} - x^t\|^2.
    \end{align*}
    Thus by using these inequalities, we get
    \begin{align*}
        f(x^{t+1}) &\leq f(x^t) + \frac{\alpha}{2L} \|\nabla f(x^t) - m^t \|^2 + \gamma_t^2 \frac{L}{2\alpha} \|d^t\|^2 + \frac{\alpha}{2L} \|\nabla f(x^t) - m^t \|^2 \\
        &~~~~~+ \eta_t^2 \frac{L}{2\alpha} \|x^{\star} - x^t\|^2 + \frac{L}{2} \eta_t^2 \|s^t - x^t\|^2 + \eta_t \langle \nabla f(x^t), x^{\star} - x^t \rangle \\
        &\leq f(x^t) + \frac{\alpha}{L} \|\Delta^t\|^2 + \eta_t^2 \frac{L}{2\alpha} \|s^t - x^t\|^2 + \eta_t^2 \frac{L}{2\alpha} \|x^{\star} - x^t\|^2 + \eta_t \langle \nabla f(x^t), x^{\star} - x^t \rangle + \frac{L}{2} \eta_t^2 D^2 \\
        &\leq f(x^t) +  \frac{\alpha}{L} \|\Delta^t\|^2 + \eta_t^2 \frac{L}{\alpha} D^2 + \eta_t  \langle \nabla f(x^t), x^{\star} - x^t \rangle + \frac{L}{2} \eta_t^2 D^2 \\
        &\leq f(x^t) + \frac{\alpha}{L} \|\Delta^t\|^2 + \eta_t^2 L D^2 \left( \frac{1}{\alpha} + \frac{1}{2} \right) + \eta_t \langle \nabla f(x^t), x^{\star} - x^t \rangle. 
    \end{align*}
    By using $\rho$-quasar convexity of $f$, and subtracting $f(x^{\star})$ on both sides, we get
    \begin{align*}
        F_{t+1} &\leq F_t + \frac{\alpha}{L} \|\Delta^t\|^2 + \eta_t^2 L D^2 \left( \frac{1}{\alpha} + \frac{1}{2} \right) - \rho \eta_t (f(x^t) - f(x^{\star})) \\
        & \leq (1 - \rho \eta_t ) F_t + \frac{\alpha}{L} \|\Delta^t\|^2 + \eta_t^2 L D^2 \left( \frac{1}{\alpha} + \frac{1}{2} \right).
    \end{align*}
    Taking expectation on both sides gives us
    \begin{equation*}
        \mathbb{E}[F_{t+1}] \leq (1-\rho\eta_t)\mathbb{E}[F_t] + \frac{\alpha}{L} \mathbb{E}[\|\Delta^t\|^2] + \eta_t^2 L D^2\left( \frac{1}{\alpha} + \frac{1}{2} \right).
    \end{equation*}
    \noindent\textbf{Case II}: $\gamma_t=1$. Then we have $x^{t+1} = x^{t} + \eta_t (s^t - x^t)$. Hence we get,
    \begin{align*}
        f(x^{t+1}) &\leq f(x^t) + \eta_t \langle \nabla f(x^t), s^t - x^t \rangle + \frac{L}{2} \eta_t^2 \|s^t - x^t\|^2 \\
        &\leq f(x^t) + \eta_t \langle \nabla f(x^t) - m^t, s^t - x^t \rangle + \eta_t \langle m^t, s^t - x^t \rangle +  \frac{L}{2} \eta_t^2 \|s^t - x^t\|^2 \\
        &\leq f(x^t) + \eta_t \langle \nabla f(x^t) - m^t, s^t - x^t \rangle + \eta_t \langle m^t, x^{\star} - x^t \rangle +  \frac{L}{2} \eta_t^2 \|s^t - x^t\|^2 \\
        &\leq f(x^t) + \eta_t \langle \nabla f(x^t) - m^t, s^t - x^t \rangle + \eta_t \langle m^t - \nabla f(x^t), x^{\star} - x^t \rangle \\
        &~~~~~+ \eta_t \langle \nabla f(x^t), x^{\star} - x^t \rangle + \frac{L}{2} \eta_t^2 \|s^t - x^t\|^2 \\
        &\leq f(x^t) + \eta_t \langle \nabla f(x^t) - m^t, s^t - x^{\star} \rangle + \eta_t \langle \nabla f(x^t), x^{\star} - x^t \rangle + \frac{L}{2} \eta_t^2 \|s^t - x^t\|^2.
    \end{align*}
    Using Young's inequality on $ \langle \nabla f(x^t) - m^t, \eta_t(s^t - x^{\star}) \rangle $ with $\beta = \frac{2 \alpha}{L}$ with an arbitrary $\alpha > 0$, we have the inequality
    \begin{equation*}
        \langle \nabla f(x^t) - m^t, \eta_t (s^t - x^{\star}) \rangle \leq \frac{\alpha}{L} \| \nabla f(x^t) - m^t \|^2 + \frac{L}{4 \alpha} \eta_t^2 \|s^t - x^{\star}\|^2.
    \end{equation*}
    By using this inequality, and $\rho$-quasar-convexity of $f$, we get
    \begin{align*}
        f(x^{t+1}) &\leq f(x^t) + \frac{\alpha}{L} \|\Delta^t\|^2 + \eta_t^2 \frac{L}{4\alpha} \|s^t - x^{\star}\|^2 - \rho \eta_t (f(x^t) - f(x^{\star})) + \frac{L}{2} \eta_t^2 D^2 \\
        &\leq f(x^t) + \frac{\alpha}{L} \|\Delta^t\|^2  - \rho \eta_t F_t + \eta_t^2 L D^2 \left( \frac{1}{4\alpha} + \frac{1}{2} \right).
    \end{align*}
    By subtracting $f(x^{\star})$ and taking expectation on both sides, we get
    \begin{align*}
        \mathbb{E}[F_{t+1}] &\leq \mathbb{E}[F_t] + \frac{\alpha}{L} \mathbb{E}[\|\Delta^t\|^2]  - \rho \eta_t \mathbb{E}[F_t] + \eta_t^2 L D^2 \left( \frac{1}{4\alpha} + \frac{1}{2} \right) \\
        &\leq \mathbb{E}[F_t] + \frac{\alpha}{L} \mathbb{E}[\|\Delta^t\|^2]  - \rho \eta_t \mathbb{E}[F_t] + \eta_t^2 L D^2 \left( \frac{1}{\alpha} + \frac{1}{2} \right) \\
        &\leq (1 - \rho \eta_t ) \mathbb{E}[F_t] + \frac{\alpha}{L} \mathbb{E}[\|\Delta^t\|^2] +  \eta_t^2 L D^2 \left( \frac{1}{\alpha} + \frac{1}{2} \right),
    \end{align*}
    which gives us the desired expression in both cases
\end{proof}

\begin{theorem}[Formal Statement of Theorem~\ref{thm:stoch_fixed_horizon_quasar_convex}]
\label{proof thm: stoch_fixed_horizon_quasar_convex}
Let $f$ be an objective function that satisfies $L$-smoothness Assumption~\ref{assum: L_smooth} and $\rho$-quasar-convexity Assumption~\ref{assum: quasar-convexity}. Let $T \in \mathbb{N}$, and $\{x^t\}_{t=0}^{T}$ be a sequence generated by Algorithm~\ref{alg:BSFW} where the stochastic estimator $m^t$ and auxiliary sequence $\{\sigma_t\}$ satisfy Assumption~\ref{assum: est} with parameters $\rho_1, \rho_2 \in ]0,1]$ and constants $A, B, C, E \geq 0$.
Let the step decay be chosen piecewise as
$$
\eta_t = 
\begin{cases}
\dfrac{1}{\rho d}, &  T \leq d, \\[0.5em]
\dfrac{1}{\rho d}, &  T > d \text{ and } t < t_0, \\[0.5em]
\dfrac{2}{\rho (2d + t - t_0)}, &  T > d \text{ and } t \geq t_0,
\end{cases}
$$
with $d := \dfrac{2}{\min\{\rho_1,\rho_2\}}$ and $t_0 := \lfloor T/2 \rfloor$. Then, the functional gap satisfies
{\small\begin{equation*}
    \mathbb{E}[F_T]  \leq 
 e\cdot\exp\left(-\frac{T}{2d}\right)\mathbb{E}[r_0] 
+ \frac{16 D^2 L}{\rho^2(T+d)}
+ \sqrt{\frac{32 D^2}{\rho^2(T+d)} \left(\frac{64 D^2 B}{\rho^2(T+d) \rho_1}+\frac{128 D^2 A E}{\rho^2(T+d) \rho_1 \rho_2}+\frac{2 C T}{\rho_1}\right)  },
\end{equation*}}
where $r_t$ is a Lyapunov function defined by
\begin{equation*}
    \forall t: \quad r_t = F_t + \frac{2\alpha^{\star}}{\rho_1 L} \|\Delta^t\|^2 + \frac{4\alpha^{\star}A}{\rho_1 \rho_2  L} \sigma_t^2,
\end{equation*}
with 
\begin{equation*}
    \alpha^{\star} = \sqrt{\left( \frac{32 D^2 L}{\rho^2(T+d)} \right) / \left( \frac{64 D^2 B}{\rho^2(T+d) \rho_1 L}+\frac{128 D^2 A E}{\rho^2(T+d) \rho_1 \rho_2 L }+\frac{2 C T}{\rho_1 L} \right)}.
\end{equation*}
\end{theorem}
\begin{proof}
Define a Lyapunov function $r_t = F_t + M_1 \|\Delta^t\|^2 + M_2 \sigma_t^2$, with $M_1, M_2 > 0$ as arbitrary constants. We analyze the full expectation of the Lyapunov function $r_{t+1}$
\begin{equation}\label{eq: E_Lyapunov}
R_{t+1}:=\mathbb{E}[r_{t+1}] 
= \mathbb{E}[F_{t+1}] 
+ M_1 \mathbb{E}[\|\Delta^{t+1}\|^2] 
+ M_2 \mathbb{E}[\sigma_{t+1}^2].
\end{equation}
By Lemma~\ref{lem: stoch_quasar_convex_descent}, by setting an arbitrary $\alpha > 0$, we have
\begin{equation}\label{eq:obj}
\mathbb{E}[F_{t+1}] 
\leq (1 - \rho \eta_t)\mathbb{E}[F_t] 
+ \frac{\alpha}{L} \mathbb{E}[\|\Delta^t\|^2] 
+ \eta_t^2 L D^2 \left( \frac{1}{\alpha} + \frac{1}{2} \right).
\end{equation}
By using \eqref{eq:obj}, Remark~\ref{rem: full grad bounds}, and substituting the values into \eqref{eq: E_Lyapunov}
\begin{equation}\label{eq:lyap_prelim}
\begin{aligned}
\mathbb{E}[r_{t+1}] 
&\leq (1 - \rho \eta_t)\mathbb{E}[F_t]  + \left( \frac{\alpha}{L} + M_1(1 - \rho_1) \right) \mathbb{E}[\|\Delta^t\|^2] \\
&\quad + \left( M_1 A + M_2(1 - \rho_2) \right) \mathbb{E}[\sigma_t^2] + \left( L \left( \frac{1}{\alpha} + \frac{1}{2} \right) + M_1 B + M_2 E \right) \eta_t^2 D^2 
+ M_1 C.
\end{aligned}
\end{equation}
We now choose $M_1 = \frac{2\alpha}{\rho_1 L}$ and $M_2 = \frac{2 M_1 A}{\rho_2}$. With this choice, we can verify
\begin{equation}\label{eq: identity1}
\frac{\alpha}{L} + M_1(1 - \rho_1) 
= \frac{\alpha}{L} + \frac{2\alpha}{\rho_1 L}(1 - \rho_1)
= \frac{\alpha}{L} \left( 1 + \frac{2(1 - \rho_1)}{\rho_1} \right)
= \frac{\alpha}{L} \cdot \frac{2 - \rho_1}{\rho_1}
= M_1\left(1 - \frac{\rho_1}{2}\right),
\end{equation}
and similarly,
{\small\begin{equation}\label{eq: identity2}
M_1 A + M_2(1 - \rho_2) 
= M_1 A + \frac{2M_1 A}{\rho_2}(1 - \rho_2)
= M_1 A \left( 1 + \frac{2(1 - \rho_2)}{\rho_2} \right)
= M_1 A \cdot \frac{2 - \rho_2}{\rho_2}
= M_2\left(1 - \frac{\rho_2}{2}\right).
\end{equation}}
Define
$$a := D^2 \left( L \left( \frac{1}{\alpha} + \frac{1}{2} \right) + M_1 B + M_2 E \right) \quad \text{and} \quad b := M_1 C.$$
Substituting \eqref{eq: identity1}--\eqref{eq: identity2} into \eqref{eq:lyap_prelim}
\begin{equation}\label{eq: dec1}
\mathbb{E}[r_{t+1}] \leq (1 - \rho \eta_t) \mathbb{E}[F_t] 
+ M_1\left(1 - \frac{\rho_1}{2}\right) \mathbb{E}[\|\Delta^t\|^2] 
+ M_2\left(1 - \frac{\rho_2}{2}\right) \mathbb{E}[\sigma_t^2] + a \eta_t^2 + b.
\end{equation}
Recall $R_t=\mathbb{E}[r_t] = \mathbb{E}[F_t] + M_1 \mathbb{E}[\|\Delta^t\|^2] + M_2 \mathbb{E}[\sigma_t^2]$. To obtain a contraction in terms of $R_t$, we show that the coefficients in \eqref{eq: dec1} satisfy
$$M_1\left(1 - \frac{\rho_1}{2}\right) \leq M_1(1 - \rho \eta_t) \quad \text{and} \quad M_2\left(1 - \frac{\rho_2}{2}\right) \leq M_2(1 - \rho \eta_t).$$
By construction of $d = \frac{2}{\min\{\rho_1,\rho_2\}}$, the step decay satisfies $\rho \eta_t \leq \frac{\min\{\rho_1, \rho_2\}}{2}$, so we have
$$\rho \eta_t  \leq \frac{\rho_1}{2} \quad \text{and} \quad \rho \eta_t  \leq \frac{\rho_2}{2},$$
which implies that for all $\rho \eta_t \in ]0,1]$ that
$$1 - \frac{\rho_1}{2} \leq 1 - \rho \eta_t \quad \text{and} \quad 1 - \frac{\rho_2}{2} \leq 1 - \rho \eta_t.$$
Since $M_1, M_2 > 0$, we obtain:
$$M_1\left(1 - \frac{\rho_1}{2}\right) \leq M_1 (1 - \rho \eta_t) \quad \text{and} \quad M_2\left(1 - \frac{\rho_2}{2}\right) \leq M_2 (1 - \rho \eta_t).$$
Applying these to \eqref{eq: dec1}
\begin{align*}
\mathbb{E}[r_{t+1}] 
&\leq (1 - \rho \eta_t) \mathbb{E}[F_t] + M_1(1 - \rho \eta_t) \mathbb{E}[\|\Delta^t\|^2]+ M_2(1 - \rho \eta_t) \mathbb{E}[\sigma_t^2] + a \eta_t^2 + b\\
&= (1 - \rho \eta_t) \left(\mathbb{E}[F_t] + M_1 \mathbb{E}[\|\Delta^t\|^2] + M_2 \mathbb{E}[\sigma_t^2]\right) + a \eta_t^2 + b\\
&= (1 - \rho \eta_t) \mathbb{E}[r_t] + a \eta_t^2 + b.
\end{align*}
Thus
\begin{equation}\label{eq:lyap_recursion}
    R_{t+1} \leq (1 - \rho \eta_t)R_{t} + a \eta_t^2 + b, \quad \text{where } R_{t}=\mathbb{E}[r_t].
\end{equation}

Now we claim that for constant step decay $\eta_t = \eta$:
\begin{equation}\label{eq:unrolled}
R_T \leq (1 - \rho \eta)^T R_0 + a\eta^2 \sum_{k=0}^{T-1} (1 - \rho \eta)^k + bT.
\end{equation}

The claim is proved by mathematical induction on $T$.

\noindent\textbf{Base Case} ($T=1$): From the recursion with $t=0$:
$$R_1 \leq (1 - \rho \eta) R_0 + a\eta^2 + b = (1 - \rho \eta)^1 R_0 + a\eta^2 \sum_{k=0}^{0} (1 - \rho \eta)^k + b \cdot 1.$$

\noindent\textbf{Inductive Hypothesis}: Assume for $T = n$:
$$R_n \leq (1 - \rho \eta)^n R_0 + a\eta^2 \sum_{k=0}^{n-1} (1 - \rho \eta)^k + bn.$$

\noindent\textbf{Inductive Step}: For $T = n+1$, starting from $R_{n+1} \leq (1-\rho\eta)R_n + a\eta^2 + b$:
\begin{align*}
R_{n+1} &\leq (1 - \rho \eta) \left[ (1 - \rho \eta)^n R_0 + a\eta^2 \sum_{k=0}^{n-1} (1 - \rho \eta)^k + bn \right] + a\eta^2 + b \\
&= (1 - \rho \eta)^{n+1} R_0 + a\eta^2 \left[ (1 - \rho \eta) \sum_{k=0}^{n-1} (1 - \rho \eta)^k + 1 \right] + b[n(1 - \rho \eta) + 1].
\end{align*}

Since $(1 - \rho \eta) \sum_{k=0}^{n-1} (1 - \rho \eta)^k + 1 = \sum_{k=0}^{n} (1 - \rho \eta)^k$ and $\rho\eta \in ]0,1]$ gives $n(1 - \rho \eta) + 1 \leq n + 1$
$$R_{n+1} \leq (1 - \rho \eta)^{n+1} R_0 + a\eta^2 \sum_{k=0}^{n} (1 - \rho \eta)^k + b(n+1).$$
By induction:
$$R_T \leq (1 - \rho \eta)^T R_0 + a\eta^2 \sum_{k=0}^{T-1} (1 - \rho \eta)^k + bT.$$
We now analyze the convergence behavior under different step decay values. The following cases arise.

\noindent\textbf{Case 1}: $T \leq d$ with $\eta_t=\eta=\frac{1}{\rho d}$. 
Using the geometric series $\sum_{k=0}^{T-1} (1 - \frac{1}{d})^k \leq d$ and $(1-x)^n \leq \exp(-nx)$
$$
R_T \leq \left(1 - \frac{1}{d}\right)^T R_0 + \frac{a}{\rho^2 d} + bT \leq \exp\left(-\frac{T}{d}\right) R_0 + \frac{a}{\rho^2 d} + bT.
$$
Since $T \leq d$ gives $\frac{1}{d} 
\leq \frac{2}{T+d}$ and $\exp(-T/d) \leq \exp(-T/(2d))$
\begin{align*}
    R_T \leq \exp \left(-\frac{T}{2d}\right) R_0 + \frac{2a}{\rho^2(T+d)} + bT \leq e \cdot \exp\left(-\frac{T}{2d}\right) R_0 + \frac{32a}{\rho^2(T+d)} + bT.
\end{align*}
\noindent\textbf{Case 2}: $T > d$, $t < t_0$, with $\eta_t=\frac{1}{\rho d}$ and $t_0 := \lfloor T/2 \rfloor$. 
Similarly
$$
R_{t_0} \leq \exp\left(-\frac{t_0}{d}\right) R_0 + \frac{a}{\rho^2 d} + bt_0.
$$
\noindent\textbf{Case 3}: $T > d$, $t \geq t_0$, with $\eta_t=\frac{2}{\rho(2d + t - t_0)}$. Let $k = t - t_0$, $K := T - t_0 = \lceil T/2 \rceil$, $\bar{R}_k := R_{t_0+k}$, $\bar{\eta}_k := \frac{2}{\rho(2d+k)}$. The recursion becomes
$$
\bar{R}_{k+1} \leq \frac{2d+k-2}{2d+k}\bar{R}_k + \frac{4a}{\rho^2(2d+k)^2} + b.
$$
Multiplying by $(2d+k)^2$ and using $(2d+k)(2d+k-2) \leq (2d+k-1)^2$
$$
(2d+k)^2\bar{R}_{k+1} \leq (2d+k-1)^2\bar{R}_k + \frac{4a}{\rho^2} + b(2d+k)^2.
$$
Summing for $k=0,\cdots,K-1$
$$
(2d+K-1)^2\bar{R}_K \leq (2d)^2\bar{R}_0 + K\frac{4a}{\rho^2} + bK(2d+K-1)^2.
$$
For $K = \lceil T/2 \rceil \geq T/2$, we have $2d+K-1 \geq (T+d)/2$, giving
$$
\frac{4d^2}{(2d+K-1)^2} \leq \frac{16d^2}{(T+d)^2}, \quad \frac{K}{(2d+K-1)^2} \leq \frac{4}{T+d}.
$$
Since $t_0 = \lfloor T/2 \rfloor \geq T/2 - 1$
$$
\exp\left(-\frac{t_0}{d}\right) \leq e\cdot\exp\left(-\frac{T}{2d}\right).
$$
Combining with Case 2 and using $t_0 + K = T$, $\frac{d}{T+d} \leq 1$
$$
R_T \leq e\cdot\exp\left(-\frac{T}{2d}\right)R_0 + \frac{32a}{\rho^2(T+d)} + bT.
$$
Hence by either Case I or Case II and Case III, we have the resulting expression
\begin{equation*}
    R_T \leq e\cdot\exp\left(-\frac{T}{2d}\right)R_0 + \frac{32a}{\rho^2(T+d)} + bT.
\end{equation*}
Substituting the expressions for $a$ and $b$ 
$$a = D^2 \left( L \left( \frac{1}{\alpha} + \frac{1}{2} \right) + \frac{2\alpha B}{\rho_1 L} + \frac{4\alpha AE}{\rho_1\rho_2L} \right) \quad \text{and} \quad b = \frac{2\alpha C}{\rho_1 L},$$
{\small\begin{equation*}
\begin{aligned}
R_T &\leq e\cdot\exp\left(-\frac{T}{2d}\right)R_0 + \frac{32 D^2 \left( L \left( \frac{1}{\alpha} + \frac{1}{2} \right) + \frac{2\alpha B}{\rho_1 L} + \frac{4\alpha AE}{\rho_1\rho_2L} \right)}{\rho^2(T+d)} + \frac{2\alpha C T}{\rho_1 L}\\&
= e\cdot\exp\left(-\frac{T}{2d}\right)R_0 
+ \frac{16 D^2 L}{\rho^2(T+d)} +\alpha \left[\frac{64 D^2 B}{\rho^2(T+d) \rho_1 L}+\frac{128 D^2 A E}{\rho^2(T+d) \rho_1 \rho_2 L}+\frac{2 C T}{\rho_1 L}\right] +\frac{32 D^2 L}{\rho^2(T+d) \alpha}.
\end{aligned}    
\end{equation*}}
Since the inequality holds for all $\alpha>0$, we take the infimum over 
$\alpha$ to get the tightest bound.
The $\alpha$-dependent terms in the bound are of the form
$$g(\alpha):=\alpha \left[\frac{64 D^2 B}{\rho^2(T+d) \rho_1 L}+\frac{128 D^2 A E}{\rho^2(T+d) \rho_1 \rho_2 L}+\frac{2 C T}{\rho_1 L}\right] +\frac{32 D^2 L}{\rho^2(T+d) \alpha}.$$
This has the form $g(\alpha) = \alpha u+\frac{v}{\alpha}$, which is minimized at $\alpha^\star = \sqrt{v/u}$ with minimum value $g(\alpha^\star) = 2\sqrt{uv}$.
So, 
$$
g(\alpha^\star) = 2\sqrt{\frac{32 D^2}{\rho^2(T+d)} \left(\frac{64 D^2 B}{\rho^2(T+d) \rho_1}+\frac{128 D^2 A E}{\rho^2(T+d) \rho_1 \rho_2}+\frac{2 C T}{\rho_1}\right)  },
$$
which completes the proof since $\mathbb{E}[F_t] \leq R_T$. Consequentially when $C = 0$, it yields a $\mathcal{O}(1/T)$ convergence rate.
\end{proof}
\begin{theorem}[Formal Statement of Theorem~\ref{thm:stoch_quasar_convex_param_ag}]\label{proof thm:stoch_quasar_convex_param_ag}
Let $f$ be a function that satisfies $L$-smoothness Assumption~\ref{assum: L_smooth} and $\rho$-quasar-convexity Assumption~\ref{assum: quasar-convexity}. Suppose the stochastic estimator $m^t$ and auxiliary sequence $\{\sigma_t\}$ satisfy Assumption~\ref{assum: est} with parameters $\rho_1, \rho_2 \in ]0,1]$ and constants $A, B, C, E \geq 0$. Let $\{x^t\}_{t=0}^{+\infty}$ be a sequence generated by Algorithm~\ref{alg:BSFW} by choosing the step decay
\begin{equation*}
\eta_t = \frac{2}{\rho(t + \nu)}, \quad \text{where }\quad  \nu = \max\left\{2, \frac{4}{\min\{\rho_1, \rho_2\}}\right\}.
\end{equation*}
Then, the expected functional-value gap satisfies 
\begin{equation*}
\mathbb{E}[F_t] \leq\sqrt{\frac{16 D^2}{\rho^2(t + \nu)} \left( \frac{32 D^2 B}{\rho^2(t + \nu)\rho_1} + \frac{64 D^2 AE}{\rho^2(t + \nu)\rho_1 \rho_2} + \frac{2CT}{\rho_1} \right)} +\frac{4\nu^2 \mathbb{E}[r_0]}{(t + \nu)^2} + \frac{8 D^2 L}{\rho^2(t + \nu)}.
\end{equation*}
where $r_t$ is a Lyapunov function defined by
\begin{equation*}
\forall t, \quad r_t = F_t + \tfrac{2\alpha^{\star}}{\rho_1 L} \|\Delta^t\|^2 + \tfrac{4\alpha^{\star} A}{\rho_1\rho_2 L} \sigma_t^2,
\end{equation*}
with 
\begin{equation*}
    \alpha^{\star} = \sqrt{\left( \frac{16 D^2 L}{\rho^2(T + \nu)} \right) / \left( \frac{32 D^2 B}{\rho^2(T + \nu) \rho_1 L} + \frac{64 D^2 AE}{\rho^2(T + \nu) \rho_1 \rho_2 L} + \frac{2CT}{\rho_1 L} \right)}.
\end{equation*}
If $C = 0$, the last term in the square root vanishes and we obtain a $\mathcal{O}(1/t)$ rate.
\end{theorem}
\begin{proof}
Define a Lyapunov function $r_t = F_t + M_1 \|\Delta^t\|^2 + M_2 \sigma_t^2$, with $M_1, M_2 > 0$ as arbitrary constants. \\ \\
\textbf{Recursion.}
Following the same derivation as in Theorem~\ref{proof thm: stoch_fixed_horizon_quasar_convex}, we analyze the full expectation of the Lyapunov function. Define $R_t := \mathbb{E}[r_t]$. From Lemma~\ref{lem: stoch_quasar_convex_descent} and Assumption~\ref{assum: est}, we have established that for any $\alpha > 0$, i.e., \eqref{eq:lyap_recursion}:
\begin{equation}\label{eq:lyap_recursion_open}
R_{t+1} \leq (1 - \rho \eta_t) R_t + a \eta_t^2 + b,
\end{equation}
where
$$
a = D^2 \left( L \left( \frac{1}{\alpha} + \frac{1}{2} \right) + M_1 B + M_2 E \right), \quad b = M_1 C,
$$
with $M_1 = \frac{2\alpha}{\rho_1 L}$ and $M_2 = \frac{2M_1 A}{\rho_2}$.
\\

The derivation of \eqref{eq:lyap_recursion_open} requires that
$
\rho \eta_t \leq \frac{\min\{\rho_1, \rho_2\}}{2}
$
for all $t \geq 0$. With the proposed step size $\eta_t = \frac{2}{\rho(t + \nu)}$, we have
$
\rho \eta_t = \frac{2}{t + \nu}.
$
Since $t \geq 0$, the maximum value of $\rho \eta_t$ occurs at $t = 0$:
$$
\max_{t \geq 0} \rho \eta_t = \frac{2}{\nu}.
$$
By the definition $\nu \geq \frac{4}{\min\{\rho_1, \rho_2\}}$, we have
$$
\frac{2}{\nu} \leq \frac{2}{\frac{4}{\min\{\rho_1, \rho_2\}}} = \frac{\min\{\rho_1, \rho_2\}}{2}.
$$
Therefore, for all $t \geq 0$:
$$
\rho \eta_t = \frac{2}{t + \nu} \leq \frac{2}{\nu} \leq \frac{\min\{\rho_1, \rho_2\}}{2},
$$
which implies both $\rho \eta_t \leq \frac{\rho_1}{2}$ and $\rho \eta_t \leq \frac{\rho_2}{2}$, validating the contraction conditions used in deriving \eqref{eq:lyap_recursion_open}.
Additionally, the condition $\nu \geq 2$ ensures that $t + \nu - 2 \geq 0$ for all $t \geq 0$, which will be needed in the subsequent analysis.
\\
\textbf{Computing the Recursion Coefficients.}
With the step size $\eta_t = \frac{2}{\rho(t + \nu)}$, we compute:
\begin{align}
\rho \eta_t &= \frac{2}{t + \nu}, \label{eq:rho_eta}\\[0.5em]
1 - \rho \eta_t &= 1 - \frac{2}{t + \nu} = \frac{t + \nu - 2}{t + \nu}, \label{eq:one_minus_rho_eta}\\[0.5em]
\eta_t^2 &= \frac{4}{\rho^2(t + \nu)^2}. \label{eq:eta_squared}
\end{align}
Substituting \eqref{eq:one_minus_rho_eta} and \eqref{eq:eta_squared} into \eqref{eq:lyap_recursion_open}:
\begin{equation}\label{eq:recursion_explicit}
R_{t+1} \leq \frac{t + \nu - 2}{t + \nu} R_t + \frac{4a}{\rho^2(t + \nu)^2} + b.
\end{equation}
\\
\textbf{Weighted Lyapunov Analysis.}
Multiply both sides of \eqref{eq:recursion_explicit} by $(t + \nu)^2$:
\begin{equation}\label{eq:weighted_recursion}
(t + \nu)^2 R_{t+1} \leq (t + \nu)(t + \nu - 2) R_t + \frac{4a}{\rho^2} + b(t + \nu)^2.
\end{equation}
We now use the algebraic identity:
$$
(t + \nu)(t + \nu - 2) = (t + \nu - 1)^2 - 1.
$$
To verify this, expand both sides:
\begin{align*}
(t + \nu)(t + \nu - 2) &= (t + \nu)^2 - 2(t + \nu), \\
(t + \nu - 1)^2 - 1 &= (t + \nu)^2 - 2(t + \nu) + 1 - 1 = (t + \nu)^2 - 2(t + \nu).
\end{align*}
Since $(t + \nu - 1)^2 - 1 \leq (t + \nu - 1)^2$, inequality \eqref{eq:weighted_recursion} becomes:
\begin{equation}\label{eq:weighted_recursion_simplified}
(t + \nu)^2 R_{t+1} \leq (t + \nu - 1)^2 R_t + \frac{4a}{\rho^2} + b(t + \nu)^2.
\end{equation}
\\
\textbf{Defining the Weighted Sequence.}
Define the weighted sequence:
$$
W_t := (t + \nu - 1)^2 R_t.
$$
Note that with this setting we have for all $0 \leq t \leq T$,
$$
W_0 = (\nu - 1)^2 R_0,~~W_T = (T + \nu - 1)^2 R_T,~~\text{and}~~
W_{t+1} = (t + \nu)^2 R_{t+1}. 
$$
Inequality \eqref{eq:weighted_recursion_simplified} can be rewritten as:
\begin{equation}\label{eq:W_recursion}
W_{t+1} \leq W_t + \frac{4a}{\rho^2} + b(t + \nu)^2.
\end{equation}
\\
Summing inequality \eqref{eq:W_recursion} from $t = 0$ to $t = T - 1$:
\begin{align}
W_T - W_0 &\leq \sum_{t=0}^{T-1} \left( \frac{4a}{\rho^2} + b(t + \nu)^2 \right) = \frac{4aT}{\rho^2} + b \sum_{t=0}^{T-1} (t + \nu)^2. \label{eq:telescoped}
\end{align}
Rearranging:
\begin{equation}\label{eq:W_T_bound}
W_T \leq W_0 + \frac{4aT}{\rho^2} + b \sum_{t=0}^{T-1} (t + \nu)^2.
\end{equation}
\\
\textbf{Bounding the Summation.}
We bound the sum $\sum_{t=0}^{T-1} (t + \nu)^2$ as follows. Since $(t + \nu)$ is increasing in $t$, the largest term in the sum is $(T - 1 + \nu)^2$. Thus:
\begin{equation}\label{eq:sum_bound}
\sum_{t=0}^{T-1} (t + \nu)^2 \leq T \cdot (T - 1 + \nu)^2 = T(T + \nu - 1)^2.
\end{equation}
Substituting \eqref{eq:sum_bound} into \eqref{eq:W_T_bound}:
\begin{equation}\label{eq:W_T_bound_final}
W_T \leq W_0 + \frac{4aT}{\rho^2} + bT(T + \nu - 1)^2.
\end{equation}
\\
\textbf{Converting Back to $R_T$.}
Substituting $W_T = (T + \nu - 1)^2 R_T$ and $W_0 = (\nu - 1)^2 R_0$:
\begin{equation}\label{eq:R_T_intermediate}
(T + \nu - 1)^2 R_T \leq (\nu - 1)^2 R_0 + \frac{4aT}{\rho^2} + bT(T + \nu - 1)^2.
\end{equation}
Dividing both sides by $(T + \nu - 1)^2$:
\begin{equation}\label{eq:R_T_bound_raw}
R_T \leq \frac{(\nu - 1)^2}{(T + \nu - 1)^2} R_0 + \frac{4aT}{\rho^2(T + \nu - 1)^2} + bT.
\end{equation}
\\
\textbf{Simplifying Using Bounds on $\nu$.}
Since $\nu \geq 2$, we have $\nu - 1 \geq 1$. This implies:
$$
T + \nu - 1 \geq \frac{T}{2} + \nu - 1 =\frac{T + \nu}{2} + \frac{\nu - 2}{2} \geq \frac{T + \nu}{2}\geq 0,
$$
where we used $\nu \geq 2$. Therefore:
\begin{equation}\label{eq:T_nu_bound}
(T + \nu - 1)^2 \geq \frac{(T + \nu)^2}{4}.
\end{equation}
Using \eqref{eq:T_nu_bound}, we obtain the following bounds:

\textit{Bound 1:} For the initial condition term:
$$
\frac{(\nu - 1)^2}{(T + \nu - 1)^2} \leq \frac{4(\nu - 1)^2}{(T + \nu)^2} \leq \frac{4\nu^2}{(T + \nu)^2},
$$
where the last inequality uses $\nu - 1 \leq \nu$.

\textit{Bound 2:} For the $a$-dependent term:
$$
\frac{T}{(T + \nu - 1)^2} \leq \frac{4T}{(T + \nu)^2}.
$$
We further simplify using $T \leq T + \nu$:
$$
\frac{4T}{(T + \nu)^2} = \frac{4}{T + \nu} \cdot \frac{T}{T + \nu} \leq \frac{4}{T + \nu}.
$$
\textbf{Final Bound in Terms of $\alpha$.}
Substituting the bounds from Step 9 into \eqref{eq:R_T_bound_raw}:
\begin{equation}\label{eq:final_bound_alpha}
R_T \leq \frac{4\nu^2 R_0}{(T + \nu)^2} + \frac{16a}{\rho^2(T + \nu)} + bT.
\end{equation}
\textbf{Explicit Form of the Bound.}
Substituting the definitions of $a$, $b$, $M_1$, and $M_2$:
\begin{align*}
a &= D^2 \left( L \left( \frac{1}{\alpha} + \frac{1}{2} \right) + M_1 B + M_2 E \right) = D^2 \left( L \left( \frac{1}{\alpha} + \frac{1}{2} \right) + \frac{2\alpha B}{\rho_1 L} + \frac{4\alpha AE}{\rho_1 \rho_2 L} \right),\\
b &= M_1 C = \frac{2\alpha C}{\rho_1 L}.
\end{align*}
Therefore, \eqref{eq:final_bound_alpha} becomes:
{\small\begin{align}
R_T &\leq \frac{4\nu^2 R_0}{(T + \nu)^2} + \frac{16 D^2}{\rho^2(T + \nu)} \left( L \left( \frac{1}{\alpha} + \frac{1}{2} \right) + \frac{2\alpha B}{\rho_1 L} + \frac{4\alpha AE}{\rho_1 \rho_2 L} \right) + \frac{2\alpha C T}{\rho_1 L} \nonumber\\
&= \frac{4\nu^2 R_0}{(T + \nu)^2} + \frac{8 D^2 L}{\rho^2(T + \nu)} + \frac{16 D^2 L}{\rho^2(T + \nu) \alpha} + \alpha \left( \frac{32 D^2 B}{\rho^2(T + \nu) \rho_1 L} + \frac{64 D^2 AE}{\rho^2(T + \nu) \rho_1 \rho_2 L} + \frac{2CT}{\rho_1 L} \right). \label{eq:explicit_alpha}
\end{align}}
\textbf{Optimizing Over $\alpha$.}
The $\alpha$-dependent terms in \eqref{eq:explicit_alpha} have the form $g(\alpha) := \frac{v}{\alpha} + \alpha u$, where
\begin{align*}
v &:= \frac{16 D^2 L}{\rho^2(T + \nu)}, \\
u &:= \frac{32 D^2 B}{\rho^2(T + \nu) \rho_1 L} + \frac{64 D^2 AE}{\rho^2(T + \nu) \rho_1 \rho_2 L} + \frac{2CT}{\rho_1 L}.
\end{align*}
Now, $g(\alpha)$ is minimized at $\alpha^\star = \sqrt{v/u}$ with minimum value $g(\alpha^\star) = 2\sqrt{uv}$. Thus, substituting $g(\alpha^\star) = 2\sqrt{uv}$ into \eqref{eq:explicit_alpha}:
\begin{align}
R_T &\leq \frac{4\nu^2 R_0}{(T + \nu)^2} + \frac{8 D^2 L}{\rho^2(T + \nu)} + 2\sqrt{uv} \nonumber\\
&= \frac{4\nu^2 R_0}{(T + \nu)^2} + \frac{8 D^2 L}{\rho^2(T + \nu)} + 2\sqrt{\frac{16 D^2}{\rho^2(T + \nu)} \left( \frac{32 D^2 B}{\rho^2(T + \nu) \rho_1} + \frac{64 D^2 AE}{\rho^2(T + \nu) \rho_1 \rho_2} + \frac{2CT}{\rho_1} \right)}. \label{eq:final_optimized}
\end{align}
Simplifying the constant under the square root:
\begin{equation}\label{eq:final_result}
\mathbb{E}[r_T] \leq \frac{4\nu^2 \mathbb{E}[r_0]}{(T + \nu)^2} + \frac{8 D^2 L}{\rho^2(T + \nu)} + \sqrt{\frac{64 D^2}{\rho^2(T + \nu)} \left( \frac{32 D^2 B}{\rho^2(T + \nu) \rho_1} + \frac{64 D^2 AE}{\rho^2(T + \nu) \rho_1 \rho_2} + \frac{2CT}{\rho_1} \right)}.
\end{equation}
This completes the proof since $\mathbb{E}[F_T] \leq \mathbb{E}[r_T]$.
\end{proof}

\subsubsection{Nonconvex Case}\label{appendix_A_bsfw: non_convex}
Under the nonconvex setting, we first build descent Lemma~\ref{lem: stoch_ncv_descent} which is used in the convergence analysis that follows. We offer convergence rates using a fixed horizon step decay in Theorem~\ref{proof thm:stoch_fixed_horizon_nonconvex} and an any-time convergent step decay in Theorem~\ref{thm: stoch_anytime_non_convex}.
\begin{lemma}[Descent under Smoothness in Stochastic Setting for nonconvex Functions]\label{lem: stoch_ncv_descent} 
Let $f$ be an objective function that satisfies $L$-smoothness Assumption~\ref{assum: L_smooth}. Consider the sequence $\{x^t\}_{t=0}^{+\infty}$ generated by Algorithm~\ref{alg:BSFW}. Then for any $\alpha > 0$, the following inequality holds
\begin{equation*}
\eta_t \mathbb{E}\left[\max_{u\in\mathcal{C}} \langle \nabla f(x^t),  x^t-u \rangle \right]
\leq 
 \mathbb{E}\left[F_t \right] - \mathbb{E}\left[F_{t+1} \right]  + \frac{\alpha}{L}  \mathbb{E}\left[\|\Delta^t\|^2 \right] + \eta_t^2 LD^2  \left(\frac{1}{\alpha}+\frac{1}{2}\right).
\end{equation*}
\end{lemma}
\begin{proof}
    From Assumption~\ref{assum: L_smooth}, we have
    \begin{equation*}
        f(x^{t+1}) \leq f(x^t) + \langle \nabla f(x^t), x^{t+1} - x^t \rangle + \frac{L}{2} \|x^{t+1} - x^t\|^2.
    \end{equation*}
    We distinguish two cases based on the value of $\gamma_t$, following the analysis of Algorithm~\ref{alg:BSFW}. \\
    \noindent\textbf{Case I}: $\gamma_t<1$. Then we have $x^{t+1} = x^t + \gamma_t d^t$. By using Lemma~\ref{lem:alignment},
    {\small\begin{align*}
        f(x^{t+1}) &\leq f(x^t) + \gamma_t \langle \nabla f(x^t), d^t \rangle + \frac{L}{2} \gamma_t^2 \|d^t\|^2 \\
        &\leq f(x^t) + \gamma_t \langle \nabla f(x^t) - m^t, d^t \rangle + \gamma_t \langle m^t, d^t \rangle + \frac{L}{2} \gamma_t^2 \|d^t\|^2 \\
        &\leq f(x^t) + \gamma_t \langle \nabla f(x^t) - m^t, d^t \rangle + \eta_t \left( \frac{\|s^t - x^t\|}{\|d^t\|} \right) \left( \frac{\|d^t\|}{\|s^t - x^t\|} \right)  \langle m^t, s^t - x^t \rangle + \frac{L}{2} \eta_t^2 \frac{\|s^t - x^t\|^2}{\|d^t\|^2} \|d^t\|^2 \\ 
        &\leq f(x^t) + \gamma_t \langle \nabla f(x^t) - m^t, d^t \rangle + \eta_t \langle m^t, s^t - x^t \rangle + \frac{L}{2} \eta_t^2 \|s^t - x^t\|^2.
    \end{align*}}
    Since $s^t \in \mathrm{lmo}(m^t)$, $\forall u \in \mathcal{C}$, we get
    {\footnotesize\begin{align*}
        f(x^{t+1}) &\leq f(x^t) + \gamma_t \langle \nabla f(x^t) - m^t, d^t \rangle + \eta_t \langle m^t, u - x^t \rangle + \frac{L}{2} \eta_t^2 \|s^t - x^t\|^2 \\
        &\leq f(x^t) + \gamma_t \langle \nabla f(x^t) - m^t, d^t \rangle + \eta_t \langle m^t - \nabla f(x^t), u - x^t \rangle + \langle \nabla f(x^t), u - x^t \rangle + \frac{L}{2} \eta_t^2 \|s^t - x^t\|^2.
    \end{align*}}
    By using the Young's inequality on both $ \langle \nabla f(x^t) - m^t, \gamma_t d^t \rangle $ and $\langle m^t - \nabla f(x^t), \eta_t( u - x^t ) \rangle$ with a parameter $\beta = \frac{\alpha}{L}$ for an arbitrary $\alpha > 0$, we have
    \begin{align*}
        \langle \nabla f(x^t) - m^t, \gamma_t d^t \rangle &\leq \frac{\alpha}{2L} \|\nabla f(x^t) - m^t \|^2 + \gamma_t^2 \frac{\alpha}{2L} \|d^t\|^2 \\
        \langle m^t - \nabla f(x^t), \eta_t( u - x^t ) \rangle &\leq \frac{\alpha}{2L} \|m^t - \nabla f(x^t)\|^2 + \eta_t^2 \frac{L}{2\alpha} \|u - x^t\|^2
    \end{align*}
    By using these inequalities, we get
    \begin{align*}
        f(x^{t+1}) &\leq f(x^t) + \frac{\alpha}{2L} \|\nabla f(x^t) - m^t \|^2 + \gamma_t^2 \frac{L}{2\alpha} \|d^t\|^2  \\
        &~~~~+ \frac{\alpha}{2L} \|\nabla f(x^t) - m^t \|^2 + \eta_t^2 \frac{L}{2\alpha} \|u - x^t\|^2 + \frac{L}{2} \eta_t^2 \|s^t - x^t\|^2 + \eta_t \langle \nabla f(x^t), u - x^t \rangle \\
        &\leq f(x^t) + \frac{\alpha}{L} \|\Delta^t\|^2 + \eta_t^2 \frac{L}{2\alpha} \|s^t - x^t\|^2 + \eta_t^2 \frac{L}{2\alpha} \|u - x^t\|^2 + \eta_t \langle \nabla f(x^t), u - x^t \rangle + \frac{L}{2} \eta_t^2 D^2 \\
        &\leq f(x^t) +  \frac{\alpha}{L} \|\Delta^t\|^2 + \eta_t^2 \frac{L}{\alpha} D^2 + \eta_t  \langle \nabla f(x^t), u - x^t \rangle + \frac{L}{2} \eta_t^2 D^2 \\
        &\leq f(x^t) + \frac{\alpha}{L} \|\Delta^t\|^2 + \eta_t^2 L D^2 \left( \frac{1}{\alpha} + \frac{1}{2} \right) + \eta_t \langle \nabla f(x^t), u - x^t \rangle. 
    \end{align*}
    Subtracting $f^{\star}$ and taking expectation on both sides gives us
    \begin{align*}
        &\mathbb{E}[F_{t+1}] \leq \mathbb{E}[F_t] + \frac{\alpha}{L} \mathbb{E}[ \|\Delta^t\|^2] + \eta_t^2 L D^2 \left( \frac{1}{\alpha} + \frac{1}{2} \right) + \eta_t \mathbb{E}[\langle \nabla f(x^t), x^t - u \rangle] \\
        &\implies \eta_t \mathbb{E}[\langle \nabla f(x^t), x^t - u \rangle] \leq \mathbb{E}[F_t] - \mathbb{E}[F_{t+1}]+ \frac{\alpha}{L} \mathbb{E}[ \|\Delta^t\|^2] + \eta_t^2 L D^2 \left( \frac{1}{\alpha} + \frac{1}{2} \right) \\
        &\implies \eta_t \mathbb{E}\left[ \max_{u \in \mathcal{C}}\langle \nabla f(x^t), x^t - u \rangle \right] \leq \mathbb{E}[F_t] - \mathbb{E}[F_{t+1}]+ \frac{\alpha}{L} \mathbb{E}[ \|\Delta^t\|^2] + \eta_t^2 L D^2 \left( \frac{1}{\alpha} + \frac{1}{2} \right).
    \end{align*}
    \noindent\textbf{Case II}: $\gamma_t = 1$. Then we have $x^{t+1} = x^{t} + \eta_t (s^t - x^t)$. Hence $\forall u \in \mathcal{C}$,  we get
    \begin{align*}
        f(x^{t+1}) &\leq f(x^t) + \eta_t \langle \nabla f(x^t), s^t - x^t \rangle + \frac{L}{2} \eta_t^2 \|s^t - x^t\|^2 \\
        &\leq f(x^t) + \eta_t \langle \nabla f(x^t) - m^t, s^t - x^t \rangle + \eta_t \langle m^t, s^t - x^t \rangle +  \frac{L}{2} \eta_t^2 \|s^t - x^t\|^2 \\
        &\leq f(x^t) + \eta_t \langle \nabla f(x^t) - m^t, s^t - x^t \rangle + \eta_t \langle m^t, u - x^t \rangle +  \frac{L}{2} \eta_t^2 \|s^t - x^t\|^2 \\
        &\leq f(x^t) + \eta_t \langle \nabla f(x^t) - m^t, u - x^t \rangle + \eta_t \langle m^t - \nabla f(x^t), u - x^t \rangle \\
        &~~~~~+ \eta_t \langle \nabla f(x^t), u - x^t \rangle + \frac{L}{2} \eta_t^2 \|s^t - x^t\|^2 \\
        &\leq f(x^t) + \eta_t \langle \nabla f(x^t) - m^t, s^t - u \rangle + \eta_t \langle \nabla f(x^t), u - x^t \rangle + \frac{L}{2} \eta_t^2 \|s^t - x^t\|^2.
    \end{align*}
    Using Young's inequality on $ \langle \nabla f(x^t) - m^t, \eta_t (s^t - u) \rangle $ with $\beta = \frac{2 \alpha}{L}$ with an arbitrary $\alpha > 0$, we have
    \begin{equation*}
        \langle \nabla f(x^t) - m^t, \eta_t(s^t - u) \rangle \leq \frac{\alpha}{L} \| \nabla f(x^t) - m^t \|^2 + \frac{L}{4\alpha} \eta_t^2 \|s^t - u\|^2.
    \end{equation*}
    By using this inequality,
    \begin{align*}
        f(x^{t+1}) &\leq f(x^t) + \frac{\alpha}{L} \|\Delta^t\|^2 + \eta_t^2 \frac{L}{4\alpha} \|s^t - u\|^2 + \eta_t \langle \nabla f(x^t), u - x^t \rangle + \frac{L}{2} \eta_t^2 D^2 \\
        &\leq f(x^t) + \frac{\alpha}{L} \|\Delta^t\|^2  + \eta_t \langle \nabla f(x^t), u - x^t \rangle + \eta_t^2 L D^2 \left( \frac{1}{4\alpha} + \frac{1}{2} \right).
    \end{align*}
    By subtracting $f^{\star}$ and taking expectation on both sides, we get
    \begin{align*}
        \mathbb{E}[F_{t+1}] &\leq \mathbb{E}[F_t] + \frac{\alpha}{L} \mathbb{E}[\|\Delta^t\|^2]  + \eta_t \mathbb{E}[\langle \nabla f(x^t), u - x^t \rangle] + \eta_t^2 L D^2 \left( \frac{1}{4\alpha} + \frac{1}{2} \right) \\
        &\leq \mathbb{E}[F_t] + \frac{\alpha}{L} \mathbb{E}[\|\Delta^t\|^2]  + \eta_t \mathbb{E}[\langle \nabla f(x^t), u - x^t \rangle] + \eta_t^2 L D^2 \left( \frac{1}{\alpha} + \frac{1}{2} \right), 
        \end{align*}
        \begin{align*}
        &\implies\eta_t \mathbb{E}[\langle \nabla f(x^t), x^t - u \rangle] \leq \mathbb{E}[F_t] - \mathbb{E}[F_{t+1}] + \frac{\alpha}{L} \mathbb{E}[ \|\Delta^t\|^2] +  \eta_t^2 L D^2 \left( \frac{1}{\alpha} + \frac{1}{2} \right), \\
        &\implies \eta_t \mathbb{E}\left[ \max_{u \in \mathcal{C}}\langle \nabla f(x^t), x^t - u \rangle \right] \leq \mathbb{E}[F_t] - \mathbb{E}[F_{t+1}]+ \frac{\alpha}{L} \mathbb{E}[ \|\Delta^t\|^2] + \eta_t^2 L D^2 \left( \frac{1}{\alpha} + \frac{1}{2} \right).
    \end{align*}
    Which gives us the desired expression in both cases
\end{proof}

\begin{theorem}[Formal Statement of Theorem~\ref{thm:stoch_fixed_horizon_nonconvex}]\label{proof thm:stoch_fixed_horizon_nonconvex}
Let $f$ be an objective function that satisfies $L$-smoothness Assumption~\ref{assum: L_smooth}. Let $T \in \mathbb{N}$, and $\{x^t\}_{t=0}^{T}$ be a sequence generated by Algorithm~\ref{alg:BSFW} where the stochastic estimator $m^t$ and auxiliary sequence $\{\sigma_t\}$ satisfy Assumption~\ref{assum: est} with parameters $\rho_1, \rho_2 \in ]0,1]$ and constants $A, B, C, E \geq 0$. Let the step decay be chosen as a constant for all $t$, $\eta_t = \frac{1}{\sqrt{T}}$. Then we have
\begin{equation}
\begin{aligned}
\mathbb{E}\!\left[\,\min_{0 \le t \le T-1} \operatorname{Gap}(x^t)\,\right]   \leq  \frac{\mathbb{E}[r_0]}{\sqrt{T}} + \frac{D^2L}{2\sqrt{T}} + D \sqrt{ 
\frac{D^2}{\rho_1 T} \left( B + \frac{A E}{\rho_2} \right) 
+ \frac{C}{\rho_1} }.
\end{aligned}
\end{equation}
where $r_t$ is a Lyapunov function defined by 
\begin{equation*}
    \forall t: \quad r_t = F_t + \frac{\alpha^{\star}}{L \rho_1} \|\Delta^t\|^2 + \frac{\alpha^{\star}A}{L\rho_1 \rho_2} \sigma_t^2,
\end{equation*}
with
\begin{equation*}
    \alpha^{\star} = \sqrt{\left( \frac{D^2L}{\sqrt{T}} \right) / \left( \frac{D^2 B}{L \rho_1 \sqrt{T}} + \frac{D^2 A E}{L \rho_1 \rho_2 \sqrt{T}} + \frac{C \sqrt{T}}{L \rho_1}  \right)}.
\end{equation*}
If $C = 0$, the last term vanishes and we recover the standard $\mathcal{O}(1/\sqrt{T})$ rate.
\end{theorem}
\begin{proof}
Multiply inequality \eqref{eq: assum 1}  by $M_1$ (at iteration $t+1$), inequality \eqref{eq: assum 2} by $M_2$, and add the two to obtain
{\footnotesize\begin{equation}\label{eq:4.6 M1_M2}
\begin{aligned}
 M_1\left\|m^{t+1}-\nabla f\left(x^{t+1}\right)\right\|^2+M_2 \mathbb{E}\left[\sigma_{t+1}^2\right]
& \leq M_1\left(1-\rho_1\right)\left\|m^t-\nabla f\left(x^t\right)\right\|^2+M_1 A \mathbb{E}\left[\sigma_{t}^2\right] \\
&~~~~~ + M_1 B \eta_t^2 D^2+M_1 C +M_2\left(1-\rho_2\right) \mathbb{E}\left[\sigma_{t}^2\right]+M_2 E \eta_t^2 D^2\\
& = M_1\left(1-\rho_1\right)\left\|m^t-\nabla f\left(x^t\right)\right\|^2 \\
&~~~~~ + M_2 \left(1-\rho_2+\frac{M_1A}{M_2}\right) \mathbb{E}\left[\sigma_{t}^2\right] + \eta_t^2 D^2(M_1B+M_2E)+M_1C.
\end{aligned}
\end{equation}}
Define the Lyapunov function $r_t:=f(x^t) - f^{\star} + M_1 \|m^t - \nabla f(x^t)\|^2 + M_2 \sigma_t^2$. From its definition, the expected Lyapunov difference satisfies
\begin{equation}\label{eq:4.6 (5)}
\begin{aligned}
\mathbb{E}\left[r_t-r_{t+1}\right]&=  \mathbb{E}\left[f\left(x^t\right)-f\left(x^{t+1}\right)\right]   \\
&~~~~ + M_1 \mathbb{E}\left[\left\|m^t-\nabla f\left(x^t\right)\right\|^2-\left\|m^{t+1}-\nabla f\left(x^{t+1}\right)\right\|^2\right]  +M_2 \mathbb{E}\left[\sigma_t^2-\sigma_{t+1}^2\right].
\end{aligned}
\end{equation}
Rearranging \eqref{eq:4.6 (5)} gives
\begin{equation}\label{eq:4.6 (6)}
\begin{aligned}
\mathbb{E}\left[f\left(x^t\right)-f\left(x^{t+1}\right)\right]= & \mathbb{E}\left[r_t-r_{t+1}\right]  -M_1 \mathbb{E}\left[\left\|m^t-\nabla f\left(x^t\right)\right\|^2-\left\|m^{t+1}-\nabla f\left(x^{t+1}\right)\right\|^2\right] \\
&~~~ -M_2 \mathbb{E}\left[\sigma_t^2-\sigma_{t+1}^2\right].
\end{aligned}
\end{equation}
Now, using \eqref{eq: assum 1}-\eqref{eq: assum 2}, we obtain lower bounds on the error decrements
{\small\begin{align}
\mathbb{E}\!\left[\|m^t - \nabla f(x^t)\|^2 - \|m^{t+1} - \nabla f(x^{t+1})\|^2\right] 
&\ge \rho_1 \,\mathbb{E}\!\left[\|m^t - \nabla f(x^t)\|^2\right] - A\,\mathbb{E}[\sigma_t^2] - B \eta_t^2 D^2 - C, \label{eq:4.6 (7)}\\
\mathbb{E}\!\left[\sigma_t^2 - \sigma_{t+1}^2\right] 
&\ge \rho_2 \,\mathbb{E}[\sigma_t^2] - E \eta_t^2 D^2. \label{eq:4.6 (8)}
\end{align}}
Substituting~\eqref{eq:4.6 (7)}--\eqref{eq:4.6 (8)} into~\eqref{eq:4.6 (6)} yields
\begin{equation}\label{eq:4.6 (9)}
 \begin{aligned}
\mathbb{E}\!\left[f(x^t) - f(x^{t+1})\right]
\le&\; \mathbb{E}[r_t - r_{t+1}]
- \rho_1 M_1 \,\mathbb{E}\!\left[\|m^t - \nabla f(x^t)\|^2\right] 
- (\rho_2 M_2 - M_1 A)\,\mathbb{E}[\sigma_t^2]  \\
&+ \eta_t^2 D^2 (M_1 B + M_2 E) + M_1 C.
\end{aligned}  
\end{equation}
Plugging~\eqref{eq:4.6 (9)} into Lemma~\ref{lem: stoch_ncv_descent}, and grouping terms, we obtain
{\footnotesize\begin{equation}\label{eq:4.6 (10)}
\begin{aligned}
\eta_t \, \mathbb{E}\!\left[\operatorname{Gap}(x^t) \right]
\le&\; \mathbb{E}\!\left[
f(x^t) - f(x^\star) 
+ \left(1 - \rho_1 + \frac{\alpha}{M_1 L}\right) M_1 \|m^t - \nabla f(x^t)\|^2  + \left(1 - \rho_2 + \frac{M_1 A}{M_2} \right) M_2 \sigma_t^2
\right] \\
&- \mathbb{E}\!\left[
f(x^{t+1}) - f(x^\star) 
+ M_1 \|m^{t+1} - \nabla f(x^{t+1})\|^2 
+ M_2 \sigma_{t+1}^2
\right] \nonumber \\
&+ D^2 \eta_t^2 \left( \frac{L}{2} + \frac{L}{\alpha} + M_1 B + M_2 E \right) + M_1 C.
\end{aligned}
\end{equation}}
Now choose
$M_1 := \frac{\alpha}{L \rho_1}$ and $M_2 := \frac{M_1 A}{\rho_2} = \frac{\alpha A}{L \rho_1 \rho_2}$, for $\alpha>0$.
So that $1 - \rho_1 + \frac{\alpha}{M_1 L} = 1$ and $1 - \rho_2 + \frac{M_1 A}{M_2} = 1$. Then the first expectation in~\eqref{eq:4.6 (10)} equals $\mathbb{E}[r_t]$, and the second equals $\mathbb{E}[r_{t+1}]$, giving
\begin{equation}\label{eq:4.6 (11)}
\eta_t \, \mathbb{E}[\operatorname{Gap}(x^t)] 
\le \mathbb{E}[r_t - r_{t+1}] 
+ D^2 \eta_t^2 \left( \frac{L}{2} + \frac{L}{\alpha} + M_1 B + M_2 E \right) + M_1 C.
\end{equation}
Summing~\eqref{eq:4.6 (11)} over $t = 0, \cdots, T-1$ telescopes
\begin{equation}
\sum_{t=0}^{T-1} \eta_t \, \mathbb{E}[\operatorname{Gap}(x^t)] 
\le \mathbb{E}[r_0 - r_T] 
+ D^2 \sum_{t=0}^{T-1} \eta_t^2 \left( \frac{L}{2} + \frac{L}{\alpha} + M_1 B + M_2 E \right) + T M_1 C.
\end{equation}
Since $r_T \geq 0$, we have
\begin{equation}
\sum_{t=0}^{T-1} (\eta_t \, \mathbb{E}[\operatorname{Gap}(x^t)] )
\le \mathbb{E}[r_0] 
+ D^2\left( \frac{L}{2} + \frac{L}{\alpha} + M_1 B + M_2 E \right) \sum_{t=0}^{T-1} \eta_t^2  + T M_1 C.
\end{equation}
With constant step decay $\eta_t \equiv \eta = \frac{1}{\sqrt{T}}$ and Since $\min_{0 \le t \le T-1} \operatorname{Gap}(x^t) \le \operatorname{Gap}(x^t)$ for all $t$, taking expectations we have
\begin{equation}
   \mathbb{E}\!\left[\,\min_{0 \le t \le T-1} \operatorname{Gap}(x^t)\,\right] \sqrt{T} \leq  \mathbb{E}[r_0] 
+ D^2  \left( \frac{L}{2} + \frac{L}{\alpha} + M_1 B + M_2 E \right) + T M_1 C.
\end{equation}
where $M_1 = \frac{\alpha}{L \rho_1}$ and $M_2 = \frac{\alpha A}{L \rho_1 \rho_2}$ which means that 
\begin{equation}
\begin{aligned}
\mathbb{E}\!\left[\,\min_{0 \le t \le T-1} \operatorname{Gap}(x^t)\,\right]  & \leq  \frac{\mathbb{E}[r_0]}{\sqrt{T}} 
+ \frac{D^2  \left( \frac{L}{2} + \frac{L}{\alpha} + \frac{\alpha}{L \rho_1} B + \frac{\alpha A}{L \rho_1 \rho_2} E \right)}{\sqrt{T}} + \sqrt{T} \frac{\alpha}{L \rho_1} C\\
& 
= \frac{\mathbb{E}[r_0]}{\sqrt{T}} + \frac{D^2L}{2\sqrt{T}} + \frac{D^2 L}{\alpha \sqrt{T}} + \alpha \left[ \frac{D^2 B}{L \rho_1 \sqrt{T}} + \frac{D^2 A E}{L \rho_1 \rho_2 \sqrt{T}} + \frac{C \sqrt{T}}{L \rho_1} \right].
\end{aligned}
\end{equation}
Since the inequality holds for all $\alpha>0$, so take the infimum over 
$\alpha$ to get the tightest bound.
The $\alpha$-dependent terms in the bound are of the form
\begin{equation}
g(\alpha) = \frac{D^2 L}{\alpha \sqrt{T}} + \alpha \left[ \frac{D^2 B}{L \rho_1 \sqrt{T}} + \frac{D^2 A E}{L \rho_1 \rho_2 \sqrt{T}} + \frac{C \sqrt{T}}{L \rho_1} \right].
\end{equation}
This has the form $g(\alpha) = \frac{v}{\alpha} + u\alpha$, which is minimized at $\alpha^\star = \sqrt{v/u}$ with minimum value $g(\alpha^\star) = 2\sqrt{uv}$. where
\begin{align*}
    u: &= \frac{D^2 B}{L \rho_1 \sqrt{T}} + \frac{D^2 A E}{L \rho_1 \rho_2 \sqrt{T}} + \frac{C \sqrt{T}}{L \rho_1},  \\
    v: &= \frac{D^2 L}{\sqrt{T}}.
\end{align*}
Hence we have,
$$g(\alpha^{*}) =2\sqrt{\frac{D^{2}L}{\sqrt{T}} \left( \frac{D^{2}B}{L\rho_{1}\sqrt{T}} +\frac{D^{2}AE}{L\rho_{1}\rho_{2}\sqrt{T}} +\frac{C\sqrt{T}}{L\rho_{1}} \right)} =  D \sqrt{ 
\frac{D^2}{\rho_1 T} \left( B + \frac{A E}{\rho_2} \right) 
+ \frac{C}{\rho_1} },
$$
which completes the proof.
\end{proof}

\begin{theorem}[Any-time Convergence under nonconvexity and Stochastic Setting]\label{thm: stoch_anytime_non_convex}
Let $f$ be an objective function that satisfies $L$-smoothness Assumption~\ref{assum: L_smooth}. Let $\{x^t\}_{t=0}^{+\infty}$ be a sequence generated by Algorithm~\ref{alg:BSFW} where the stochastic estimator $m^t$ and auxiliary sequence $\{\sigma_t\}$ satisfy Assumption~\ref{assum: est} with parameters $\rho_1, \rho_2 \in ]0,1]$ and constants $A, B, C, E \geq 0$. Let the step decay be chosen as for all $t$,  $ \eta_t = \tfrac{1}{\sqrt{t + 1}} $. Then we have
\begin{align*}
    \mathbb{E}\left[\min_{0 \leq t \leq T - 1} \operatorname{Gap}(x^t)\right] \leq &\tfrac{\mathbb{E}[r_0]}{2(\sqrt{T + 1} - 1)}  + \tfrac{LD^2(1 + \ln(T))}{4(\sqrt{T + 1} - 1)} \\
    &~~~~
    + D\sqrt{\left( \tfrac{2BD^2(1 + \ln(T))^2}{\rho_1(\sqrt{T + 1} - 1)^2} + \tfrac{D^2 A E(1 + \ln(T))^2}{\rho_1 \rho_2 (\sqrt{T + 1} - 1)^2} + \tfrac{C (\sqrt{T + 1} + 1)(1 + \ln(T))}{\rho_1 (\sqrt{T + 1} - 1)} \right)}.
\end{align*}
If $C = 0$, the last term vanishes and we recover the standard $\mathcal{O}(\ln(t)/\sqrt{t})$ rate.
\end{theorem}

\begin{proof}
Multiply inequality \eqref{eq: assum 1}  by $M_1$ (at iteration $t+1$), inequality \eqref{eq: assum 2} by $M_2$, and add the two to obtain
{\small\begin{equation}\label{eq:A.13 M1_M2}
\begin{aligned}
 M_1\left\|m^{t+1}-\nabla f\left(x^{t+1}\right)\right\|^2+M_2 \mathbb{E}\left[\sigma_{t+1}^2\right]
& \leq M_1\left(1-\rho_1\right)\left\|m^t-\nabla f\left(x^t\right)\right\|^2+M_1 A \mathbb{E}\left[\sigma_{t}^2\right] \\
& \quad +M_1 B \eta_t^2 D^2+M_1 C +M_2\left(1-\rho_2\right) \mathbb{E}\left[\sigma_{t}^2\right]+M_2 E \eta_t^2 D^2\\
& = M_1\left(1-\rho_1\right)\left\|m^t-\nabla f\left(x^t\right)\right\|^2 + M_2 \left(1-\rho_2+\frac{M_1A}{M_2}\right) \mathbb{E}\left[\sigma_{t}^2\right]\\
& \quad + \eta_t^2 D^2(M_1B+M_2E)+M_1C.
\end{aligned}
\end{equation}}
From the definition of $r_t$, i.e.,$ r_t = f(x^t) - f^\star + M_1 \|m^t - \nabla f(x^t)\|^2 + M_2 \sigma_t^2$, the expected Lyapunov difference satisfies
\begin{equation}\label{eq:A.13 (5)}
\begin{aligned}
\mathbb{E}\left[r_t-r_{t+1}\right]= & \mathbb{E}\left[f\left(x^t\right)-f\left(x^{t+1}\right)\right] \\
& +M_1 \mathbb{E}\left[\left\|m^t-\nabla f\left(x^t\right)\right\|^2-\left\|m^{t+1}-\nabla f\left(x^{t+1}\right)\right\|^2\right] \\
& +M_2 \mathbb{E}\left[\sigma_t^2-\sigma_{t+1}^2\right].
\end{aligned}
\end{equation}
Rearranging \eqref{eq:A.13 (5)} gives
\begin{equation}\label{eq:A.13 (6)}
\begin{aligned}
\mathbb{E}\left[f\left(x^t\right)-f\left(x^{t+1}\right)\right]= & \mathbb{E}\left[r_t-r_{t+1}\right]  -M_1 \mathbb{E}\left[\left\|m^t-\nabla f\left(x^t\right)\right\|^2-\left\|m^{t+1}-\nabla f\left(x^{t+1}\right)\right\|^2\right] \\
&~~~ -M_2 \mathbb{E}\left[\sigma_t^2-\sigma_{t+1}^2\right].
\end{aligned}
\end{equation}
Now, using \eqref{eq: assum 1}-\eqref{eq: assum 2}, we obtain lower bounds on the error decrements
{\small\begin{align}
\mathbb{E}\!\left[\|m^t - \nabla f(x^t)\|^2 - \|m^{t+1} - \nabla f(x^{t+1})\|^2\right] 
&\ge \rho_1 \,\mathbb{E}\!\left[\|m^t - \nabla f(x^t)\|^2\right] - A\,\mathbb{E}[\sigma_t^2] - B \eta_t^2 D^2 - C, \label{eq:A.13 (7)}\\
\mathbb{E}\!\left[\sigma_t^2 - \sigma_{t+1}^2\right] 
&\ge \rho_2 \,\mathbb{E}[\sigma_t^2] - E \eta_t^2 D^2. \label{eq:A.13 (8)}
\end{align}}
Substituting~\eqref{eq:A.13 (7)}--\eqref{eq:A.13 (8)} into~\eqref{eq:A.13 (6)} yields
\begin{equation}\label{eq:A.13 (9)}
\begin{aligned}
\mathbb{E}\!\left[f(x^t) - f(x^{t+1})\right]
\le&\; \mathbb{E}[r_t - r_{t+1}]
- \rho_1 M_1 \,\mathbb{E}\!\left[\|m^t - \nabla f(x^t)\|^2\right] 
- (\rho_2 M_2 - M_1 A)\,\mathbb{E}[\sigma_t^2]  \\
&+ \eta_t^2 D^2 (M_1 B + M_2 E) + M_1 C.
\end{aligned}    
\end{equation}
Plugging~\eqref{eq:A.13 (9)} into Lemma~\ref{lem: stoch_ncv_descent}, and grouping terms, we obtain
{\small\begin{equation}\label{eq:A.13 (10)}
\begin{aligned}
\eta_t \, \mathbb{E}\!\left[\operatorname{Gap}(x^t) \right]
\le&\; \mathbb{E}\!\left[
f(x^t) - f(x^\star) 
+ \left(1 - \rho_1 + \frac{\alpha}{M_1 L}\right) M_1 \|m^t - \nabla f(x^t)\|^2  + \left(1 - \rho_2 + \frac{M_1 A}{M_2} \right) M_2 \sigma_t^2
\right] \\
&~~~- \mathbb{E}\!\left[
f(x^{t+1}) - f(x^\star) 
+ M_1 \|m^{t+1} - \nabla f(x^{t+1})\|^2 
+ M_2 \sigma_{t+1}^2
\right] \nonumber \\
&~~~+ D^2 \eta_t^2 \left( \frac{L}{2} + \frac{L}{\alpha} + M_1 B + M_2 E \right) + M_1 C.
\end{aligned}
\end{equation}}
Now choose
$M_1 := \frac{\alpha}{L \rho_1}$ and $M_2 := \frac{M_1 A}{\rho_2} = \frac{\alpha A}{L \rho_1 \rho_2}$, for $\alpha>0$.
So that $1 - \rho_1 + \frac{\alpha}{M_1 L} = 1$ and $1 - \rho_2 + \frac{M_1 A}{M_2} = 1$. Then the first expectation in~\eqref{eq:A.13 (10)} equals $\mathbb{E}[r_t]$, and the second equals $\mathbb{E}[r_{t+1}]$, giving
\begin{equation}\label{eq:A.13 (11)}
\eta_t \, \mathbb{E}[\operatorname{Gap}(x^t)] 
\le \mathbb{E}[r_t - r_{t+1}] 
+ D^2 \eta_t^2 \left( \frac{L}{2} + \frac{L}{\alpha} + M_1 B + M_2 E \right) + M_1 C.
\end{equation}
Summing~\eqref{eq:A.13 (11)} over $t = 0, \cdots, T-1$ telescopes
\begin{equation}
\sum_{t=0}^{T-1} \eta_t \, \mathbb{E}[\operatorname{Gap}(x^t)] 
\le \mathbb{E}[r_0 - r_T] 
+ D^2 \sum_{t=0}^{T-1} \eta_t^2 \left( \frac{L}{2} + \frac{L}{\alpha} + M_1 B + M_2 E \right) + T M_1 C.
\end{equation}
Since $r_T \geq 0$, we have
\begin{equation}
\sum_{t=0}^{T-1} (\eta_t \, \mathbb{E}[\operatorname{Gap}(x^t)] )
\le \mathbb{E}[r_0] 
+ D^2\left( \frac{L}{2} + \frac{L}{\alpha} + M_1 B + M_2 E \right) \sum_{t=0}^{T-1} \eta_t^2  + T M_1 C.
\end{equation}
\begin{equation}
\begin{aligned}
\mathbb{E}\left[\min_{0 \leq t \leq T - 1} \operatorname{Gap}(x^t)\right] \left(\sum_{t=0}^{T-1} \eta_t \right) &\leq \sum_{t=0}^{T-1} (\eta_t \, \mathbb{E}[\operatorname{Gap}(x^t)] )
\\ &\leq \mathbb{E}[r_0] 
+ D^2\left( \frac{L}{2} + \frac{L}{\alpha} + M_1 B + M_2 E \right) \sum_{t=0}^{T-1} \eta_t^2  + T M_1 C. 
\end{aligned}
\end{equation}
\noindent
Using the step decay $\eta_t = \tfrac{1}{\sqrt{t + 1}}$, we have by the integration test,
\begin{equation*}
    \sum_{t=0}^{T-1} \tfrac{1}{\sqrt{t + 1}} \geq 2(\sqrt{T+1} - 1) \quad \text{and} \quad \sum_{t=0}^{T-1} (\tfrac{1}{\sqrt{t + 1}})^2 \leq 1 + \ln(T).
\end{equation*}
Using these results in the expression, we get
{\small\begin{align*}
    &\mathbb{E}\left[\min_{0 \leq t \leq T - 1} \operatorname{Gap}(x^t)\right] (2\sqrt{T + 1} - 1) \leq \mathbb{E}[r_0] 
+ D^2\left( \frac{L}{2} + \frac{L}{\alpha} + M_1 B + M_2 E \right) (1 + \ln(T))  + T M_1 C, \\
 &\implies\mathbb{E}\left[\min_{0 \leq t \leq T - 1} \operatorname{Gap}(x^t)\right]\leq \tfrac{\mathbb{E}[r_0]}{2(\sqrt{T + 1} - 1)} + \tfrac{D^2(1 + \ln(T))}{2(\sqrt{T + 1} - 1)}\left( \frac{L}{2} + \frac{L}{\alpha} + M_1 B + M_2 E \right) + \tfrac{T M_1 C}{2(\sqrt{T + 1} - 1)}.
\end{align*}}
Using the constants $M_1 = \tfrac{\alpha}{L\rho_1}$ and $M_2 = \tfrac{\alpha A}{L \rho_1 \rho_2}$, we get
\begin{equation*}
     \mathbb{E}\left[\min_{0 \leq t \leq T - 1} \operatorname{Gap}(x^t)\right] \leq \tfrac{\mathbb{E}[r_0]}{2(\sqrt{T + 1} - 1)} + \tfrac{D^2(1 + \ln(T))}{2(\sqrt{T + 1} - 1)}\left( \tfrac{L}{2} + \tfrac{L}{\alpha} + \tfrac{\alpha B}{L\rho_1} + \tfrac{\alpha A E}{L \rho_1 \rho_2} \right) + \tfrac{T M_1 C}{2(\sqrt{T + 1} - 1)}.
\end{equation*}
Multiplying $\sqrt{T + 1} + 1$ in the numerator and denominator for the term containing $C$, we get
{\small\begin{align*}
    &\mathbb{E}\left[\min_{0 \leq t \leq T - 1} \operatorname{Gap}(x^t)\right] \leq \tfrac{\mathbb{E}[r_0]}{2(\sqrt{T + 1} - 1)} + \tfrac{LD^2(1 + \ln(T))}{4(\sqrt{T + 1} - 1)} + \tfrac{LD^2(1 + \ln(T))}{2\alpha(\sqrt{T + 1} - 1)} \\
    &~~~~+ \tfrac{BD^2 \alpha(1 + \ln(T))}{L\rho_1(\sqrt{T + 1} - 1)}+ \tfrac{\alpha D^2 A E(1 + \ln(T))}{2L \rho_1 \rho_2 (\sqrt{T + 1} - 1)} + \tfrac{\alpha C (\sqrt{T + 1} + 1)}{2 L \rho_1} \\
     &\leq \tfrac{\mathbb{E}[r_0]}{2(\sqrt{T + 1} - 1)}  + \tfrac{LD^2(1 + \ln(T))}{4(\sqrt{T + 1} - 1)} + \tfrac{1}{\alpha}\left(\tfrac{LD^2(1 + \ln(T))}{2(\sqrt{T + 1} - 1)}\right) + \alpha \left( \tfrac{BD^2(1 + \ln(T))}{L\rho_1(\sqrt{T + 1} - 1)} + \tfrac{D^2 A E (1 + \ln(T)))}{2L \rho_1 \rho_2 (\sqrt{T + 1} - 1)} + \tfrac{C (\sqrt{T + 1} + 1)}{2 L \rho_1} \right).
\end{align*}}
This has the form $g(\alpha) = u\alpha + \tfrac{v}{\alpha}$, with
\begin{equation*}
    u = \left( \tfrac{BD^2(1 + \ln(T))}{L\rho_1(\sqrt{T + 1} - 1)} + \tfrac{D^2 A E(1 + \ln(T))}{2L \rho_1 \rho_2 (\sqrt{T + 1} - 1)} + \tfrac{C (\sqrt{T + 1} + 1)}{2 L \rho_1} \right), \quad v = \left(\tfrac{LD^2(1 + \ln(T))}{2(\sqrt{T + 1} - 1)}\right).
\end{equation*}
Thus it reaches its minimum at $\alpha^* = \sqrt{\tfrac{v}{u}}$, with $g(\alpha^*) = 2\sqrt{uv}$. Substituting these values in the equation gives us
\begin{align*}
    g(\alpha^*) &= 2\sqrt{\left(\tfrac{LD^2(1 + \ln(T))}{2(\sqrt{T + 1} - 1)}\right) \left( \tfrac{BD^2(1 + \ln(T))}{L\rho_1(\sqrt{T + 1} - 1)} + \tfrac{D^2 A E(1 + \ln(T))}{2L \rho_1 \rho_2 (\sqrt{T + 1} - 1)} + \tfrac{C (\sqrt{T + 1} + 1)}{2 L \rho_1} \right)} \\
    &= D\sqrt{\left(\tfrac{2(1 + \ln(T))}{(\sqrt{T + 1} - 1)}\right) \left( \tfrac{BD^2(1 + \ln(T))}{\rho_1(\sqrt{T + 1} - 1)} + \tfrac{D^2 A E(1 + \ln(T))}{2 \rho_1 \rho_2 (\sqrt{T + 1} - 1)} + \tfrac{C (\sqrt{T + 1} + 1)}{2 \rho_1} \right)} \\
    &= D\sqrt{\left( \tfrac{2BD^2(1 + \ln(T))^2}{\rho_1(\sqrt{T + 1} - 1)^2} + \tfrac{D^2 A E(1 + \ln(T))^2}{\rho_1 \rho_2 (\sqrt{T + 1} - 1)^2} + \tfrac{C (\sqrt{T + 1} + 1)(1 + \ln(T))}{\rho_1 (\sqrt{T + 1} - 1)} \right)}. 
\end{align*}
Substituting this back into the main expression gives us 
\begin{align*}
    \mathbb{E}\left[\min_{0 \leq t \leq T - 1} \operatorname{Gap}(x^t)\right] &\leq \tfrac{\mathbb{E}[r_0]}{2(\sqrt{T + 1} - 1)}  + \tfrac{LD^2(1 + \ln(T))}{4(\sqrt{T + 1} - 1)} \\
    &~~~+ D\sqrt{\left( \tfrac{2BD^2(1 + \ln(T))^2}{\rho_1(\sqrt{T + 1} - 1)^2} + \tfrac{D^2 A E(1 + \ln(T))^2}{\rho_1 \rho_2 (\sqrt{T + 1} - 1)^2} + \tfrac{C (\sqrt{T + 1} + 1)(1 + \ln(T))}{\rho_1 (\sqrt{T + 1} - 1)} \right)}.
\end{align*}
Hence from the expression, when $C = 0$, we recover the rate $\mathcal{O}(\ln(t)/\sqrt{t})$.
\end{proof}

\section{Estimators}
\label{appendix: estimators}
The estimators used in the experiments section~\ref{sec: experiments} are taken from \cite{nazykov2024stochastic}, \cite{pmlr-v119-negiar20a} and \cite{JMLR:v21:18-764}. We use a set of stochastic and coordinate methods as gradient estimators $m^t$ of the gradient $\nabla f(x^t)$ at each iteration $t$. In Algorithm~\ref{alg:BSFW}, note the usage of $m^{\mathrm{init}}$, whose purpose is only to assign $x^0$, while that of $\{\Phi_t\}_{t=0}^{T-1}$ which is to denote the gradient estimator used. For these estimators, the proofs for the parameters satisfying Assumption~\ref{assum: est} are provided in the subsections \ref{subsec: SAG} to \ref{subsec: Heavy Ball}. A summary of the parameters is provided in Table~\ref{table: estimator_params}.

\begin{table}[ht]
\centering
\small
\caption{Summary of parameter values for various gradient estimators satisfying Assumption~\ref{assum: est}.}
\label{table: estimator_params}
\begin{tabular}{lcccccc}
\toprule
\textbf{Gradient Estimator} & $\rho_1$ & $\rho_2$ & $A$ & $B$ & $C$ & $E$ \\ \midrule
\hyperref[subsec: SAG]{SAG} {\cite{pmlr-v119-negiar20a}} & $\tfrac{b_s}{2m}$ & $1$ & $0$ & $(1 - \tfrac{b_s}{m} )(1 + \tfrac{2m}{b_s} )L^2$ & $0$ & $0$ \\  \addlinespace[3pt]
\hyperref[subsec: L-SVRG]{L-SVRG} {\cite{pmlr-v117-kovalev20a}} & $1$ & $\tfrac{p}{2}$ & $\tfrac{L^2}{b_s} - \tfrac{pL^2}{2b_s}$ & $\tfrac{8L^2}{pb_s}$ & $0$ & $\tfrac{8}{p}$ \\ \addlinespace[3pt]
\hyperref[subsec: SAGA]{SAGA}  {\cite{reddi2016stochastic}} & $1$ & $\tfrac{b_s}{2m}$ & $ \tfrac{1}{b_s} + \tfrac{1}{2m}$ & $\tfrac{L^2}{b_s m}(1 + \tfrac{2m}{b_s} )$ & $0$ & $\tfrac{2m}{b_s}L^2$ \\  \addlinespace[3pt]
\hyperref[subsec: SEGA]{SEGA} {\cite{NEURIPS2018_fc2c7c47}} & $1$ & $\tfrac{1}{2n}$ & $n$ & $n^2L^2$ & $0$ & $3L^2n$ \\  \addlinespace[3pt]
\hyperref[subsec: JAGUAR]{JAGUAR} {\cite{nazykov2024stochastic}} & $\tfrac{1}{2n}$ & $1$ & $0$ & $3nL^2$ & $0$ & $0$ \\  \addlinespace[3pt]
\hyperref[subsec: ZOJA]{ZOJA}  {\cite{nazykov2024stochastic}} & $\tfrac{1}{4n}$ & $1$ & $0$ & $3nL^2$ & {$2nL^2\tau^2$} & 0 \\ \addlinespace[3pt]
\hyperref[subsec: SARAH]{SARAH} {\cite{beznosikov2024sarahfrankwolfemethodsconstrained}} & $p$ & $1$ & $0$ & $\tfrac{1-p}{b_s}L^2$ & $0$ & $0$ \\ 
\hyperref[subsec: Heavy Ball]{Heavy Ball}  {\cite{JMLR:v21:18-764}} & $\tfrac{\tilde{\rho}_T}{2}$ & $1 - (\tfrac{T+7}{T+8} )^{\tfrac{4}{3}}$ & $1$ & $\tfrac{2L^2}{\tilde{\rho}_T}$ & $0$ & $0$ \\
\bottomrule
\end{tabular} 
\begin{tablenotes}
{\footnotesize\item 
The name of each estimator links to its description and proof in the appendix. For the stochastic estimators, the constant $b_s$ refers to the stochastic batch size (number of indices) sampled per-iteration. The parameter $p$ is the probability of computing a deterministic gradient used in the algorithms L-SVRG and SARAH, as explained in Appendices \ref{subsec: L-SVRG} and \ref{subsec: SARAH}. The parameter $\tau$ is the zeroth-order approximation parameter in ZOJA, explained in  Appendix \ref{subsec: ZOJA}. The parameter $\tilde{\rho_t}$ in Heavy Ball is the momentum, explained in Appendix \ref{subsec: Heavy Ball}. It is worth noting that, \textbf{for all the estimators except ZOJA, $C = 0$ which is significant for the convergence analysis}. For ZOJA, $C$ diminishes with respect to $\tau$. The values for the sequence $\{\sigma_t\}$ for each estimator are given under each estimator's corresponding section. 
The parameter $B$ of Heavy Ball is dependent on the horizon $T$, which affects the convergence analysis. Hence, additional analysis using alternate step decays for $\rho$-quasar-convex and nonconvex functions are explained in subsection~\ref{subsec: Heavy Ball}. }
\end{tablenotes}
\end{table}

\subsection{Stochastic Methods}\label{appendix: abt stoch}
The stochastic estimators randomly sample a batch of size $b_s$ from the dataset $\{a_1,a_2, \cdots, a_m\}$ at each iteration $t$ based on $\xi_t \sim \mathcal{P}$, and the corresponding gradient oracle output is denoted by $\nabla f(x^t, \xi_t)$. 

\subsection{Coordinate Methods}\label{appendix: abt coord}
The coordinate estimators randomly sample an index $j$ from the set $\{1,2,\cdots n \}$ based on $\xi_t \sim \mathcal{P}$. The estimators then use the partial derivative of the function $f$, denoted by $\nabla f_j(x^t)$ as the estimate of the gradient $\nabla f(x^t)$. We refer to $e_i$ as the $i$th standard basis vector in subsections \ref{subsec: SEGA}, \ref{subsec: JAGUAR} and \ref{subsec: ZOJA} describing the coordinate estimators SEGA, JAGUAR, and ZOJA respectively.

\subsection{Properties of Estimators}\label{appendix: est props}

In this section, we discuss the convergence properties of the estimators used in the experiments in section~\ref{sec: experiments}. We prove that if the estimators provide parameters that satisfy Assumption~\ref{assum: est}, we have fixed-horizon convergence rates in Theorems~\ref{proof thm: stoch_fixed_horizon_quasar_convex} and~\ref{proof thm:stoch_fixed_horizon_nonconvex}, and any-time convergence rates in Theorems~\ref{proof thm:stoch_quasar_convex_param_ag} and~\ref{thm: stoch_anytime_non_convex}. To discuss the analysis, we first provide Lemma~\ref{lem: BSFW_iter_bound} which will permit us to prove the parameters necessary to satisfy Assumption~\ref{assum: est}.

\begin{lemma}[BSFW Iteration Bound] Let $\{x^t\}_{t=0}^{+\infty}$ be a sequence generated by Algorithm~\ref{alg:BSFW}, using a step decay $\{\eta_t\}_{t=0}^{+\infty}$. Then we have that for all $t \geq 1$,
\label{lem: BSFW_iter_bound}
\begin{equation*}
    \|x^t - x^{t-1}\| \leq \eta_{t-1}D.
\end{equation*}
\end{lemma} 
\begin{proof}
    From Algorithm~\ref{alg:BSFW}, we have for any $t \geq 1$, if $\gamma_{t-1} = 1$, then 
    \begin{equation*}
        \|x^t - x^{t-1}\| = \eta_{t-1}\|d^{t-1}\| = \eta_{t-1}\|s^{t-1} - x^{t-1}\|  \leq \eta_{t-1}D.
    \end{equation*}
    Suppose instead, $\gamma_{t-1} < 1$, then we have
    \begin{equation*}
        \|x^t - x^{t-1}\| = \gamma_{t-1}\|d^{t-1}\| = \eta_{t-1} \frac{\|s^{t-1} - x^{t-1}\|}{\|d^{t-1}\|} \| d^{t-1} \| =\eta_{t-1}\|s^{t-1} - x^{t-1}\|  \leq \eta_{t-1}D.
    \end{equation*}
    Hence proven.
\end{proof}

\subsubsection{L-SVRG \cite{pmlr-v117-kovalev20a}}
\label{subsec: L-SVRG}
\paragraph{Description} For the L-SVRG estimator \cite{pmlr-v117-kovalev20a}, we use an additional variable $w^t$. We initiate it by
\begin{equation*}
    w^0 = x^0, \quad  m^0 = \nabla f(x^0).
\end{equation*}
For every iteration $t > 0$, we sample a batch $S_t \subset \{1, 2, \cdots, m\}$ of size $b_s$ uniformly at random. $b_s$ is a pre-defined constant per-iteration sample size parameter. Then, we make the update
\begin{align*}
    w^t &= \begin{cases}
        x^{t-1}, &\text{ with probability } p \\
        w^{t-1}, &\text{ with probability } 1 - p 
    \end{cases}\\
    m^t &= \frac{1}{b_s} \sum_{i \in S_t} \left( \nabla f_i(x^t) - \nabla f_i(w^t) \right) + \nabla f(w^t).
\end{align*}
where $p$ is a defined probability parameter. The parameters satisfying Assumption~\ref{assum: est} are provided in Lemma~\ref{params:L-SVRG}.

\paragraph{Verification of Assumption~\ref{assum: est}}
\begin{lemma}
    \label{params:L-SVRG}
    (Parameters for L-SVRG) Let $\{x^t\}_{t=0}^T$ be a sequence generated by Algorithm~\ref{alg:BSFW} using a step decay $\{\eta_t\}_{t=0}^{T-1}$, where the gradient estimator $\{\Phi_t\}_{t=0}^{T-1}$ is L-SVRG defined by \cite{pmlr-v117-kovalev20a}. Then we have the following parameters used in Assumption~\ref{assum: est}.
    \begin{equation*}
        \rho_1 = 1, \quad \rho_2 = \frac{p}{2}, \quad A = \frac{L^2}{b_s} - \frac{L^2p}{2b_s}, \quad B = \frac{8L^2}{pb_s}, \quad C = 0, \quad E = \frac{8}{p}, \quad \sigma^2_t = \|x^t - w^t\|^2.
    \end{equation*}
    where $b_s$ is the stochastic batch size sampled per-iteration, and $p$ is the probability parameter.
\end{lemma} 
\begin{proof}
    This proof follows exactly from \cite{nazykov2024stochastic}. According to Lemma 3 from \cite{li2020unifiedanalysisstochasticgradient}, for any $t$ such that $1 \leq t \leq T - 1$, we get an estimation: 
\begin{equation*}
    \mathbb{E}_{t-1}[\|m^t\|^2]\leq\frac{{L}^2}{b_s}\mathbb{E}_{t-1}[\|x^t - w^t\|^2] + \mathbb{E}_{t-1}[\|\nabla f(x^t)\|^2].
\end{equation*}
Since $m^t$ is an unbiased gradient estimator, the previous inequality turns to:
\begin{equation*}
\mathbb{E}_{t-1}[\|\nabla f(x^t) - m^t\|^2] \leq \frac{L^2}{b_s}\mathbb{E}_{t-1}[\|x^t - w^t\|^2]
\end{equation*}
Suppose at iteration $t$, we have $\gamma_{t-1} < 1$ (case I). Then we have $x^t = x^{t-1} + \gamma_{t-1} d^{t-1}$. Hence,
\begin{align*}
    \mathbb{E}_{t-1}[\|x^{t} - w^{t}\|^2] 
    &= p\mathbb{E}_{t-1}\left[\|x^{t} - x^{t-1}\|^2] + (1-p)\mathbb{E}_{t-1}[\|x^{t} - w^{t-1}\|^2\right] \\
    &= p\gamma_{t-1}^2\mathbb{E}_{t-1} \left[\|d^{t-1}\|^2] + (1-p)\mathbb{E}_{t-1}[\|x^{t-1} + \gamma_{t-1}d^{t-1} - w^{t-1}\|^2\right] \\
    &= \gamma_{t-1}^2\mathbb{E}_{t-1}[\|d^{t-1}\|^2] + (1-p)\|x^{t-1}-w^{t-1}\|^2 \\
    &~~~+ 2\gamma_{t-1}(1-p)\mathbb{E}_{t-1}[\langle x^{t-1} - w^{t-1}, d^{t-1}\rangle] \\
    &=\gamma_{t-1}^2\|d^{t-1}\|^2 + (1-p)\|x^{t-1}-w^{t-1}\|^2 \\
    &~~~~+ 2(1-p)\langle x^{t-1} - w^{t-1},\gamma_{t-1}d^{t-1}\rangle.
\end{align*}
According to Young's inequality for a $\beta > 0$, we have
\begin{equation*}
    \langle x^{t-1} - w^{t-1},\gamma_{t-1}d^{t-1}\rangle \leq \beta\|x^{t-1} - w^{t-1}\|^2 + \frac{1}{\beta}\gamma_{t-1}^2\|d^{t-1}\|^2.
\end{equation*}
Hence,
\begin{align*}
    \mathbb{E}_{t-1}[\|x^{t} - w^{t}\|^2] 
    &\leq \gamma_{t-1}^2\|d^{t-1}\|^2 + (1-p)\|x^{t-1} - w^{t-1}\|^2 + 2(1-p)\beta\|x^{t-1} - w^{t-1}\|^2 \\
    &~~~+ 2\frac{1-p}{\beta}\gamma_{t-1}^2\|d^{t-1}\|^2.
\end{align*}
From Lemma~\ref{lem: BSFW_iter_bound}, $\gamma_{t-1}\|d^{t-1}\| \leq \eta_{t-1}D$, and so we have
\begin{equation*}
    \mathbb{E}_{t-1}[\|x^{t} - w^{t}\|^2]  \leq \left(1 + \frac{2(1-p)}{\beta}\right)\eta_{t-1}^2D^2 + (1-p)(1 + 2\beta)\|x^{t-1} - w^{t-1}\|^2.
\end{equation*}
Suppose instead $\gamma_{t-1} = 1$ (case II). Then we instead have $x^{t} = x^{t-1} + \eta_{t-1}(s^{t-1} - x^{t-1})$. Hence, we have
\begin{align*}
    \mathbb{E}_{t-1}[\|x^{t} - w^{t}\|^2] 
    &= p\mathbb{E}_{t-1}\left[\|x^{t} - x^{t-1}\|^2] + (1-p)\mathbb{E}_{t-1}[\|x^{t} - w^{t-1}\|^2\right] \\
    &= p\eta_{t-1}^2\mathbb{E}_{t-1} \left[\|s^{t-1} - x^{t-1}\|^2] + (1-p)\mathbb{E}_{t-1}[\|x^{t-1} + \eta_{t-1}(s^{t-1} - x^{t-1}) - w^{t-1}\|^2\right] \\
    &= \eta_{t-1}^2\mathbb{E}_{t-1}[\|s^{t-1} - x^{t-1}\|^2] + (1-p)\|x^{t-1}-w^{t-1}\|^2 \\
    &~~~+ 2\eta_{t-1}(1-p)\mathbb{E}_{t-1}[\langle x^{t-1} - w^{t-1}, s^{t-1} - x^{t-1}\rangle] \\
    &=\eta_{t-1}^2\|s^{t-1} - x^{t-1}\|^2 + (1-p)\|x^{t-1}-w^{t-1}\|^2 \\
    &~~~+ 2(1-p)\langle x^{t-1} - w^{t-1},\eta_{t-1}(s^{t-1} - x^{t-1})\rangle.
\end{align*}
Again by using Young's inequality for a $\beta > 0$,
\begin{align*}
    \mathbb{E}_{t-1}[\|x^{t} - w^{t}\|^2] 
    &\leq \eta_{t-1}^2\|s^{t-1} - x^{t-1}\|^2 + (1-p)\|x^{t-1} - w^{t-1}\|^2 + 2(1-p)\beta\|x^{t-1} - w^{t-1}\|^2 \\
    &~~~~+ 2\frac{1-p}{\beta}\eta_{t-1}^2\|s^{t-1} - x^{t-1}\|^2 \\
    &\leq \left(1 + \frac{2(1-p)}{\beta}\right)\eta_{t-1}^2D^2 + (1-p)(1 + 2\beta)\|x^{t-1} - w^{t-1}\|^2.
\end{align*}
Giving us the same result in both cases I and II. Finally, choosing $\beta=\frac{p}{4}$, we get $(1-p)(1 +2\beta) \leq \left(1 - \frac{p}{2}\right).$ and 
\begin{align*}
    \mathbb{E}_{t-1}[\|x^{t} - w^{t}\|^2] &\leq \frac{8}{p}\eta_{t-1}^2D^2 + \left(1-\frac{p}{2}\right)\|x^{t-1} - w^{t-1}\|^2, \\
    \mathbb{E}_{t-1}[\|\nabla f(x^{t}) - m^t\|^2] &\leq\frac{8L^2}{pb_s}\eta_{t-1}^2D^2 + \frac{L^2}{b_s}\left(1-\frac{p}{2}\right)\|x^{t-1} - w^{t-1}\|^2.
\end{align*}
\end{proof}

\subsubsection{SARAH \cite{beznosikov2024sarahfrankwolfemethodsconstrained}}
\label{subsec: SARAH}
\paragraph{Description} For the estimator SARAH \cite{beznosikov2024sarahfrankwolfemethodsconstrained}, we need an additional assumption that the objective function $f$ in \eqref{eq:P} can be represented as a finite sum, i.e., $f(x) = \sum_{i=1}^m f_i(x)$. We start by setting $m^0 = \nabla f(x^0)$. Then, for each iteration $t > 0$, we sample a batch $S_t \subset \{1, 2, \cdots, m\}$ of size $b_s$ uniformly at random. Specifically we have
\begin{align*}
    \label{m_t: SARAH}
    m^0 &= \nabla f(x^0),  \\
    m^{t} &= \begin{cases}
        \nabla f(x^t), & \text{ with probability } p \\
        m^{t-1} + \frac{1}{b_s} \left( \sum_{i \in S_t} \nabla f_i(x^t) - \nabla f_i(x^{t-1}) \right), & \text{ with probability } 1 - p
    \end{cases}
\end{align*}
where $p$ is a defined probability parameter, and $b_s$ is a defined per-iteration batch size. The parameters satisfying Assumption~\ref{assum: est} are provided in Lemma~\ref{params:SARAH}.

\paragraph{Verification of Assumption~\ref{assum: est}}
\begin{lemma}
    \label{params:SARAH}
    (Parameters for SARAH) Let $\{x^t\}_{t=0}^T$ be a sequence generated by Algorithm~\ref{alg:BSFW} using a step decay $\{\eta_t\}_{t=0}^{T-1}$, where the gradient estimator $\{\Phi_t\}_{t=0}^{T-1}$ is SARAH defined by \cite{beznosikov2024sarahfrankwolfemethodsconstrained}. Then we have the following parameters used in Assumption~\ref{assum: est}.
    \begin{equation*}
        \rho_1 = p, \quad \rho_2 = 1, \quad A = 0, \quad B = \frac{1 - p}{b_s}L^2, \quad C = 0, \quad E = 0, \quad \sigma^2_t = 0.
    \end{equation*}
    where $b_s$ refers to the stochastic batch size sampled per-iteration, and $p$ is the probability parameter.
\end{lemma} 
\begin{proof}
    This proof follows exactly from \cite{nazykov2024stochastic}. According to Lemma 3 from \cite{pmlr-v139-li21a}, for a $L$-Lipschitz smooth function $f$, using Lemma~\ref{lem: BSFW_iter_bound}, we bound the difference by
    \begin{align*}
        \mathbb{E}_{t-1}\left[\|\nabla f(x^{t}) - m^{t}\|^2\right] = \mathbb{E}_{t-1}[\|\Delta^t\|^2] &\leq (1 - p)\|\Delta^{t-1}\|^2 + \frac{1-p}{b}L^2\|x^{t}-x^{t-1}\|^2 \\
& \leq (1 - p)\|\Delta^{t-1}\|^2 + \frac{1-p}{b_s}L^2\eta_{t-1}^2D^2.
    \end{align*}
Hence giving us the required coefficients.
\end{proof}

\subsubsection{SAGA \cite{reddi2016stochastic}}
\label{subsec: SAGA}
\paragraph{Description}
To use the SAGA gradient estimator \cite{reddi2016stochastic}, we need an additional assumption that the objective function $f$ in \eqref{eq:P} can be represented as a finite sum, i.e., $f(x) = \sum_{i=1}^m f_i(x)$. For SAGA, we use an additional variable $y^t$ as implemented by \cite{nazykov2024stochastic}. We initiate SAGA by setting
\begin{align*}
    y^0_i &= \nabla f_i (x^0) \text{ for all } i \in \{1, 2, \cdots, m\}, \\
    m^0 &= \nabla f(x^0).
\end{align*}
 For every iteration $t > 0$, we sample a batch $S_t \subset \{1, 2, \cdots, m\}$ of size $b_s$ uniformly at random. $b_s$ is a pre-defined per-iteration sample size parameter. After sampling $S_t$, we make the gradient estimate $m^t$ by setting
 \begin{equation*}
     m^t = \frac{1}{b_s} \sum_{i \in S_t} \left(\nabla f_i(x^{t}) - y^{t-1}_i \right) + \frac{1}{m} \sum_{i=1}^m y^{t-1}_i.
 \end{equation*}
 Then we have for every $i \in \{ 1, 2, \cdots, m\}$, 
 \begin{align*}
     y^t_i &= \begin{cases}
         \nabla f_i(x^t), & i \in S_t \\
     y^{t-1}_i. & i \notin S_t
     \end{cases}
 \end{align*}
The parameters satisfying Assumption~\ref{assum: est} are provided in Lemma~\ref{params:SAGA}.
\paragraph{Verification of Assumption~\ref{assum: est}}
\begin{lemma}
     \label{params:SAGA}
    (Parameters for SAGA) Let $\{x^t\}_{t=0}^T$ be a sequence generated by Algorithm~\ref{alg:BSFW}, using a step decay $\{\eta_t\}_{t=0}^{T-1}$, where the gradient estimator $\{\Phi_t\}_{t=0}^{T-1}$ is SAGA defined by \cite{reddi2016stochastic}. Suppose the objective function $f$ can be represented as $f(x) = \sum_{i=1}^m f_i(x)$. Then we have the following parameters used in Assumption~\ref{assum: est}.
    {\footnotesize\begin{align*}
        \rho_1 = 1, ~~ \rho_2 = \frac{b_s}{2m}, ~~ A = \frac{1}{b_s} + \frac{1}{2m}, ~~ B = \frac{L^2}{b_s m}\left(1 + \frac{2m}{b_s} \right), ~~ C = 0, ~~ E = \frac{2m}{b_s}L^2, ~~\sigma_{t}^2 = \frac{1}{m}\sum\limits_{j=1}^m\|\nabla f_j(x^{t}) - y_j^{t}\|^2.
    \end{align*}}
    where $b_s$ is defined as the stochastic batch size sampled per-iteration.
\end{lemma}
\begin{proof}
     This proof follows exactly from \cite{nazykov2024stochastic}. The difference between estimator and exact gradient is bounded by:
    {\small\begin{align*}
        \mathbb{E}_{t-1}\left[\|m^t - \nabla f(x^t)\|^2\right] &=\mathbb{E}_{t-1}\left[\Biggl\|\frac{1}{b_s}\sum\limits_{i \in S_t}\left[\nabla f_i(x^{t}) - y_i^{t-1}\right] + \frac{1}{m}\sum\limits_{j = 1}^m y_j^{t-1} - \nabla f(x^t) \Biggr\|^2\right] \\
        &= \mathbb{E}_{t-1}\Biggl[\Biggl\|\frac{1}{b_s}\left(\sum\limits_{i \in S_t}\left[\nabla f_i(x^{t}) - y_i^{t-1}\right] - \left(\frac{1}{m}\sum\limits_{j = 1}^m\left[\nabla f_j(x^t)-y_j^{t-1}\right]\right)\right) \Biggr\|^2\Biggr]. \\
    \end{align*}}
    By using Lemma B.2 from \cite{nazykov2024stochastic}
    \begin{equation*}
        \mathbb{E}_{t-1}\left[\|m^t - \nabla f(x^t)\|^2\right]\leq \frac{1}{b_s m}\sum\limits_{j=1}^m\left\|\nabla f_j(x^t) - y_j^{t-1} - \left(\frac{1}{m}\sum\limits_{i=1}^m\left[\nabla f_i(x^t)-y_i^{t-1}\right]\right)\right\|^2.
    \end{equation*}
    Now, $\frac{1}{m}\sum\limits_{i=1}^m$ can be described as an expected value and $\mathbb{E}[\|x - \mathbb{E}[x]\|^2 ]\leq \mathbb{E}[\|x\|^2]$. Furthermore, using Young's inequality with $\alpha > 0$,
    {\footnotesize\begin{align*}
        \mathbb{E}_{t-1}\left[\|m^t - \nabla f(x^t)\|^2\right]&\leq\frac{1}{b_s m}\sum\limits_{j=1}^m\left\|\nabla f_j(x^t) - y_j^{t-1}\right\|^2 \\
        &\leq \frac{1}{b_s m}\left(1+\alpha\right)\sum\limits_{j=1}^m\|\nabla f_j(x^t) - \nabla f_j(x^{t-1})\|^2 + \frac{1}{b_s m}\left(1+\frac{1}{\alpha}\right)\sum\limits_{j=1}^m\|\nabla f_j(x^{t-1}) - y_j^{t-1}\|^2. 
    \end{align*}}
    Using $L$-smoothness of $f$, Lemma~\ref{lem: BSFW_iter_bound}, and by setting $\sigma_{t-1}^2 = \frac{1}{m}\sum\limits_{j=1}^m\|\nabla f_j(x^{t-1}) - y_j^{t-1}\|^2$,
    \begin{equation*}
        \mathbb{E}_{t-1}\left[\|m^t - \nabla f(x^t)\|^2\right] \leq \frac{1}{b_s m}\left(1+\alpha\right)L^2 \eta_{t-1}^2 D^2 + \frac{1}{b_s}\left(1+\frac{1}{\alpha}\right)\sigma_{t-1}^2.
    \end{equation*}
We can put $\alpha = \frac{2m}{b_s}$ to obtain the needed estimates, i.e. 
\begin{equation*}
    \mathbb{E}_{t-1}\left[\|m^t - \nabla f(x^t)\|^2\right] \leq \frac{L^2}{b_s m}\left(1+ \frac{2m}{b_s} \right) \eta_{t-1}^2 D^2 + \frac{1}{b_s}\left(1+\frac{b_s}{2m}\right)\sigma_{t-1}^2.
\end{equation*}
To bound the term $\sigma_{t-1}^2$,
\begin{align*}
    \mathbb{E}_{t-1}[\sigma_{t}^2]&=\mathbb{E}_{t-1}\left[\frac{1}{m}\sum\limits_{j = 1}^m\|\nabla f_j(x^{t}) - y_j^{t}\|^2 \right] = \left(1 - \frac{b_s}{m}\right)\frac{1}{m}\sum\limits_{j = 1}^m\|\nabla f_j(x^{t}) - y_j^{t-1}\|^2 \\ 
    & = \left(1 - \frac{b_s}{m}\right)\frac{1}{m}\sum\limits_{j = 1}^m\|\nabla f_j(x^{t}) - \nabla f_j(x^{t-1}) + \nabla f_j(x^{t-1})-  y_j^{t-1}\|^2.
\end{align*}
By using Young's Inequality again with $\beta > 0$, 
\begin{equation*}
    \mathbb{E}_{t-1}[\sigma_{t}^2] \leq \left(1 - \frac{b_s}{m}\right)(1 + \beta)\frac{1}{m}\sum\limits_{j = 1}^m\|\nabla f_j(x^{t-1}) - y_j^{t-1}\|^2 + \left(1 - \frac{b_s}{m}\right)\left(1 + \frac{1}{\beta}\right)L^2\|x^{t} - x^{t-1}\|^2.
\end{equation*}
With $\beta = \frac{b_s}{2m}$, and by using Lemma~\ref{lem: BSFW_iter_bound}, we have:
\begin{equation*}
    \mathbb{E}_{t-1}[\sigma_{t}^2] \leq \left(1 - \frac{b_s}{2m}\right)\sigma_{t-1}^2 + \frac{2m}{b_s}L^2\eta_{t-1}^2D^2.
\end{equation*}
\end{proof}

\subsubsection{SEGA \cite{NEURIPS2018_fc2c7c47}}
\label{subsec: SEGA}

\paragraph{Description} For SEGA \cite{NEURIPS2018_fc2c7c47}, we require an additional variable $h^t$ as implemented by \cite{nazykov2024stochastic}. We initialize by setting both $h^0$ and $m^0$ by the full gradient at $t = 0$, specifically, $m^0 = \nabla f(x^0)$ and $h^0 = \nabla f(x^0)$.  Then at every iteration $t > 0$, we randomly sample $i_t \in \{1, 2, \cdots, n\}$ and approximate the gradient by
\begin{equation*}
    m^t = ne_{i_t} \left( \nabla_{i_t} f(x^t) - h^{t-1}_{i_t} \right) + h^{t-1}.
\end{equation*}
We also update $h^t$ by setting $h^t = h^{t-1} + e_{i_t}\left( \nabla_{i_t} f(x^t) - h^{t-1} \right)$. The parameters satisfying Assumption~\ref{assum: est} are provided in Lemma~\ref{params:SEGA}.

\paragraph{Verification of Assumption~\ref{assum: est}}
\begin{lemma}
    \label{params:SEGA}
    (Parameters for SEGA) Let $\{x^t\}_{t=0}^T$ be a sequence generated by Algorithm~\ref{alg:BSFW} using a step decay $\{\eta_t\}_{t=0}^{T-1}$, where the gradient estimator $\{\Phi_t\}_{t=0}^{T-1}$ is SEGA defined by \cite{NEURIPS2018_fc2c7c47}. Then we have the following parameters used in Assumption~\ref{assum: est}.
    \begin{equation*}
        \rho_1 = 1, \quad \rho_2 = \frac{1}{2n}, \quad A = n, \quad B = n^2L^2, \quad C = 0, \quad E = 3L^2n, \quad \sigma_t^2 = \|h^{t} - \nabla f(x^t)\|^2.
    \end{equation*}
\end{lemma} 
\begin{proof}
    This proof follows exactly from \cite{nazykov2024stochastic}. $I$ refers to the identity matrix of dimensions $n \times n$. We first bound the difference between estimator and exact gradient:
{\small\begin{align*}
    \mathbb{E}_{t-1}\left[\|m^{t} - \nabla f(x^{t})\|^2\right] &= \mathbb{E}_{t-1}\left[\|n e_{i_{t}} e_{i_{t}}^{\top}(\nabla f(x^{t}) - h^{t - 1}) + h^{t- 1 } - \nabla f(x^{t})\|^2\right] \\
    &= \mathbb{E}_{t-1}\left[\|(I - ne_{i_{t}} e_{i_{t}}^{\top})(h^{t-1} - \nabla f(x^{t}))\|^2\right]\\
    &= \mathbb{E}_{t-1}\left[(h^{t-1} - \nabla f(x^{t}))^{\top}(I - n e_{i_{t}} e_{i_{t}}^{\top})^{\top}(I - n e_{i_{t}} e_{i_{t}}^{\top})(h^{t-1} - \nabla f(x^{t}))\right] \\
    &= (h^{t-1} - \nabla f(x^{t}))^{\top}\mathbb{E}_{t-1}\left[I - 2n e_{i_{t}} e_{i_{t}}^{\top} + n^2e_{i_{t}} e_{i_{t}}^{\top}\right](h^{t-1} - \nabla f(x^{t})) \\
    &= (h^{t} - \nabla f(x^{t}))^{\top}\left[I - 2\cdot I + n\cdot I\right](h^{t-1} - \nabla f(x^{t})) \\
    &= (n-1)\|h^{t-1} - \nabla f(x^{t})\|^2 \\
    &= (n-1)\|h^{t-1} - \nabla f(x^{t-1}) + \nabla f(x^{t-1}) - \nabla f(x^{t})\|^2  \\
     &= (n-1)\bigl(\|h^{t-1} - \nabla f(x^{t-1}) \|^2 + \|\nabla f(x^{t-1}) - \nabla f(x^t) \|^2 \\ 
    &~~~~~+ 2\langle h^{t-1} - \nabla f(x^{t-1}), \nabla f(x^{t-1}) - \nabla f(x^t) \rangle \bigr).
\end{align*}}
By using Young's inequality with a parameter $\alpha > 0$ on the inner product ${2\langle h^{t-1} - \nabla f(x^{t-1}), \nabla f(x^{t-1}) - \nabla f(x^t) \rangle}$ and the $L$-Lipschitz smoothness of $f$, and Lemma~\ref{lem: BSFW_iter_bound}, we get the bound
{\small\begin{align*}
    \mathbb{E}_{t-1}\left[\|m^{t} - \nabla f(x^{t})\|^2\right] &\leq (n-1)(1+\alpha)\|h^{t-1} - \nabla f(x^{t-1})\|^2 + (n-1)\left(1+\frac{1}{\alpha}\right)L^2 \|x^t - x^{t-1}\|^2 \\
     &\leq (n-1)(1+\alpha)\|h^{t-1} - \nabla f(x^{t-1})\|^2 + (n-1)\left(1+\frac{1}{\alpha}\right)\eta_{t-1}^2L^2D^2.
\end{align*}}
Then by using similar arguments as above ($L$-Lipschitz smoothness of $f$, Young's Inequality with $\beta > 0$ and Lemma~\ref{lem: BSFW_iter_bound}),
{\small\begin{align*}
    \mathbb{E}_{t-1}\left[\|h^{t} - \nabla f(x^{t})\|^2\right] &= \mathbb{E}_{t-1}\left[\|h^{t-1} + e_{i_{t}} e_{i_{t}}^{\top}(\nabla f(x^t) - h^{t-1}) - \nabla f(x^{t})\|^2\right]\\
    &=\mathbb{E}_{t}\left[\|(I- e_{i_{t}} e_{i_{t}}^{\top})(h^{t-1} - \nabla f(x^{t}))\|^2\right]\\
    &=\left(1 - \frac{1}{n}\right)\|h^{t-1} - \nabla f(x^t)\|^2 \\
    &=\left(1 - \frac{1}{n}\right)\|h^{t-1} - \nabla f(x^{t-1}) + \nabla f(x^{t-1}) - \nabla f(x^t)\|^2 \\
    &\leq \left(1-\frac{1}{n}\right)(1+\beta)\|h^{t-1} - \nabla f(x^{t-1})\|^2 + \left(1-\frac{1}{n}\right)\left(1 + \frac{1}{\beta}\right)\eta_{t-1}^2L^2D^2.
\end{align*}}
If $\beta = \frac{1}{2n}$ then $(1 - \frac{1}{n})(1 + \frac{1}{2n}) \leq 1 - \frac{1}{2n}$ and $(1-\frac{1}{n})(1 + 2n) \leq 2n$, then as $n \geq 1$:
\begin{equation*}
     \mathbb{E}_{t}\left[\|h^{t} - \nabla f(x^{t})\|^2\right] \leq \left(1-\frac{1}{2n}\right)\|h^{t-1} - \nabla f (x^{t-1})\|^2 + 3nL^2\eta_{t-1}^2D^2.
\end{equation*}
Setting $\alpha = \frac{1}{n}$, we have $(n - 1)(1 + \frac{1}{n}) \leq n$ and $(n-1)(n + 1) \leq n^2$ and thus
\begin{equation*}
      \mathbb{E}_{t-1}[\|\Delta^t\|^2] = \mathbb{E}_{t-1}\left[\|m^{t} - \nabla f(x^{t})\|^2\right]  \leq n \|h^{t-1} - \nabla f(x^{t-1})\|^2 + n^2\eta_{t-1}^2L^2D^2.
\end{equation*}
\end{proof}

\subsubsection{JAGUAR \cite{nazykov2024stochastic}}
\label{subsec: JAGUAR}
\paragraph{Description} For the JAGUAR estimator \cite{nazykov2024stochastic}, we initiate $m^0$ by the full gradient at $x^0$, i.e., $m^0 = \nabla f(x^0)$. Then at every iteration $t > 0$, we randomly sample $i_t \in \{1, 2, \cdots, n\}$ and approximate the gradient by
\begin{align*}
    m^t &= e_{i_t}\left(\nabla_{i_t} f(x^{t-1}) - m^{t-1}_{i_t} \right) + m^{t-1}.
\end{align*}
The parameters satisfying Assumption~\ref{assum: est} are provided in  Lemma~\ref{params:JAGUAR}.
\paragraph{Verification of Assumption~\ref{assum: est}}
\begin{lemma}
     \label{params:JAGUAR}
    (Parameters for JAGUAR) Let $\{x^t\}_{t=0}^T$ be a sequence generated by Algorithm~\ref{alg:BSFW} using a step decay $\{\eta_t\}_{t=0}^{T-1}$, where the gradient estimator $\{\Phi_t\}_{t=0}^{T-1}$ is JAGUAR defined by \cite{nazykov2024stochastic}.  Then we have the following parameters used in Assumption~\ref{assum: est}.
    \begin{equation*}
        \rho_1 = \frac{1}{2n}, \quad \rho_2 = 1, \quad A = 0, \quad B = 3nL^2, \quad C = 0, \quad E = 0, \quad \sigma_t^2 = 0.
    \end{equation*}
\end{lemma}
\begin{proof}
    This proof follows exactly from \cite{nazykov2024stochastic}. By first bounding the difference between estimator and exact gradient:
{\footnotesize\begin{align*}
    \mathbb{E}_{t-1}\left[\|m^t - \nabla f(x^{t})\|^2\right] &= \mathbb{E}_{t-1}\left[\| e_{i_t} e_{i_t}^{\top}(\nabla f(x^{t-1}) - m^{t-1}) + m^{t-1} - \nabla f(x^{t})\|^2\right] \\
    &= \mathbb{E}_{t-1}\left[\| e_{i_t} e_{i_t}^{\top}(\nabla f(x^{t-1}) - m^{t-1}) + m^{t-1} - \nabla f(x^{t}) + \nabla f(x^{t-1}) - \nabla f(x^{t-1})\|^2\right]\\
    &= \mathbb{E}_{t-1}\left[\|(I - e_{i_t} e_{i_t}^{\top})(\nabla f(x^{t-1}) - m^{t-1})  + \nabla f(x^{t-1}) - \nabla f(x^{t})\|^2\right].
\end{align*}}
Using Young's Inequality with a parameter $\beta > 0$, $L$-Lipschitz smoothness property of $f$ and Lemma~\ref{lem: BSFW_iter_bound}, we have
{\small\begin{align*}
    \mathbb{E}_{t-1}\left[\|m^t - \nabla f(x^{t})\|^2\right] &\leq (1+\beta)\mathbb{E}_{t-1}\left[\|(I - e_{i_t} e_{i_t}^{\top})(m^{t-1} - \nabla f(x^{t-1}))\|^2\right] + \left(1 + \frac{1}{\beta}\right)\eta_{t-1}^2L^2D^2 \\
    &\leq (1+\beta)\left(1 - \frac{1}{n}\right)\|m^{t-1} - \nabla f(x^{t-1})\|^2+ \left(1 + \frac{1}{\beta}\right)\eta_{t-1}^2L^2D^2.
\end{align*}}
By setting $\beta = \frac{1}{2n}$, we have $(1 - \frac{1}{n})(1 + \frac{1}{2n}) \leq 1 - \frac{1}{2n}$ and as $n \geq 1$, we get
\begin{equation*}
    \mathbb{E}_{t-1}\left[\|\Delta^t\|^2\right] \leq \left(1 - \frac{1}{2n}\right)\| \Delta^{t-1} \|^2 + 3\eta_{t-1}^2nL^2D^2.
\end{equation*}
\end{proof}

\subsubsection{ZOJA \cite{nazykov2024stochastic}}
\label{subsec: ZOJA}
\paragraph{Description} For the ZOJA estimator introduced by \cite{nazykov2024stochastic},  we initiate $m^0$ by the zero order approximation using a defined parameter $\tau$ across every coordinate. Specifically,
\begin{equation*}
    m^0 = \sum_{i = 1}^n \left( \frac{f(x^0 + \tau e_i) - f(x^0)}{\tau} \right) e_{i}.
\end{equation*}
Next at every iteration $t > 0$, we randomly sample $i_t \in \{1, 2, \cdots, n\}$ and approximate the gradient by
\begin{align*}
    \widetilde{\nabla}_{i_t} f(x^{t-1}) &= \frac{f(x^{t-1} + \tau e_{i_t}) - f(x^{t-1})}{\tau}, \\
    m^t &= e_{i_t}\left(\widetilde{\nabla}_{i_t} f(x^{t-1}) - m^{t-1}_{i_t} \right) + m^{t-1}.
\end{align*}
The parameters satisfying Assumption~\ref{assum: est} are given in Lemma~\ref{params:ZOJA}.

\paragraph{Verification of Assumption~\ref{assum: est}}
\begin{lemma}
\label{params:ZOJA}
    (Parameters for ZOJA) Let $\{x^t\}_{t=0}^T$ be a sequence generated by Algorithm~\ref{alg:BSFW} using a step decay $\{\eta_t\}_{t=0}^{T-1}$, where the gradient estimator $\{\Phi_t\}_{t=0}^{T-1}$ is ZOJA defined by \cite{nazykov2024stochastic}.  Then we have the following parameters used in Assumption~\ref{assum: est}.
    \begin{equation*}
        \rho_1 = \frac{1}{4n}, \quad \rho_2 = 1, \quad A = 0, \quad B = 3nL^2, \quad C = 2nL^2 \tau^2, \quad E = 0, \quad \sigma_t^2 = 0.
    \end{equation*}
    where $\tau > 0$ is the zero-order approximation parameter.
\end{lemma} 
\begin{proof}
    This proof follows exactly from \cite{nazykov2024stochastic}. We bound the difference between estimator and exact gradient:
\begin{align*}
    \mathbb{E}_{t-1}\left[\|m^{t} - \nabla f(x^{t})\|^2\right] = \mathbb{E}_{t-1}\left[\|\Delta^t\|^2 \right] &= \mathbb{E}_{t-1}[\| e_{i_t}(\widetilde{\nabla}_{i_{t}} f(x^{t-1}) - m_{i_{t}}^{k-1}) + m^{t-1} - \nabla f(x^{t})\|^2] \\
    &= \mathbb{E}_{t-1}[\| e_{i_t}(\widetilde{\nabla}_{i_{t}} f(x^{t-1}) - m_{i_{t}}^{t-1}) + m^{t-1} - \nabla f(x^{t}) \\
    &+ \nabla f(x^{t-1}) - \nabla f(x^{t-1})\|^2]. 
\end{align*}
Using Young's inequality with $\beta > 0$, $L$-Lipschitz smoothness of $f$ and Lemma~\ref{lem: BSFW_iter_bound}, we get
{\footnotesize\begin{align*}
    \mathbb{E}_{t-1}\left[\|\Delta^t\|^2\right]  &\leq  (1+\beta)\mathbb{E}_{t-1}\left[\| e_{i_t}(\widetilde{\nabla}_{i_{t}} f(x^{t-1}) - m_{i_{t}}^{t-1}) + m^{t-1} - \nabla f(x^{t-1})\|^2\right]  + \left(1 + \frac{1}{\beta}\right) \|\nabla f(x^{t}) - \nabla f(x^{t-1}) \|^2 \\
    &\leq  (1+\beta)\mathbb{E}_{t-1}\left[\| e_{i_t}(\widetilde{\nabla}_{i_{t}} f(x^{t-1}) - m_{i_{t}}^{t-1}) + m^{t-1} - \nabla f(x^{t-1})\|^2\right]  + \left(1 + \frac{1}{\beta}\right) L^2 \| x^t - x^{t-1} \|^2 \\
    &\leq  (1+\beta)\mathbb{E}_{t-1}\left[\| e_{i_t}(\widetilde{\nabla}_{i_{t}} f(x^{t-1}) - m_{i_{t}}^{t-1}) + m^{t-1} - \nabla f(x^{t-1})\|^2\right]  + \left(1 + \frac{1}{\beta}\right)\eta_{t-1}^2L^2D^2\\
    &\leq (1+\beta)\mathbb{E}_{t-1}\left[\|(I - e_{i_t}e_{i_t}^{\top})(m^{t-1} - \nabla f(x^{t-1})) + e_{i_t}(\widetilde{\nabla}_{i_t}f(x^{t-1}) - \nabla_{i_t} f(x^{t-1}))\|^2\right]\\
    &~~~~~+ \left(1 + \frac{1}{\beta}\right)\eta_{t-1}^2L^2D^2.
\end{align*}}
Using Lemma B.1 from \cite{nazykov2024stochastic}, for any index $j$ such that $1 \leq j \leq n$,
\begin{align*}
    \|e_j(\widetilde{\nabla}_{j}f(x^{t-1}) - \nabla_j f(x^{t-1})\|^2 &= \left(\widetilde{\nabla}_j f(x^{t-1}) - \nabla_j f(x^{t-1}) \right)^2 \\
    &=\left(\frac{f(x^{t-1} + \tau e_j) - f(x^{t-1})}{\tau} - \nabla_j f(x^{t-1}) \right)^2 \\
    &= \frac{1}{\tau^2} \left(f(x^{t-1} + \tau e_j) - f(x^{t-1}) - \tau \nabla_j f(x^{t-1}) \right)^2 \\
    &= \frac{1}{\tau^2} \left(f(x^{t-1} + \tau e_j) - f(x^{t-1}) - \langle \tau e_j, \nabla f(x^{t-1}) \rangle \right)^2 \\
    &\leq \frac{L}{4\tau^2} \|\tau e_j \|^4 \\
    &\leq \frac{L \tau^2}{4}.
\end{align*}
Hence, we have the expression
\begin{equation*}
    \mathbb{E}_{t-1}\left[ \|e_{i_t}(\widetilde{\nabla}_{i_t}f(x^{t-1}) - \nabla f(x^{t-1})\|^2 \right] \leq \frac{L \tau^2}{4}.
\end{equation*}
Reusing Young's inequality with a parameter $\alpha > 0$, we get
\begin{align*}
    \mathbb{E}_{t-1}[\|\Delta^t\|^2]     &\leq (1+\beta)(1+\alpha)\left(1 - \frac{1}{n}\right)\|m^{t-1} - \nabla f(x^{t-1})\|^2 \\
    &~~~+ (1+\beta)\left(1+\frac{1}{\alpha}\right) \mathbb{E}_{t-1}\left[ \|e_{i_t}(\widetilde{\nabla}_{i_t}f(x^{t-1}) - \nabla f(x^{t-1})\|^2 \right]\\
    &~~~+ \left(1 + \frac{1}{\beta}\right)\eta_{t-1}^2L^2D^2,\\
    \mathbb{E}_{t-1}[\|\Delta^t\|^2] &\leq (1+\beta)(1+\alpha)\left(1 - \frac{1}{n}\right)\| \Delta^{t-1} \|^2 + (1+\beta)\left(1+\frac{1}{\alpha}\right)\frac{L^2\tau^2}{4} \\
    &~~~+ \left(1 + \frac{1}{\beta}\right)\eta_{t-1}^2L^2D^2.
\end{align*}

If $\beta = \frac{1}{2n}$, then $(1 - \frac{1}{n})(1 + \frac{1}{2n}) \leq 1 - \frac{1}{2n}$. And with $\alpha = \frac{1}{4n}$, we get the upper bounds $(1 -\frac{1}{2n})(1 + \frac{1}{4n}) \leq (1 - \frac{1}{4n})$ and $(1 + \frac{1}{2n})(1 + 4n) = 4n + 3 + \frac{1}{2n} \leq 4n + 4n = 8n$. Hence, 
\begin{equation*}
 \mathbb{E}_{t-1}\left[\|\Delta^t\|^2\right] \leq \left(1 - \frac{1}{4n}\right)\|\Delta^{t-1}\|^2 + 3\eta_{t-1}^2nL^2D^2 + 2nL^2\tau^2.
\end{equation*}
\end{proof}

\subsubsection{Heavy Ball \cite{JMLR:v21:18-764}}
\label{subsec: Heavy Ball}
\paragraph{Description} For the Heavy Ball estimator as proposed by \cite{JMLR:v21:18-764}, $\{m^t\}_{t=0}^{T-1}$ in Algorithm~\ref{alg:BSFW} is defined by the following equations
\begin{align*}
    m^0 &= 0 ,\\
    m^t &= (1 - \tilde{\rho_t})m^{t-1} + \tilde{\rho_t} \tilde{\nabla}f(x^t, \xi_t),
\end{align*}
where $\tilde{\rho_t}$ is a defined momentum function. For this estimator, we also assume that there exists a constant bound of the variance of unbiased stochastic gradients \cite{JMLR:v21:18-764}, as described by Assumption~\ref{assump: heavy_ball constant bound}. This assumption follows directly from Assumption 3 stated in \cite{JMLR:v21:18-764}. \\ \\
We first provide a recursion Lemma~\ref{lem: heavy ball error recursion} when Algorithm~\ref{alg:BSFW} uses the Heavy Ball estimator. Then, Lemma~\ref{params:Heavy Ball*} uses Lemma~\ref{lem: heavy ball error recursion} to find the constants necessary for Assumption~\ref{assum: est}. However, the constants are a function of the horizon $T$, which causes problems when attempting to find desired convergence bounds and complexities. Due to this, we propose alternate any-time convergence rates for $\rho$-quasar-convex and nonconvex objective functions. \\ \\
For $\rho$-quasar-convex functions, we bound the error term through Lemma~\ref{lem: heavy ball error bounds quasar convex}, provide an alternate recursion Theorem~\ref{thm: alt recursion quasar convex} for $\rho$-quasar-convex functions, and show a convergence rate of $\mathcal{O}\left( \frac{1}{t^{1/3}} \right)$ in Theorem~\ref{thm: convergence heavy ball quasar convex}. For nonconvex functions, we bound the error term by Lemma~\ref{lem: heavy ball error bounds non convex} and show a convergence rate of $\mathcal{O}\left( \frac{\ln(t)}{t^{1/4}} \right)$ in Theorem~\ref{thm: convergence heavy ball non convex}.

\paragraph{Convergence Analysis}
\begin{assumption}[Heavy Ball Variance Bound]
    The variance of unbiased stochastic gradients $\nabla \tilde{f}(x, \xi)$ is bounded above by $\sigma^2$, i.e., for all random variables $\xi$, 
    \begin{equation*}
        \mathbb{E}\left[ \| \nabla \tilde{f}(x, \xi) - \nabla f(x) \|^2 \right] \leq \sigma^2
    .\end{equation*}
    \label{assump: heavy_ball constant bound}
\end{assumption}

\begin{lemma}
    \label{lem: heavy ball error recursion}
    (Heavy Ball Error Recursion) Let $\{x^t\}_{t=0}^{T}$ be a sequence generated by Algorithm~\ref{alg:BSFW}. If the objective function $f$ satisfies Assumptions~\ref{assum: L_smooth} and \ref{assump: heavy_ball constant bound}, we have the bound for the sequence of squared errors
    \begin{equation*}
        \mathbb{E}_{t-1}\left[ \|\Delta^t\|^2\right] \leq \left(1 - \frac{\tilde{\rho_t}}{2} \right) \|\Delta^{t-1}\|^2 + \tilde{\rho_t}^2 \sigma^2 + \frac{2L^2D^2\eta_{t-1}^2}{\tilde{\rho_t}}.
    \end{equation*}
\end{lemma}
\begin{proof}
    The proof follows exactly from Lemma 1 in \cite{JMLR:v21:18-764}. By definition of $m^t$, we have $m^t = (1 - \rho_t) m^{t-1} + \rho_t \tilde{\nabla} f(x^t, \xi_t)$. Hence,
    \begin{equation*}
        \mathbb{E}_{t-1}\left[\|\nabla f(x^t) - m^t\|^2\right] = \mathbb{E}_{t-1}\left[\|\nabla f(x^t) - (1 - \tilde{\rho_t})m^{t-1} - \tilde{\rho_t} \tilde{\nabla} f(x^t, \xi_i)\|^2 \right].
    \end{equation*}
    Adding and subtracting $(1 - \tilde{\rho_t}) \nabla f(x^{t-1})$,
    {\small\begin{align*}
        &\mathbb{E}_{t-1}\left[\|\nabla f(x^t) - m^t\|^2\right] \\
        &= \mathbb{E}_{t-1}\left[\|\nabla f(x^t) - (1 - \tilde{\rho_t})\nabla f(x^{t-1}) + (1 - \tilde{\rho_t})\nabla f(x^{t-1}) - (1 - \tilde{\rho_t})m^{t-1} - \tilde{\rho_t} \tilde{\nabla} f(x^t, \xi_t)\|^2 \right] \\
        &= \mathbb{E}_{t-1}\left[\| \tilde{\rho_t}(\nabla f(x^t) - \tilde{\nabla}f(x^t, \xi_t)) + (1 - \tilde{\rho_t})(\nabla f(x^t) - \nabla f(x^{t-1})) + (1 - \tilde{\rho_t})(\nabla f(x^{t-1}) - m^{t-1}) \|^2\right] \\
        &= \tilde{\rho_t}^2\mathbb{E}_{t-1}\left[\|\nabla f(x^t) - \tilde{\nabla}f(x^t, \xi_t)\|^2 \right] + (1 - \tilde{\rho_t})^2\|\nabla f(x^t) - \nabla f(x^{t-1}) \|^2 \\
        &~~~+ (1 - \tilde{\rho_t})^2\|\nabla f(x^{t-1}) - m^{t-1} \|^2  + 2\tilde{\rho_t}(1 - \tilde{\rho_t})\mathbb{E}_{t-1} \left[\langle \nabla f(x^t) - \tilde{\nabla}f(x^t, \xi_t),\nabla f(x^t) - \nabla f(x^{t-1}) \rangle \right] \\
        &~~~+ 2\tilde{\rho_t}(1 - \tilde{\rho_t})\langle \nabla f(x^t) - \nabla f(x^{t-1}), \nabla f(x^{t-1}) - m^{t-1} \rangle \\
        &~~~+ 2(1 - \tilde{\rho_t})^2\mathbb{E}_{t-1}\left[\langle \nabla f(x^{t-1}) - m^{t-1}, \nabla f(x^t) - \tilde{\nabla}f(x^t, \xi_t)\rangle \right].
    \end{align*}}
    Recall that $\Delta^t = \nabla f(x^t) - m^t$. Using the fact that $\tilde{\nabla} f(x^t, \xi_t)$ is unbiased, i.e. $\mathbb{E}_{t-1}\left[\tilde{\nabla}f(x^t, \xi_t)\right] = \nabla f(x^t)$, $L$-smoothness of $f$, Assumption~\ref{assump: heavy_ball constant bound} and Lemma~\ref{lem: BSFW_iter_bound}, we get
    {\footnotesize\begin{align*}
        \mathbb{E}_{t-1}\left[\|\Delta^t\|^2\right] &\leq \tilde{\rho_t}^2\sigma^2 + (1 - \tilde{\rho_t})^2L^2 \|x^t - x^{t-1}\|^2 + (1 - \tilde{\rho_t})^2 \|\Delta^{t-1} \|^2 + 2(1 - \tilde{\rho_t})^2\langle \nabla f(x^t) - \nabla f(x^{t-1}), \Delta^{t-1} \rangle \\
        &\leq \tilde{\rho_t}^2\sigma^2 + (1 - \tilde{\rho_t})^2L^2\eta_{t-1}^2D^2 + (1 - \tilde{\rho_t})^2 \|\Delta^{t-1} \|^2 + 2(1 - \tilde{\rho_t})^2\langle \nabla f(x^t) - \nabla f(x^{t-1}), \Delta^{t-1} \rangle.
    \end{align*}}
    Using Young's inequality with a parameter $\beta_t > 0$, we get
    {\footnotesize\begin{align*}
        2(1 - \tilde{\rho_t})^2\langle \nabla f(x^t) - \nabla f(x^{t-1}), \Delta^{t-1} \rangle &\leq 2(1 - \tilde{\rho_t})^2 \beta_t \|\Delta^{t-1}\|^2 + 2(1 - \tilde{\rho_t})^2\frac{1}{\beta_t}\|\nabla f(x^t) - \nabla f(x^{t-1})\|^2 \\
        &\leq 2(1 - \tilde{\rho_t})^2\beta_t\|\Delta^{t-1}\|^2 + 2(1 - \tilde{\rho_t})^2\frac{1}{\beta_t}L^2\eta_{t-1}^2D^2.
    \end{align*}}
    This gives us
    \begin{align*}
        \mathbb{E}_{t-1}\left[\|\Delta^t\|^2\right] &\leq \tilde{\rho_t}^2\sigma^2 + (1 - \tilde{\rho_t})^2\left(1 + \frac{1}{\beta_t} \right) L^2 \eta_{t-1}^2D^2 + (1 - \tilde{\rho_t})^2 (1 + \beta_t)\|\Delta^{t-1} \|^2.
    \end{align*}
    Since $\tilde{\rho_t} \leq 1$, $(1 - \tilde{\rho_t})^2 \leq (1 - \tilde{\rho_t})$. By setting $\beta_t = \frac{\tilde{\rho_t}}{2}$ we have $(1 - \tilde{\rho_t})(1 + (2/\tilde{\rho_t})) \leq (2/\tilde{\rho_t})$ and $(1 - \tilde{\rho_t})(1 + (\rho_t/2)) \leq (1 - (\tilde{\rho_t}/2))$ and thus, 
    \begin{align*}
        \mathbb{E}_{t-1}\left[\|\Delta^t\|^2\right] &\leq \tilde{\rho_t}^2\sigma^2 + (1 - \tilde{\rho_t})\left(1 + \frac{1}{\beta_t} \right) L^2 \eta_{t-1}^2D^2 + (1 - \tilde{\rho_t}) (1 + \beta_t)\|\Delta^{t-1} \|^2 \\
        &\leq \tilde{\rho_t}^2\sigma^2 + \frac{2L^2D^2\eta_{t-1}^2}{\tilde{\rho_t}} + \left(1 - \frac{\tilde{\rho_t}}{2}\right)\|\Delta^{t-1}\|^2 .
    \end{align*}
    Hence proven.
\end{proof}

\begin{lemma}
    \label{params:Heavy Ball*}
    (Parameters for Heavy Ball) Let $\{x^t\}_{t=0}^T$ be a sequence generated by Algorithm~\ref{alg:BSFW} using a step decay $\{\eta_t\}_{t=0}^{T-1}$, where the gradient estimator $\{\Phi_t\}_{t=0}^{T-1}$ is the Heavy Ball estimator defined by \cite{JMLR:v21:18-764}. Then we have the following parameters used in Assumption~\ref{assum: est}.
    \begin{equation*}
        \rho_1 = \frac{\tilde{\rho}_T}{2}, \quad \rho_2 = 1 - \left(1 - \frac{1}{T + 8} \right)^{\frac{4}{3}}, \quad A = 1, \quad B = \frac{2L^2}{\tilde{\rho}_T}, \quad C = 0, \quad E = 0, \quad \sigma^2_t = \tilde{\rho_t}^2 \sigma^2.
    \end{equation*}
    where $\tilde{\rho}_t = 4/(t + 8)^{\frac{2}{3}}$ is the decay used in the estimator \cite{JMLR:v21:18-764} and $\sigma^2$ is the variance bound by Assumption~\ref{assump: heavy_ball constant bound}.
\end{lemma} 
\begin{proof}
    From Lemma~\ref{lem: heavy ball error recursion}, taking expectation on both sides, we have
    \begin{align*}
        \mathbb{E}[\|\Delta\|^2] \leq \left(1 - \frac{\tilde{\rho}_t}{2} \right)\mathbb{E}[\|\Delta^{t-1}\|^2] + \tilde{\rho_t}^2 \sigma^2 + \frac{2L^2D^2}{\tilde{\rho}_t}\eta_{t-1}^2.
    \end{align*}
    Since $\forall t \in \{0, 1 \cdots T \}, \tilde{\rho_t} \geq \tilde{\rho_T}$, we have,
    \begin{equation*}
        \left(1 - \frac{\tilde{\rho}_t}{2} \right) \leq \left(1 - \frac{\tilde{\rho}_T}{2} \right) \quad \& \quad \frac{2L^2D^2}{\tilde{\rho}_t} \leq \frac{2L^2D^2}{\tilde{\rho}_T}.
    \end{equation*}
    Also, setting $\sigma_t^2 = \tilde{\rho_t}^2 \sigma^2$, we thus have the equation which gives the parameters $\rho_1, A, B$:
    \begin{equation*}
        \mathbb{E}[\| \Delta^t \|^2] \leq \left(1 - \frac{\tilde{\rho}_T}{2} \right)\mathbb{E}[\| \Delta^{t-1}\|^2] + \sigma_t^2 + \frac{2L^2D^2}{\tilde{\rho}_T}\eta_{t-1}^2.
    \end{equation*}
    We now need to solve the recurrent inequality given below, to fit Assumption~\ref{assum: est}.
    \begin{equation*}
        \forall t > 0: \quad \sigma_t^2 \leq (1 - \rho_2) \sigma_{t-1}^2.
    \end{equation*}
    Using the expression of $\sigma_t^2$, we have
    \begin{align*}
        \tilde{\rho}_t^2 \sigma^2 &\leq (1 - \rho_2) \tilde{\rho}_{t-1}^2 \sigma^2, \\
        \tilde{\rho}_t^2 & \leq (1 - \rho_2) \tilde{\rho}_{t-1}^2 \\
        \frac{4^2}{(t + 8)^{4/3}} &\leq (1 - \rho_2)  \frac{4^2}{(t + 7)^{4/3}}, \\
        \left( \frac{t + 7}{t + 8} \right)^{4/3} &\leq 1 - \rho_2, \\
        \left(1 - \frac{1}{t + 8} \right)^{4/3}&\leq 1 - \rho_2, \\
        \forall t \in \{1, 2, \cdots T \}: \quad \rho_2 &\leq 1 - \left(1 - \frac{1}{t + 8} \right)^{4/3}.
    \end{align*}
    Thus to solve $\rho_2$, 
    \begin{equation*}
        \rho_2 = \displaystyle\argmin_{t \in \{1, 2, \cdots T \}} \quad 1 - \left(1 - \frac{1}{t + 8} \right)^{4/3} = 1 - \left(1 - \frac{1}{T + 8} \right)^{4/3}.
    \end{equation*}
    Hence giving us the required parameters
\end{proof}

\begin{lemma} 
    (Heavy Ball Error Bounds for Quasar-Convex Functions) 
    \label{lem: heavy ball error bounds quasar convex}
    \cite{JMLR:v21:18-764}
    Let $\{x^t\}_{t=0}^T$ be a sequence generated by Algorithm~\ref{alg:BSFW} using the Heavy Ball estimator \cite{JMLR:v21:18-764}, under the assumption that $f$ satisfies Assumption~\ref{assum: L_smooth} and $\rho$-quasar-convexity under Assumption~\ref{assum: quasar-convexity}. Let $Q$ be the constant defined as 
    $$Q = \max \left\{ 9^{2/3} \|\nabla f(x^0) - m^0\|^2, \; (16\sigma^2 + 2L^2 D^2)/\rho^2 \right\},$$ 
    where $\sigma^2$ is the constant upper bound of variance of unbiased stochastic gradients, as given by Assumption~\ref{assump: heavy_ball constant bound}. Suppose 
    $$\forall t: \quad \eta_t = \frac{2}{\rho(t + 9)}, \quad \tilde{\rho_t} = \frac{4}{(t+8)^\frac{2}{3}}.$$
    Then we have the following result:
    \begin{equation*}
        \forall t: \quad \mathbb{E}[\|\Delta^t\|^2] \leq \frac{Q}{(t + 9)^{\frac{2}{3}}}.
    \end{equation*}
\end{lemma}
\begin{proof}
    The proof follows from \cite{JMLR:v21:18-764}. From Lemma~\ref{lem: heavy ball error recursion} we have
    \begin{align*}
        \mathbb{E}_{t-1}[\|\Delta^t\|^2] &\leq \left(1 - \frac{\tilde{\rho_t}}{2} \right)\|\Delta^{t-1}\|^2 +  \tilde{\rho_t}^2 \sigma^2 + \frac{2\eta_{t-1}^2L^2D^2}{\tilde{\rho_t}} \\
        &\leq \left(1 - \frac{2}{(t+8)^\frac{2}{3}} \right)\|\Delta^{t-1}\|^2 + \frac{16\sigma^2 + 2L^2 D^2}{\rho^2 (t+8)^{\frac{4}{3}}}.
    \end{align*}
    Taking expectation on both sides, we have,
    \begin{align*}
        \mathbb{E}[\|\Delta^t\|^2] \leq \left(1 - \frac{2}{(t+8)^\frac{2}{3}} \right)\mathbb{E}[\|\Delta^{t-1}\|^2] + \frac{16\sigma^2 + 2L^2 D^2}{\rho^2 (t+8)^{\frac{4}{3}}}.
    \end{align*}
    Using the following parameters in Lemma 19 of \cite{JMLR:v21:18-764},
    \begin{align*}
        \phi_t = \mathbb{E}[\| \Delta^t \|^2], 
        \quad \alpha = \tfrac{2}{3}, 
        \quad b = (16\sigma^2 + 2L^2 D^2)/\rho^2, 
        \quad c = 2, 
        \quad t_0 = 8.
    \end{align*}
    We have the following result:
    \begin{align*}
        \mathbb{E}[\|\Delta^t\|^2] \leq \frac{Q}{(t + 9)^{\frac{2}{3}}},
    \end{align*}
    where $Q = \max \left\{ 9^{2/3} \|\nabla f(x^0) - m^0\|^2, \; (16\sigma^2 + 2L^2 D^2)/\rho^2 \right\}$.
\end{proof}

\begin{theorem} [Alternative Recursion for Quasar-Convex Functions]
    \label{thm: alt recursion quasar convex}
    Let $\{x^t\}_{t=0}^T$ be a sequence generated by Algorithm~\ref{alg:BSFW}, where the function $f$ satisfies Assumptions~\ref{assum: L_smooth} and~\ref{assum: quasar-convexity}. Then we have
    \begin{equation*}
        F_{t+1} \leq \left(1 - \rho \eta_t\right) F_t + 2\eta_t \| \Delta^t \|D + \frac{L}{2} \eta^2_t D^2.
    \end{equation*}    
\end{theorem}
\begin{proof}
From Assumption~\ref{assum: L_smooth}, we have
    \begin{equation*}
        f(x^{t+1}) \leq f(x^t) + \langle \nabla f(x^t), x^{t+1} - x^t \rangle + \frac{L}{2} \|x^{t+1} - x^t\|^2.
    \end{equation*}
    \textbf{Case I}: suppose $\gamma_t < 1$. Then we have $x^{t+1} = x^t + \gamma_t d^t$. By using Lemma~\ref{lem:alignment},
    {\footnotesize\begin{align*}
        f(x^{t+1}) &\leq f(x^t) + \gamma_t \langle \nabla f(x^t), d^t \rangle + \frac{L}{2} \gamma_t^2 \|d^t\|^2 \\
        &\leq f(x^t) + \gamma_t \langle \nabla f(x^t) - m^t, d^t \rangle + \gamma_t \langle m^t, d^t \rangle + \frac{L}{2} \gamma_t^2 \|d^t\|^2 \\
        &\leq f(x^t) + \gamma_t \langle \nabla f(x^t) - m^t, d^t \rangle + \eta_t \left( \frac{\|s^t - x^t\|}{\|d^t\|} \right) \left( \frac{\|d^t\|}{\|s^t - x^t\|} \right)  \langle m^t, s^t - x^t \rangle + \frac{L}{2} \eta_t^2 \frac{\|s^t - x^t\|^2}{\|d^t\|^2} \|d^t\|^2 \\ 
        &\leq f(x^t) + \gamma_t \langle \nabla f(x^t) - m^t, d^t \rangle + \eta_t \langle m^t, s^t - x^t \rangle + \frac{L}{2} \eta_t^2 \|s^t - x^t\|^2.
    \end{align*}}
    Since $s^t \in \mathrm{lmo}(m^t)$, we get
    {\footnotesize\begin{align*}
        f(x^{t+1}) &\leq f(x^t) + \gamma_t \langle \nabla f(x^t) - m^t, d^t \rangle + \eta_t \langle m^t, x^{\star} - x^t \rangle + \frac{L}{2} \eta_t^2 \|s^t - x^t\|^2 \\
        &\leq f(x^t) + \gamma_t \langle \nabla f(x^t) - m^t, d^t \rangle + \eta_t \langle m^t - \nabla f(x^t), x^{\star} - x^t \rangle + \eta_t \langle \nabla f(x^t), x^{\star} - x^t \rangle + \frac{L}{2} \eta_t^2 \|s^t - x^t\|^2.
    \end{align*}}
    By using Cauchy-Schwarz inequality on both $ \langle \nabla f(x^t) - m^t, d^t \rangle$ and $ \langle m^t - \nabla f(x^t), x^{\star} - x^t \rangle$, we get
    {\footnotesize\begin{align*}
        f(x^{t+1}) &\leq f(x^t) + \gamma_t \|m^t - \nabla f(x^t)\| \|d^t\| + \eta_t \|m^t - \nabla f(x^t) \| \|x^{\star} - x^t \| + \eta_t \langle \nabla f(x^t), x^{\star} - x^t \rangle + \frac{L}{2} \eta_t^2 \|s^t - x^t\|^2 \\
        &\leq f(x^t) + \eta_t \frac{\|s^t - x^t\|}{\|d^t\|} \|\Delta^t\| \|d^t\| + \eta_t \|\Delta^t\| D + \eta_t \langle \nabla f(x^t), x^{\star} - x^t \rangle + \frac{L}{2} \eta_t^2 \|s^t - x^t\|^2 \\
        &\leq f(x^t) + 2 \eta_t \| \Delta^t \|D + \eta_t \langle \nabla f(x^t), x^{\star} - x^t \rangle + \frac{L}{2} \eta_t^2 D^2.
    \end{align*}}
    \textbf{Case II}: suppose $\gamma_t = 1$. Then we have $x^{t+1} = x^{t} + \eta_t (s^t - x^t)$. Hence we get,
    {\footnotesize\begin{align*}
        f(x^{t+1}) &\leq f(x^t) + \eta_t \langle \nabla f(x^t), s^t - x^t \rangle + \frac{L}{2} \eta_t^2 \|s^t - x^t\|^2 \\
        &\leq f(x^t) + \eta_t \langle \nabla f(x^t) - m^t, s^t - x^t \rangle + \eta_t \langle m^t, s^t - x^t \rangle +  \frac{L}{2} \eta_t^2 \|s^t - x^t\|^2 \\
        &\leq f(x^t) + \eta_t \langle \nabla f(x^t) - m^t, s^t - x^t \rangle + \eta_t \langle m^t, x^{\star} - x^t \rangle +  \frac{L}{2} \eta_t^2 \|s^t - x^t\|^2 \\
        &\leq f(x^t) + \eta_t \langle \nabla f(x^t) - m^t, s^t - x^t \rangle + \eta_t \langle m^t - \nabla f(x^t), x^{\star} - x^t \rangle + \eta_t \langle \nabla f(x^t), x^{\star} - x^t \rangle + \frac{L}{2} \eta_t^2 \|s^t - x^t\|^2 \\
        &\leq f(x^t) + \eta_t \langle \nabla f(x^t) - m^t, s^t - x^{\star} \rangle + \eta_t \langle \nabla f(x^t), x^{\star} - x^t \rangle + \frac{L}{2} \eta_t^2 \|s^t - x^t\|^2
    .\end{align*}}
    Using Cauchy-Schwarz inequality on $ \langle \nabla f(x^t) - m^t, s^t - x^{\star} \rangle $, we have
    \begin{align*}
        f(x^{t+1}) &\leq f(x^t) + \eta_t \|\nabla f(x^t) - m^t \| \|s^t - x^{\star}\| + \eta_t \langle \nabla f(x^t), x^{\star} - x^t \rangle + \frac{L}{2} \eta_t^2 \|s^t - x^t\|^2 \\
        &\leq f(x^t) + \eta_t \|\Delta^t\|D +\eta_t \langle \nabla f(x^t), x^{\star} - x^t \rangle + \frac{L}{2} \eta_t^2 D^2 \\
        &\leq f(x^t) + 2 \eta_t \|\Delta^t\|D +\eta_t \langle \nabla f(x^t), x^{\star} - x^t \rangle + \frac{L}{2} \eta_t^2 D^2.
    \end{align*}
    Thus in both cases, we can reach the same expression. Now using $\rho$-quasar-convexity of $f$, 
    \begin{equation*}
        f(x^{t+1}) \leq f(x^t) + 2 \eta_t \|\Delta^t\|D - \rho \eta_t (f(x^t) - f(x^{\star})) + \frac{L}{2} \eta_t^2 D^2.
    \end{equation*}
    Subtracting $f(x^{\star})$ on both sides gives us
    \begin{equation*}
        F_{t+1} \leq (1 - \rho \eta_t ) F_t + 2 \eta_t \|\Delta^t\| D + \frac{L}{2} \eta_t^2 D^2.
    \end{equation*}
\end{proof}

\begin{theorem}
    [Convergence under Heavy Ball Estimator for Quasar-Convex Functions]
    \label{thm: convergence heavy ball quasar convex}
    Let $Q'$ be the constant defined by:
    \begin{equation*}
        Q' = \max \left\{ 9^{\frac{1}{3}}F_0, \; \frac{4D\sqrt{Q}}{\rho} + \frac{LD^2}{2\rho^2} \right\},
    \end{equation*}
    where $Q$ is the constant defined by Lemma~\ref{lem: heavy ball error bounds quasar convex}. Let $\{x^t\}_{t=0}^T$ be a sequence generated by Algorithm~\ref{alg:BSFW} using the Heavy Ball estimator \cite{JMLR:v21:18-764}, where the function $f$ satisfies Assumptions~\ref{assum: L_smooth} and~\ref{assum: quasar-convexity}. Suppose
    \begin{equation*}
        \forall t: \quad \eta_t = \frac{2}{\rho(t + 9)}, \quad \rho_t = \frac{4}{(t + 8)^\frac{2}{3}}.
    \end{equation*}
Then we have
\begin{equation*}
    \mathbb{E}[F_t] \leq \frac{Q'}{(t + 9)^{\frac{1}{3}}}.
\end{equation*}
Hence, this shows a convergence rate of $\mathcal{O}\left( \frac{1}{t^{1/3}} \right)$.
\end{theorem}
\begin{proof}
    The proof follows exactly as in \cite{JMLR:v21:18-764}. From Theorem~\ref{thm: alt recursion quasar convex}, we have
    \begin{align*}
        F_{t+1} &\leq \left(1 - \rho \eta_t\right) F_t + 2\eta_t \|\Delta^t\|D_2 + \frac{L}{2} \eta^2_t D_2^2.
    \end{align*}
    Taking expectation on both sides and using Lemma~\ref{lem: heavy ball error bounds quasar convex} we get,
    \begin{align*}
        \mathbb{E}[F_{t+1}] &\leq \left(1 - \rho \eta_t\right) \mathbb{E}[F_t] + 2\eta_t D_2 \mathbb{E}[\|\Delta^t\|] + \frac{L}{2} \eta^2_t D_2^2 \\
        &\leq \left(1 - \rho \eta_t\right) \mathbb{E}[F_t] + 2\eta_t D_2 \sqrt{\mathbb{E}[\|\Delta^t\|^2]} + \frac{L}{2} \eta^2_t D^2 \\
        &\leq \left(1 - \frac{2}{t+9}\right)\mathbb{E}[F_t] + \left( \frac{4}{\rho(t + 9)} \right)  \frac{D\sqrt{Q}}{(t+9)^{\frac{1}{3}}} + \frac{4LD^2}{2\rho^2(t+9)^2} \\
        &\leq \left(1  -\frac{2}{t+9} \right) \mathbb{E}[F_t] + \left(\frac{4D\sqrt{Q}}{\rho} + \frac{LD^2}{2\rho^2} \right)\frac{1}{(t+9)^{\frac{4}{3}}}.
    \end{align*}
Aim is to prove by PMI that
\begin{equation*}
    \forall t: \quad \mathbb{E}[F_t] \leq \frac{Q'}{(t+9)^{\frac{1}{3}}}.
\end{equation*}
\textbf{Base Case}: At $t=0$, by definition of $Q'$,
\begin{align*}
    9^{\frac{1}{3}} \mathbb{E}[F_0] \leq Q'\implies\mathbb{E}[F_0] \leq \frac{Q'}{(0 + 9)^{\frac{1}{3}}}.
\end{align*}
\textbf{Induction Step}: Suppose $\exists r$ such that
\begin{align*}
    \mathbb{E}[F_r] \leq \frac{Q'}{(r + 9)^{\frac{1}{3}}}.
\end{align*}
Hence,
\begin{align*}
    \mathbb{E}[F_{r + 1}] &\leq \left(1  -\frac{2}{r+9} \right) \mathbb{E}[F_r] + \left(\frac{4D\sqrt{Q}}{\rho} + \frac{LD^2}{2\rho^2} \right)\frac{1}{(r+9)^{\frac{4}{3}}} \\
    &\leq \left(1  -\frac{2}{r+9} \right)\left( \frac{Q'}{(r + 9)^{\frac{1}{3}}} \right) + \left(\frac{4D\sqrt{Q}}{\rho} + \frac{LD^2}{2\rho^2} \right)\frac{1}{(r+9)^{\frac{4}{3}}} \\
    &\leq Q' \left( \frac{r + 8}{(r + 9)^{\frac{2}{3}}} \right) \\
    &\leq \frac{Q'}{(r + 10)^{\frac{1}{3}}}.
\end{align*}
Hence proven by PMI.
\end{proof}

\begin{lemma} 
    (Heavy Ball Error Bounds for Nonconvex Functions) 
    \label{lem: heavy ball error bounds non convex}
    \cite{JMLR:v21:18-764}
    Let $\{x^t\}_{t=0}^T$ be a sequence generated by Algorithm~\ref{alg:BSFW} using the Heavy Ball estimator \cite{JMLR:v21:18-764}, under the assumption that $f$ satisfies Assumption~\ref{assum: L_smooth}. Let $M_h$ be the constant defined as 
    $$M_h = \max \left\{ \|\Delta^0\|^2, ~  24(\sigma^2 + 2L^2D^2)\right\},$$ 
    where $\sigma^2$ is the constant upper bound of variance of unbiased stochastic gradients, as given by Assumption~\ref{assump: heavy_ball constant bound}. Suppose $$\forall t: \quad \eta_t = \frac{1}{(t + 2)^{\frac{3}{4}}}, \quad \tilde{\rho_t} = \frac{1}{\sqrt{t + 1}}.$$
    Then we have the following result:
    \begin{equation*}
        \forall t: \quad \mathbb{E}[\|\Delta^t\|^2] \leq \frac{M_h}{\sqrt{t + 1}}.
    \end{equation*}
\end{lemma}
\begin{proof}
    Taking expectation on both sides of Lemma~\ref{lem: heavy ball error recursion}, we have
    \begin{equation*}
        \mathbb{E}\left[ \|\Delta^t\|^2\right] \leq \left(1 - \frac{\tilde{\rho_t}}{2} \right) \mathbb{E}\left[\|\Delta^{t-1}\|^2\right] + \tilde{\rho_t}^2 \sigma^2 + \frac{2L^2D^2\eta_{t-1}^2}{\tilde{\rho_t}}.
    \end{equation*}
    Substituting the values of $\tilde{\rho_t}$ and $\eta_{t-1}$ we get,
    \begin{align*}
         \mathbb{E}\left[ \|\Delta^t\|^2\right] &\leq \left(1 - \frac{1}{2\sqrt{t + 1}} \right) \mathbb{E}\left[\|\Delta^{t-1}\|^2\right] + \frac{\sigma^2}{(t + 1)} + \frac{2L^2D^2}{(t + 1)} \\
         &\leq \left(1 - \frac{1}{2\sqrt{t + 1}} \right) \mathbb{E}\left[\|\Delta^{t-1}\|^2\right] + \frac{\sigma^2 + 2L^2D^2}{(t + 1)}.
    \end{align*}
    The proof follows from Lemma D.10 in \cite{pethick2025training}. Using PMI, \\ \\
    \textbf{Base Case}: At $t = 0$, by definition of $M_h$,
    \begin{equation*}
        \|\Delta^0\|^2 \leq M_h.
    \end{equation*}
    \textbf{Induction Step}: Let $t > 0$ such that 
    \begin{equation*}
        \mathbb{E}[\|\Delta^{t-1}\|^2] \leq \frac{M_h}{\sqrt{t}}.
    \end{equation*}
    We thus have
    \begin{align*}
        \mathbb{E}[\|\Delta^t\|^2] &\leq \left( 1 - \frac{1}{2\sqrt{t + 1}}\right) \mathbb{E}[\|\Delta^{t-1}\|^2] + \frac{\sigma^2 + 2L^2D^2}{(t + 1)} \\
        &\leq \left( 1 - \frac{1}{2\sqrt{t + 1}}\right) \frac{M_h}{\sqrt{t}} + \frac{\sigma^2 + 2L^2D^2}{(t + 1)}.
    \end{align*}
    We can write
    \begin{equation*}
        \frac{1}{\sqrt{t}} = \frac{1}{\sqrt{t + 1}}\sqrt{\frac{t + 1}{t}} = \frac{1}{\sqrt{t + 1}}\sqrt{1 + \frac{1}{t}} \leq \frac{1}{\sqrt{t + 1}}\left( 1 + \frac{1}{2t} \right)
    .\end{equation*}
    Also, we have the fact that $\forall t \geq 1$, 
    \begin{equation*}
        \left(1 - \frac{1}{2\sqrt{t + 1}} \right)\left(1 + \frac{1}{2t} \right) \leq \left(1 - \frac{1}{24\sqrt{t + 1}} \right).
    \end{equation*}
    Hence we have
    \begin{align*}
        \mathbb{E}[\|\Delta^t\|^2] &\leq \left( 1 - \frac{1}{2\sqrt{t + 1}}\right)\frac{M_h}{\sqrt{t}} + \frac{\sigma^2 + 2L^2D^2}{(t + 1)} \\
        &\leq \left( 1 - \frac{1}{2\sqrt{t + 1}}\right)\left(1 + \frac{1}{2t} \right) \frac{M_h}{\sqrt{t + 1}} + \frac{\sigma^2 + 2L^2D^2}{(t + 1)} \\
        &\leq \left( 1 - \frac{1}{24\sqrt{t + 1}}\right) \frac{M_h}{\sqrt{t + 1}} + \frac{\sigma^2 + 2L^2D^2}{(t + 1)} \\
        &\leq \frac{M_h}{\sqrt{t + 1}} - \frac{M_h}{24(t + 1)} + \frac{\sigma^2 + 2L^2D^2}{(t + 1)} \\
        &\leq \frac{M_h}{\sqrt{t + 1}} + \frac{1}{(t + 1)}\left(\sigma^2 + 2L^2D^2 - \frac{M_h}{24} \right) \\
        &\leq \frac{M_h}{\sqrt{t + 1}}.
    \end{align*}
\end{proof}

\begin{theorem}
    \label{thm: convergence heavy ball non convex}
    (Convergence under Heavy Ball Estimator for Nonconvex Functions). Let $\{x^t\}_{t=0}^{T}$ be a sequence generated by Algorithm~\ref{alg:BSFW} with the Heavy Ball estimator \cite{JMLR:v21:18-764}, where the objective function $f$ satisfies Assumption~\ref{assum: L_smooth}. Suppose for all $t$, $\eta_t = \frac{1}{(t + 2)^{3/4}}$ and $\rho_t = \frac{1}{\sqrt{t + 1}}$. We thus have the following bound,
    \begin{equation*}
    \mathbb{E}\left[ \min_{0 \leq t \leq T - 1} \mathrm{Gap}(x^t) \right] \leq \frac{F_0 + 2D\sqrt{M_h} (1 + \ln(T))  +  LD^2 }{4\left(\left(T+2\right)^{\frac{1}{4}}-2^{\frac{1}{4}}\right)}.
\end{equation*}
where $M_h$ is the constant defined by Lemma~\ref{lem: heavy ball error bounds non convex}, Hence, this shows a convergence rate of $\mathcal{O}\left(\frac{\ln t}{t^{1/4}}\right)$.
\end{theorem}
\begin{proof}
   From Assumption~\ref{assum: L_smooth}, we have
    \begin{equation*}
        f(x^{t+1}) \leq f(x^t) + \langle \nabla f(x^t), x^{t+1} - x^t \rangle + \frac{L}{2} \|x^{t+1} - x^t\|^2.
    \end{equation*}
    \textbf{Case I}: suppose $\gamma_t < 1$. Then we have $x^{t+1} = x^t + \gamma_t d^t$. By using Lemma~\ref{lem:alignment},
    {\footnotesize\begin{align*}
        f(x^{t+1}) &\leq f(x^t) + \gamma_t \langle \nabla f(x^t), d^t \rangle + \frac{L}{2} \gamma_t^2 \|d^t\|^2 \\
        &\leq f(x^t) + \gamma_t \langle \nabla f(x^t) - m^t, d^t \rangle + \gamma_t \langle m^t, d^t \rangle + \frac{L}{2} \gamma_t^2 \|d^t\|^2 \\
        &\leq f(x^t) + \gamma_t \langle \nabla f(x^t) - m^t, d^t \rangle + \eta_t \left( \frac{\|s^t - x^t\|}{\|d^t\|} \right) \left( \frac{\|d^t\|}{\|s^t - x^t\|} \right)  \langle m^t, s^t - x^t \rangle + \frac{L}{2} \eta_t^2 \frac{\|s^t - x^t\|^2}{\|d^t\|^2} \|d^t\|^2 \\ 
        &\leq f(x^t) + \gamma_t \langle \nabla f(x^t) - m^t, d^t \rangle + \eta_t \langle m^t, s^t - x^t \rangle + \frac{L}{2} \eta_t^2 \|s^t - x^t\|^2.
    \end{align*}}
    Since $s^t \in \mathrm{lmo}(m^t)$, for all $u \in \mathcal{C}$, we get
    {\footnotesize\begin{align*}
        f(x^{t+1}) &\leq f(x^t) + \gamma_t \langle \nabla f(x^t) - m^t, d^t \rangle + \eta_t \langle m^t, u - x^t \rangle + \frac{L}{2} \eta_t^2 \|s^t - x^t\|^2 \\
        &\leq f(x^t) + \gamma_t \langle \nabla f(x^t) - m^t, d^t \rangle + \eta_t \langle m^t - \nabla f(x^t), u - x^t \rangle + \eta_t \langle \nabla f(x^t), u - x^t \rangle + \frac{L}{2} \eta_t^2 \|s^t - x^t\|^2.
    \end{align*}}
    By using Cauchy-Schwarz inequality on the both $ \langle \nabla f(x^t) - m^t, d^t \rangle$ and $ \langle m^t - \nabla f(x^t), u - x^t \rangle$, we get
    {\footnotesize\begin{align*}
        f(x^{t+1}) &\leq f(x^t) + \gamma_t \|m^t - \nabla f(x^t)\| \|d^t\| + \eta_t \|m^t - \nabla f(x^t) \| \|u - x^t \| + \eta_t \langle \nabla f(x^t), u - x^t \rangle + \frac{L}{2} \eta_t^2 \|s^t - x^t\|^2 \\
        &\leq f(x^t) + \eta_t \frac{\|s^t - x^t\|}{\|d^t\|} \|\Delta^t\| \|d^t\| + \eta_t \|\Delta^t\| D + \eta_t \langle \nabla f(x^t), u - x^t \rangle + \frac{L}{2} \eta_t^2 \|s^t - x^t\|^2 \\
        &\leq f(x^t) + 2 \eta_t \| \Delta^t \|D + \eta_t \langle \nabla f(x^t), u - x^t \rangle + \frac{L}{2} \eta_t^2 D^2.
    \end{align*}}
    \textbf{Case II}: suppose $\gamma_t = 1$. Then we have $x^{t+1} = x^{t} + \eta_t (s^t - x^t)$. Hence we get for all $u \in \mathcal{C}$,
    {\footnotesize\begin{align*}
        f(x^{t+1}) &\leq f(x^t) + \eta_t \langle \nabla f(x^t), s^t - x^t \rangle + \frac{L}{2} \eta_t^2 \|s^t - x^t\|^2 \\
        &\leq f(x^t) + \eta_t \langle \nabla f(x^t) - m^t, s^t - x^t \rangle + \eta_t \langle m^t, s^t - x^t \rangle +  \frac{L}{2} \eta_t^2 \|s^t - x^t\|^2 \\
        &\leq f(x^t) + \eta_t \langle \nabla f(x^t) - m^t, s^t - x^t \rangle + \eta_t \langle m^t, u - x^t \rangle +  \frac{L}{2} \eta_t^2 \|s^t - x^t\|^2 \\
        &\leq f(x^t) + \eta_t \langle \nabla f(x^t) - m^t, s^t - x^t \rangle + \eta_t \langle m^t - \nabla f(x^t), u - x^t \rangle + \eta_t \langle \nabla f(x^t), u - x^t \rangle + \frac{L}{2} \eta_t^2 \|s^t - x^t\|^2 \\
        &\leq f(x^t) + \eta_t \langle \nabla f(x^t) - m^t, s^t - u \rangle + \eta_t \langle \nabla f(x^t), u - x^t \rangle + \frac{L}{2} \eta_t^2 \|s^t - x^t\|^2.
    \end{align*}}
    Using Cauchy-Schwarz inequality on $ \langle \nabla f(x^t) - m^t, s^t - u \rangle $, we have
    \begin{align*}
        f(x^{t+1}) &\leq f(x^t) + \eta_t \|\nabla f(x^t) - m^t \| \|s^t - u\| + \eta_t \langle \nabla f(x^t), u - x^t \rangle + \frac{L}{2} \eta_t^2 \|s^t - x^t\|^2 \\
        &\leq f(x^t) + \eta_t \|\Delta^t\|D +\eta_t \langle \nabla f(x^t), u - x^t \rangle + \frac{L}{2} \eta_t^2 D^2 \\
        &\leq f(x^t) + 2 \eta_t \|\Delta^t\|D +\eta_t \langle \nabla f(x^t), u - x^t \rangle + \frac{L}{2} \eta_t^2 D^2.
    \end{align*}
    Thus in both cases, we can reach the same expression. Subtracting $f(x^{\star})$ on both sides gives us
    \begin{align*}
        F_{t+1} &\leq F_t + 2 \eta_t \|\Delta^t\| D +\eta_t \langle \nabla f(x^t), u - x^t \rangle + \frac{L}{2} \eta_t^2 D^2 ,\\
        \eta_t \langle \nabla f(x^t),  x^t - u \rangle &\leq F_t - F_{t+1} + 2 \eta_t \|\Delta^t\| D + \frac{L}{2} \eta_t^2 D^2,
    \end{align*}
Taking expectation on both sides, and summing from $t=0$ to $t = T-1$,
\begin{align*}
    &\mathbb{E}\left[ \eta_t \langle \nabla f(x^t), x^t - u  \rangle \right] \leq \mathbb{E}[F_t] - \mathbb{E}[F_{t+1}] + 2 \eta_t \mathbb{E}[\|\Delta^t\|]D + \eta_t^2 \frac{LD^2}{2}, \\
    \implies&\sum_{t=0}^{T-1} \mathbb{E}\left[ \eta_t \langle \nabla f(x^t), x^t - u \rangle \right] \leq F_0 - \mathbb{E}[F_T] + \sum_{t=0}^{T-1} \eta_t 2 \mathbb{E}[\|\Delta^t\|] D +  \sum_{t=0}^{T-1} \eta_t^2 \frac{LD^2}{2}, \\
    \implies&\sum_{t=0}^{T-1} \mathbb{E}\left[ \eta_t \langle \nabla f(x^t), x^t - u \rangle \right] \leq F_0 + \sum_{t=0}^{T-1} 2 \eta_t \mathbb{E}[\|\Delta^t\|] D +  \sum_{t=0}^{T-1} \eta_t^2 \frac{LD^2}{2}, \\
    \implies&\mathbb{E}\left[ \min_{0 \leq t \leq T - 1} \langle \nabla f(x^t), x^t - u \rangle \right]\sum_{t=0}^{T-1} \eta_t  \leq F_0 + \sum_{t=0}^{T-1} \eta_t 2 \mathbb{E}[\|\Delta^t\|] D +  \sum_{t=0}^{T-1} \eta_t^2 \frac{LD^2}{2}, \\
    \implies&\mathbb{E}\left[ \min_{0 \leq t \leq T - 1} \langle \nabla f(x^t), x^t - u \rangle \right] \leq \frac{F_0 + D \sum_{t=0}^{T-1} 2 \eta_t \mathbb{E}[\|\Delta^t\|]  +  \frac{LD^2}{2} \sum_{t=0}^{T-1} \eta_t^2 }{\sum_{t=0}^{T-1} \eta_t}, \\
    &~~~~~~~~~~~~~~~~~~~~~~~~~~~~~~~~~~~~~~~~~~~~~~~~~~~~\leq \frac{F_0 + D \sum_{t=0}^{T-1} 2 \eta_t \sqrt{\mathbb{E}[\|\Delta^t\|^2]}  +  \frac{LD^2}{2} \sum_{t=0}^{T-1} \eta_t^2 }{\sum_{t=0}^{T-1} \eta_t}.
\end{align*}
From Lemma~\ref{lem: heavy ball error bounds non convex}, we know that for all $t$, $\mathbb{E}[\|\Delta^t\|^2] \leq \frac{M_h}{\sqrt{t + 1}}$. Thus,
\begin{align*}
    \mathbb{E}\left[ \min_{0 \leq t \leq T - 1} \langle \nabla f(x^t), x^t - u \rangle \right] &\leq \frac{F_0 + 2D\sqrt{M_h} \sum_{t=0}^{T-1} \frac{1}{(t + 2)^{3/4}} \frac{1}{(t + 1)^{1/4}}  +  \frac{LD^2}{2} \sum_{t=0}^{T-1} \frac{1}{(t + 2)^{3/2}} }{\sum_{t=0}^{T-1} \frac{1}{(t + 2)^{3/4}}} \\
    &\leq \frac{F_0 + 2D\sqrt{M_h} \sum_{t=0}^{T-1} \frac{1}{(t + 1)}  +  \frac{LD^2}{2} \sum_{t=0}^{T-1} \frac{1}{(t + 2)^{3/2}} }{\sum_{t=0}^{T-1} \frac{1}{(t + 2)^{3/4}}}.
\end{align*}
By using the integral test, we have the following bounds
\begin{equation*}
    \sum_{t=0}^{T-1} \frac{1}{(t + 1)} \leq 1 + \ln(T), \quad \sum_{t=0}^{T-1} \frac{1}{(t + 2)^{3/2}} \leq 2, \quad \sum_{t = 0}^{T-1} \frac{1}{(t + 2)^{3/4}} \geq 4\left(\left(T+2\right)^{\frac{1}{4}}-2^{\frac{1}{4}}\right).
\end{equation*}
Hence, we get the expression, that $\forall u \in \mathcal{C}$, 
\begin{equation*}
    \mathbb{E}\left[ \min_{0 \leq t \leq T - 1} \langle \nabla f(x^t), x^t - u \rangle \right] \leq \frac{F_0 + 2D\sqrt{M_h} (1 + \ln(T))  +  LD^2 }{4\left(\left(T+2\right)^{\frac{1}{4}}-2^{\frac{1}{4}}\right)}.
\end{equation*}
Which implies that since $\mathrm{Gap}(x^t) = \max_{u \in \mathcal{C}} \langle \nabla f(x^t),x^t - u \rangle$, we can write
\begin{equation*}
    \mathbb{E}\left[ \min_{0 \leq t \leq T - 1} \mathrm{Gap}(x^t) \right] \leq \frac{F_0 + 2D\sqrt{M_h} (1 + \ln(T))  +  LD^2 }{4\left(\left(T+2\right)^{\frac{1}{4}}-2^{\frac{1}{4}}\right)}.
\end{equation*}
\end{proof}

\subsubsection{SAG \cite{pmlr-v119-negiar20a}}
\label{subsec: SAG}
\paragraph{Description }For the usage of the SAG estimator \cite{pmlr-v119-negiar20a}, we need to assume that the objective function $f$ in \eqref{eq:P} is $f:x \mapsto \tilde{f}(\tilde{A}x)$, where $\tilde{A}$ is a matrix $m \times n$ of samples. This is to ensure a finite sum structure as required by \cite{pmlr-v119-negiar20a}.  The estimator uses an additional dual variable $\alpha^t$, which is initialized by setting $\alpha^0 \in \mathbb{R}^m$. For every iteration $t > 0$, we sample a batch $S_t \subset \{1, 2, \cdots, m\}$ of size $b_s$ uniformly at random. $b_s$ is a pre-defined per-iteration sample size parameter. For every $i \in \{1, 2, \cdots, m \}$, we update the $i^{\mathrm{th}}$ index of $\alpha^t$ by 
\begin{equation*}
    \alpha^{(i),t} = \begin{cases}
        \frac{1}{m}  \tilde{f}'_i(\langle \tilde{a}_i, x^t \rangle), & i \in S_t \\
        \alpha^{(i),t - 1}. & i \notin S_t 
        \end{cases} \\
\end{equation*}
For this algorithm to work, we require that $\forall_{x,y \in \mathcal{C}} \|\tilde{A}(x - y) \| \leq D$, and that the step size $\gamma_t$ be
\begin{equation}
    \label{eq: SAG gamma}
    \gamma_{t} = \min \left\{ \eta_t \frac{\|\tilde{A}(s^{t} - x^t)\|}{\|\tilde{A}d^{t}\|}, 1 \right\},
\end{equation}
where $s^t = \lmo(\tilde{A}^{\top} \alpha^t)$ and $d^t$ is the output of the boosting procedure by passing in $m^t = \tilde{A}^{\top} \alpha^t$. This implies that the Algorithm update is given by $\tilde{A}x^{t+1} = \tilde{A}x^{t} + \gamma_t \tilde{A}d^t$ when $\gamma_t < 1$ and $\tilde{A}x^{t+1} = \tilde{A}x^{t} + \eta_t \tilde{A}(s^t - x^t)$ when $\gamma_t = 1$. 
Since $As^t \in \lmo(\alpha^t)$, it serves to set $y^{t} = Ax^t$. As explained in Appendix section~\ref{appendix: boosted_stoch_fw}, we use the parameters provided by Assumption~\ref{assum: est} to prove convergence. To consider this objective function $\tilde{f}$, in this case the noise bound needs to be specified as
\begin{equation}
    \label{eq: SAG noise}
    \Delta^t = \alpha^t - \nabla \tilde{f}(\tilde{A}x^t).
\end{equation}

Using \eqref{eq: SAG noise}, the parameters satisfying Assumption~\ref{assum: est} are given by Lemma~\ref{params:SAG}.

\paragraph{Verification of Assumption~\ref{assum: est}}
\begin{lemma}
    \label{params:SAG}
    (Parameters for SAG) Assume that the objective function $f$ in \eqref{eq:P} can be written as $f:x \mapsto \tilde{f}(\tilde{A}x)$, where $\tilde{A}$ is a matrix $m \times n$ of samples. Let $\{x^t\}_{t=0}^T$ be a sequence generated by Algorithm~\ref{alg:BSFW} using a step decay $\{\eta_t\}_{t=0}^{T-1}$, where the gradient estimator $\{\Phi_t\}_{t=0}^{T-1}$ is SAG defined by \cite{pmlr-v119-negiar20a}. Then we have the following parameters used in Assumption~\ref{assum: est}.
    \begin{equation*}
        \rho_1 = \frac{b_s}{2m}, \quad \rho_2 = 1, \quad A = 0, \quad B = \left(1 - \frac{b_s}{m} \right)\left(1 + \frac{2m}{b_s}\right)L^2, \quad C = 0, \quad E = 0, \quad\sigma^2_t = 0,
    \end{equation*}
    where $b_s$ is the stochastic batch size sampled per-iteration.
\end{lemma} 
\begin{proof}
    For any $t$ such that $0 \leq t \leq T-1$, by definition of $\Delta^t$ in \eqref{eq: SAG noise},
    \begin{equation*}
        \|\Delta^t\| = \|\alpha^t - \nabla \tilde{f}(\tilde{A}x^t) \|.
    \end{equation*}
   At each iteration $t$, since there is a probability $\frac{b_s}{m}$ of any index $j \in \{1, 2 \cdots m\}$ being sampled, meaning $\alpha^{(j),t} = \alpha^{(j),t-1}$ with a probability of $\left(1 - \frac{b_s}{m} \right)$.
    We thus have the following conditional expectation equation
    \begin{equation*}
        \mathbb{E}_{t-1}[(\Delta^{(j),t})^2] = \left(1 - \frac{b_s}{m} \right) \left(\alpha^{(j), t-1} - \nabla \tilde{f}^{(j)} (\tilde{A}x^t)\right)^2,
    \end{equation*}
    where $\tilde{a_j}$ refers to the row $j$ of $\tilde{A}$. Summing over all indices from $j = 1$ to $j = m$, we have,
    \begin{align*}
        \mathbb{E}_{t-1}[\|\Delta^t\|^2] &= \sum_{j=1}^m \mathbb{E}_{t-1}[(\Delta^{(j),t})^2] \\
        &= \left(1 - \frac{b_s}{m} \right) \| \alpha^{t-1} - \nabla \tilde{f}(\tilde{A}x^t) \|^2 \\
        &= \left(1 - \frac{b_s}{m} \right) \| \alpha^{t-1} - \nabla \tilde{f}(\tilde{A}x^{t-1}) + \nabla \tilde{f}(\tilde{A}x^{t-1}) - \nabla \tilde{f}(\tilde{A}x^t) \|^2 \\
        &= \left(1 - \frac{b_s}{m} \right) ( \| \alpha^{t-1} - \nabla \tilde{f}(\tilde{A}x^{t-1}) \|^2 +  \| \nabla \tilde{f}(\tilde{A}x^{t-1}) - \nabla \tilde{f}(\tilde{A}x^t) \|^2  \\ 
        &+ 2 \langle \alpha^{t-1} - \nabla \tilde{f}(\tilde{A}x^{t-1}), \nabla \tilde{f}(\tilde{A}x^{t-1}) - \nabla \tilde{f}(\tilde{A}x^t) \rangle.
    \end{align*}
    
   For the term $\| \nabla \tilde{f}(\tilde{A}x^{t-1}) - \nabla \tilde{f}(\tilde{A}x^t) \|$, By using the L-smoothness property of $\tilde{f}$ and due to the definition of $\gamma_t$ in \eqref{eq: SAG gamma} in Lemma~\ref{lem: BSFW_iter_bound}, we have the following,
    \begin{equation*}
        \| \nabla \tilde{f}(\tilde{A}x^{t-1}) - \nabla \tilde{f}(\tilde{A}x^t) \| \leq L \|Ax^{t-1} - Ax^t \| \leq LD \eta_{t-1}.
    \end{equation*}
    Using young's inequality with a parameter $\beta > 0$, we have
    {\small\begin{align*}
        2 \langle \alpha^{t-1} - \nabla \tilde{f}(\tilde{A}x^{t-1}), \nabla \tilde{f}(\tilde{A}x^{t-1}) - \nabla \tilde{f}(\tilde{A}x^{t}) \rangle &\leq \beta \|\alpha^{t-1} - \nabla \tilde{f}(\tilde{A}x^{t-1})\|^2+ \frac{1}{\beta} \| \nabla \tilde{f}(\tilde{A}x^{t-1}) - \nabla \tilde{f}(\tilde{A}x^{t}) \|^2 \\
        &\leq \beta \| \Delta^{t-1} \|^2 + \frac{1}{\beta} \| \nabla \tilde{f}(\tilde{A}x^{t-1}) - \nabla \tilde{f}(\tilde{A}x^{t}) \|^2 \\
        &\leq \beta \| \Delta^{t-1} \|^2 + \frac{1}{\beta} L^2 D^2 \eta_{t-1}^2.
    \end{align*}}
    Using these inequalities in the expression for $\mathbb{E}_{t-1}[\|\Delta^t\|^2]$, we have the following,
    \begin{align*}
        \mathbb{E}_{t-1}[\|\Delta^t\|^2] &\leq \left(1 - \frac{b_s}{m} \right) \left(\| \Delta^{t-1} \|^2 + L^2D^2\eta_{t-1}^2 +  \beta \| \Delta^{t-1} \|^2 + \frac{1}{\beta} L^2 D^2 \eta_{t-1}^2\right) \\
        &\leq \left(1 - \frac{b_s}{m} \right)\left(1 + \beta\right)\| \Delta^{t-1} \|^2 + \left(1 - \frac{b_s}{m} \right) \left(1 + \frac{1}{\beta} \right)L^2D^2 \eta_{t-1}^2.
    \end{align*}
    By setting $\beta = \frac{b_s}{2m}$, we have the following
    \begin{align*}
        \left(1 - \frac{b_s}{m} \right)\left(1 + \beta\right) &= \left(1 - \frac{b_s}{m} \right)\left(1 + \frac{b_s}{2m}\right) \\
        &= 1 - \frac{b_s}{m} + \frac{b_s}{2m} - \frac{b_s}{2m^2} \\
        &= 1 - \frac{b_s}{2m} - \frac{b_s}{2m^2} \\
        &\leq \left(1 - \frac{b_s}{2m} \right).
    \end{align*}
    Hence, we have
    \begin{align*}
        \mathbb{E}_{t-1}[\| \Delta^t \|^2] \leq \left(1 - \frac{b_s}{2m} \right)\| \Delta^{t-1} \|^2 + \left(1 - \frac{b_s}{m} \right)\left(1 + \frac{2m}{b_s}\right)L^2D^2 \eta_{t-1}^2.
    \end{align*}
    Taking expectations on both sides gives us
    \begin{equation*}
        \mathbb{E}[\| \Delta^t \|^2] \leq \left(1 - \frac{b_s}{2m} \right) \mathbb{E}[\|\Delta^{t-1} \|^2] + \left(1 - \frac{b_s}{m} \right)\left(1 + \frac{2m}{b_s}\right)L^2D^2 \eta_{t-1}^2.
    \end{equation*}
    Hence fitting the necessary parameters in Assumption~\ref{assum: est}.
\end{proof}


\bibliography{Ref.bib}
\bibliographystyle{abbrv} 
\end{document}